\documentclass{article}
\usepackage[utf8]{inputenc}
\usepackage{amsmath,amsthm,amsfonts,amssymb,bm, mathdots}
\usepackage{authblk}
\usepackage{tikz}
\usepackage{ytableau}
\usepackage{mathtools}

\usepackage{caption}
\usepackage{subcaption}

\usepackage{soul}

\theoremstyle{plain}
\newtheorem{thm}{Theorem}[section]
\newtheorem*{thm*}{Theorem}
\newtheorem{prop}[thm]{Proposition}
\newtheorem*{prop*}{Proposition}

\newtheorem*{conj*}{Conjecture}

\newtheorem{lem}[thm]{Lemma}
\newtheorem*{lem*}{Lemma}
\newtheorem{example}[thm]{Example}
\newtheorem{definition}[thm]{Definition}
\newtheorem{remark}[thm]{Remark}

\def\diag{ \begin{tikzpicture} \draw[dashed] (-.12,-.12) -- (.42, .42); \end{tikzpicture} }

\newcommand{\checkerboard}[1]
{
\foreach \i in {0,...,#1}
\foreach \j in {0,...,\i}
{
\draw[lightgray, fill={\ifodd\numexpr\i+\j\relax white\else lightgray\fi}] (\i,\j) rectangle (\i+1,\j+1);
\draw[lightgray, fill={\ifodd\numexpr\i+\j\relax lightgray\else white\fi}] (\i,-\j) rectangle (\i+1,-\j-1);
\draw[lightgray, fill={\ifodd\numexpr\i+\j\relax lightgray\else white\fi}] (2*#1-\i+2,\j) rectangle (2*#1-\i+1,\j+1);
\draw[lightgray, fill={\ifodd\numexpr\i+\j\relax white\else lightgray\fi}] (2*#1-\i+2,-\j) rectangle (2*#1-\i+1,-\j-1);
}
}

\newcommand{\ZZ}{\mathbb{Z}}

\title{Shuffling algorithm for coupled tilings of the Aztec diamond}
\author{David Keating}
\affil{Department of Mathematics, University of Wisconsin, Madison\\
{\small\ttfamily{dkeating3@wisc.edu}}}
\author{Matthew Nicoletti}
\affil{Department of Mathematics, Massachusetts Institute of Technology \\
{\small\ttfamily{mnicolet@mit.edu}}}
 \date{}

\begin{document}

\maketitle

\begin{abstract}
    In this article we define a generalization of the domino shuffling algorithm for tilings of the Aztec diamond to the interacting $k$-tilings recently introduced by S. Corteel, A. Gitlin, and the first author. We describe the algorithm both in terms of dynamics on a system of colored particles and as operations on the dominos themselves.
\end{abstract}

\section{Introduction}

Domino tilings of the Aztec diamond were first introduced by Elkies, Kuperberg, Larsen, and Propp \cite{aztec0} in their study of alternating-sign matrices. See Figure \ref{fig:ADex} for an example of a domino tiling of rank 3. In their work, the authors introduced \emph{domino shuffling}, an algorithm by which one can generate a tiling of the Aztec diamond of rank $(N+1)$ from a tiling of the Aztec diamond of rank $N$. Using this they were able to derive a recursive formula for the number of tilings. Solving this recursion they found the beautiful result:
\begin{thm}
The number of tilings of the Aztec diamond of rank $N$ is given by
\[
2^{\binom{N+1}{2}}.
\]
\end{thm}
The shuffling algorithm has since proved very useful in the study of these tilings. An immediate benefit is that shuffling allows for efficient exact sampling with arbitrary weights~\cite{shuffling1}. Furthermore, it has also been used as a tool for asymptotic analysis, in the following way: One central result in the study of tilings of the Aztec diamond
is the arctic circle theorem \cite{aztec3}. It states that for large $N$, a uniformly random tiling exhibits a brickwork pattern in four regions (called frozen regions or polar regions), one adjacent to each corner of the Aztec diamond, whose union is approximately the region outside of the largest circle (called the arctic circle) that can be inscribed in the Aztec diamond. The strategy used in the original proof of this fact was a careful analysis of the shuffling algorithm \cite{aztec3}.

There are many ways to view domino shuffling. Originally it was described using sequences of moves that one must perform with the dominos on the tilings themselves, see \cite{shuffling2,shuffling1}. Alternatively, one can see it as an example of dynamics on a certain space of particle configurations, as first observed by Nordenstam \cite{nord}. With a restricted class of weights, these particle dynamics can be re-derived using the algebraic structure of the Schur process \cite{RYG,steep,Ferrari2008,borodin2015random}. Furthermore, when viewed as a deterministic discrete time dynamical system on the weights, domino shuffling is an example of a cluster integrable system \cite{GoncharovKenyon2011DimersClusterIntSys}. This is a consequence of the fact that the shuffling algorithm consists of a collection of structure preserving local moves called~\emph{spider moves}.

In this article, we describe a generalization of the shuffling algorithm. Recently, in \cite{LLTaztec} the authors described a model of $k$ interacting tilings of the Aztec diamond. The authors computed the partition function of the model by relating it the LLT polynomials of Lascoux, Leclerc, and Thibon \cite{LLT1997}. We generalize the shuffling algorithm to these interacting $k$-tilings.

More precisely, a \emph{$k$-tiling} of the Aztec diamond is a collection of $k$ domino tilings of the Aztec diamond. We consider the tilings to be indexed by colors, which are ordered. Thinking of the different tilings as being overlaid one on top of the other, we define an \emph{interaction} between two of the tilings as an instance of a local configuration of the form
\[
\resizebox{2cm}{!}{
\begin{tikzpicture}[baseline = (current bounding box).center]
\draw[lightgray] (0,0) rectangle (1,1);
\draw[lightgray,fill=lightgray] (1,0) rectangle (2,1);

\draw[very thick, blue] (0,0) rectangle (2,1);
\draw[very thick, red] (0.05,0.05) rectangle (2.05,1.05);
\end{tikzpicture}
}
\text{, }
\resizebox{3cm}{!}{
\begin{tikzpicture}[baseline = (current bounding box).center]
\draw[lightgray] (0,0) rectangle (1,1);
\draw[lightgray,fill=lightgray] (1,0) rectangle (2,1);
\draw[lightgray,fill=lightgray] (-1,0) rectangle (0,1);
\draw[very thick, blue] (0,0) rectangle (2,1);
\draw[very thick, red] (-1.05,0.05) rectangle (1.05,1.05);
\end{tikzpicture}
}
\text{, }
\resizebox{2cm}{!}{
\begin{tikzpicture}[baseline = (current bounding box).center]
\draw[lightgray] (0,0) rectangle (1,1);
\draw[lightgray,fill=lightgray] (1,0) rectangle (2,1);
\draw[lightgray,fill=lightgray] (0,1) rectangle (1,2);
\draw[very thick, blue] (0,0) rectangle (2,1);
\draw[very thick, red] (0.05,0.05) rectangle (1.05,2.05);
\end{tikzpicture} 
}
\text{, or}
\resizebox{2cm}{!}{
\begin{tikzpicture}[baseline = (current bounding box).center]
\draw[lightgray] (0,0) rectangle (1,1);
\draw[lightgray,fill=lightgray] (0,0) rectangle (1,1);
\draw[lightgray,fill=lightgray] (1,1) rectangle (2,2);
\draw[very thick, red] (0,0) rectangle (2,1);
\draw[very thick, blue] (1.05,0.05) rectangle (2.05,2.05);
\end{tikzpicture}
}
\]
where above blue is the smaller color in our ordering and red the larger. We assign a \emph{weight} to the $k$-tilings given by $t^{\text{\# interactions}} $.

Even restricting to just~$k=2$, this distribution on~$2$-tilings is an integrable one parameter deformation of the double dimer model, which is obtained by setting the interaction strength~$t = 1$. In the large~$N$ limit, the model appears to exhibit many of the same phenomena as the dimer model, including arctic curves and limit shapes. On the other hand, the properties of these limit shapes and arctic curves appear to be very different from those observed in the dimer model. See Appendix~\ref{sec:sim} for a brief discussion and several simulations.

In addition to questions about limit shapes and arctic curves, there are many other natural questions one can ask about the coupling between the~$2$ (or~$k$) interacting tilings in the scaling limit. For example, near the arctic curve we expect to observe a one parameter deformation of~$k$ independent Airy processes, in which each color's edge fluctuations are coupled in a nontrivial way. It would also be interesting to study the global height fluctuations, and in particular how the fluctuations of different colors are coupled together. We expect that the efficient sampling algorithm provided by our main theorem below could be an important tool for the investigation of these questions.

The following is a special case of the main result.
\begin{thm}\label{thm:basic}
The following algorithm generates a random $k$-tiling of the rank-$N$ Aztec diamond with probability proportional to its weight.

\noindent Algorithm: Start with a rank-0 Aztec diamond. To get from a rank-$(T-1)$ to rank-$T$ $k$-tiling,
\begin{enumerate}
    \item Slide and destroy as in the normal domino shuffle, independently for each color.
    \item Fill in empty $2\times 2$ squares according to the rule:
    \begin{enumerate}
        \item For the smallest color put two horizontal dominoes with probability $$\frac{t^{\#_1(1)}}{1+t^{\#_1(1)}}$$ where 
        \[
        \#_1(l) = \# \text{ colors $m>l$ that locally have } 
\resizebox{0.08\textwidth}{!}{
\begin{tikzpicture}[baseline = (current bounding box).center]
\draw[] (0,0) rectangle (1,1); \draw[] (1,1) rectangle (2,2);
\draw[fill=lightgray] (0,1) rectangle (1,2); \draw[fill=lightgray] (1,0) rectangle (2,1);  
\draw[very thick, red] (0.05,0.05) rectangle (1.05,2.05);
\end{tikzpicture}}
\text{ or } 
\resizebox{0.08\textwidth}{!}{
\begin{tikzpicture}[baseline = (current bounding box).center]
\draw[] (0,0) rectangle (1,1); \draw[] (1,1) rectangle (2,2);
\draw[fill=lightgray] (0,1) rectangle (1,2); \draw[fill=lightgray] (1,0) rectangle (2,1);  
\draw[very thick, red] (0.05,0.05) rectangle (2.05,1.05);
\end{tikzpicture}}
\text{ or creation.}
        \]
        Do all of these first.
        \item Now do all the larger colors from smallest to largest. For color $l>1$, put two horizontal dominoes with probability 
        $$\frac{t^{\#_1(l)+\#_2(l)}}{1+t^{\#_1(l)+\#_2(l)}}$$ where 
        \[
        \#_2(l) = \# \text{ colors $m<l$ that locally have } 
\resizebox{0.12\textwidth}{!}{
\begin{tikzpicture}[baseline = (current bounding box).center]
\draw[] (0,0) rectangle (1,1); \draw[] (1,1) rectangle (2,2);
\draw[fill=lightgray] (0,1) rectangle (1,2); \draw[fill=lightgray] (1,0) rectangle (2,1);  
\draw[very thick, blue] (1,1) rectangle (3,2);
\end{tikzpicture}}
\text{ or } 
\resizebox{0.08\textwidth}{!}{
\begin{tikzpicture}[baseline = (current bounding box).center]
\draw[] (0,0) rectangle (1,1); \draw[] (1,1) rectangle (2,2);
\draw[fill=lightgray] (0,1) rectangle (1,2); \draw[fill=lightgray] (1,0) rectangle (2,1);  
\draw[very thick, blue] (1,1) rectangle (2,3);
\end{tikzpicture}}
        \]
    and $\#_1(l)$ is as in part (a).
    \end{enumerate}

    This gives a $k$-tiling of the Aztec diamond whose rank has increased by one. 
\end{enumerate}
Repeat steps (2) and (3) until you get a rank-$N$ Aztec diamond.
\end{thm}

The main tools we use are the Cauchy and branching identities for the LLT polynomials, as these allow us to apply a general construction of Borodin and Ferrari~\cite{Ferrari2008}. In fact, a standard bijection allows one to view a $k$-tiling as a tuple of interlaced particle arrays. The aforementioned construction (after some calculation) prescribes explicit transition probabilities for these particles, such that if the initial particle positions correspond to a random rank-$N$ tiling, then the update generates a random $k$-tiling of rank-($N+1$). Using the bijection to interpret the dynamics as local moves on dominos, we obtain the shuffling algorithm in Theorem~\ref{thm:main} in the text, which gives Theorem~\ref{thm:basic} by setting the weights to be uniform. This Markov chain on colored particle arrays generalizes the Markov chain on a single particle array described in~\cite[Section 2]{borodin2015random}, which corresponds to the usual shuffling algorithm. 

In addition to the proof of Theorem~\ref{thm:basic} using LLT polynomials, in Appendix~\ref{sec:genShuf} we give an alternative proof which employs a local resampling procedure which generalizes the resampling coming from the spider move. The resampling relies on a set of relationships between local partition functions, which are listed in Lemma~\ref{lem:2spider}. Contrary to the one color case, these relations are not sufficient to produce a shuffling algorithm for~$k$-tilings with arbitrary weights. However, they can still be used to construct a shuffling algorithm for certain choices of weights, including uniform weights, which is the setting of Theorem~\ref{thm:basic}.

The paper is organized as follows:
\begin{enumerate}
\item In Section \ref{sec:background}, we give a brief review of background material. We begin by reviewing the Aztec diamond and stating some fundamental results. We focus on highlighting the relationship with interlacing partitions, interlacing arrays of particles, and Schur polynomials which will be useful in later sections. We then define the $k$-tilings. We state some fundamental results from \cite{LLTaztec}. We also state the necessary properties of the LLT polynomials that will be used in the subsequent sections. Of particular importance to what follows is the bijection between tilings and particle configurations.

\item In Section \ref{sec:Markov}, we present the Markov chain on colored interlacing particle configurations which preserves a class of probability measures on colored particle arrays called `LLT processes'. First, we review the one color case, which is the Schur case, and then we describe the generalization to multiple colors, which is powered by LLT polynomials.

\item In Section \ref{sec:shuffling}, we interpret the particle dynamics described in the previous Section as an operation on dominos. We review the shuffling algorithm for a single tiling of the Aztec diamond before describing the corresponding result for the $k$-tiling. The main result is an algorithm for generating random $k$-tilings with probability proportional to their weight. 

\item In Section \ref{sec:conclusion}, we summarize our results and give some possible avenues of future research.

\item In Appendix \ref{sec:genShuf}, we give an alternate description of our shuffling algorithm in terms of a generalization of the `spider move' on the underlying double dimer model.

\item Finally, in Appendix \ref{sec:sim}, we present some simulations of the $k$-tilings generated using our shuffling algorithm. As noted above, the coupled tilings appear to exhibit limit shapes and arctic curves. We give a discussion of the apparent features.

\end{enumerate}

\noindent{\bf Acknowledgements.} The authors would like to thank Alexei Borodin, Sylvie Corteel, and Ananth Sridhar for many useful discussions.

\section{Background}\label{sec:background}

\subsection{Tilings of Aztec diamond}

\subsubsection{The Aztec diamond}
\label{subsec:ADT}
Let $A_{N+1}$ be the union of faces of $\mathbb{Z}^2$ which are entirely contained in the region $|x| +|y| \leq N+1$. A tiling of the Aztec diamond of rank $N$ is a tiling of the region $A_{N+1}$ with $2 \times 1$ or $1 \times 2$ dominos. See Figure \ref{fig:ADex} for an example.

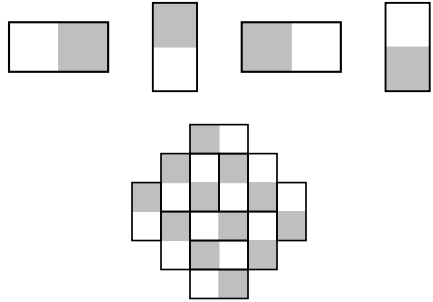
\begin{figure}[ht]
\centering\begin{tabular}{c}
\begin{tabular}{cccc}
\resizebox{1.5cm}{!}{
\begin{tikzpicture}[baseline = (current bounding box).center]
\draw[lightgray, fill=lightgray] (0,0) rectangle (1,1);
\draw[very thick] (-1,0) rectangle (1,1);
\end{tikzpicture}
}
& 
\resizebox{0.75cm}{!}{
\begin{tikzpicture}[baseline = (current bounding box).center]
\draw[lightgray, fill=lightgray] (0,0) rectangle (1,1);
\draw[very thick] (0,-1) rectangle (1,1);
\end{tikzpicture}
}
& 
\resizebox{1.5cm}{!}{
\begin{tikzpicture}[baseline = (current bounding box).center]
\draw[lightgray, fill=lightgray] (0,0) rectangle (1,1);
\draw[very thick] (0,0) rectangle (2,1);
\end{tikzpicture}
}
& 
\resizebox{0.75cm}{!}{
\begin{tikzpicture}[baseline = (current bounding box).center]
\draw[lightgray, fill=lightgray] (0,0) rectangle (1,1);
\draw[very thick] (0,0) rectangle (1,2);
\end{tikzpicture}
}
\end{tabular} \\ \\
\resizebox{0.2\textwidth}{!}{
\begin{tikzpicture}[baseline = (current bounding box).center]
\checkerboard{2}
\draw[ultra thick] (0,-1) rectangle (1,1); \draw[ultra thick] (1,-2) rectangle (2,0); \draw[ultra thick] (1,0) rectangle (2,2);
\draw[ultra thick] (2,-3) rectangle (4,-2); \draw[ultra thick] (2,-1) rectangle (4,0);
\draw[ultra thick] (2,2) rectangle (4,3);
\draw[ultra thick] (2,0) rectangle (3,2); \draw[ultra thick] (2,-2) rectangle (4,-1); \draw[ultra thick] (3,0) rectangle (4,2); 
\draw[ultra thick] (4,-2) rectangle (5,0); \draw[ultra thick] (4,0) rectangle (5,2); 
\draw[ultra thick] (5,-1) rectangle (6,1);  
\end{tikzpicture}
}
\end{tabular}
\caption{The four possible dominos and an example of a tiling of the Aztec diamond of rank 3.} \label{fig:ADex}
\end{figure}

Label the faces by $(i, j) \in (\mathbb{Z}+\frac{1}{2})^2$, and label the diagonals from $0$ to $2 N$ by declaring that the face $(i, j)$ is on diagonal $j-i + N$. 

We assign a (position dependent) weight to each domino in a tiling. Let $C_N = (c_1, \dots, c_N), B_N = (b_1, \dots, b_N)$ be two tuples of real numbers. The domino weights we are interested in are given by:
\begin{itemize}
    \item Suppose a horizontal domino $D$ is occupying the two squares $(i, j), (i+1, j)$ with $(i, j)$ on diagonal $2 m-1$. Then the weight of this domino is $c_m$.
    \item Suppose a horizontal domino $D$ in a tiling is occupying the two squares $(i, j), (i+1, j)$ with $(i, j)$ on diagonal $2 m$. Then the weight of $D$ is $b_{N-m+1}$.
    \item  The weights of vertical dominos are $1$.
\end{itemize}
 The weight of a whole tiling tiling $T $ is given by the product of the weights of each domino,  
$$\text{wt}(T) = \prod_{\text{dominos } D \in T} \text{wt}(D) \;\;.$$
Define the rank-$N$ Aztec diamond partition function with these weights as 
$$Z_{AD}(C_N,B_N) \coloneqq \sum_{\text{tilings } T} \text{wt}(T) \;\;.
$$
The probability of a random rank-$N$ tiling is given by 
$$\frac{\text{wt}(T)}{Z_{AD}(C_N,B_N) } \;\;.$$

\begin{thm}[\cite{steep}]\label{thm:ADPF}
The partition function of the Aztec diamond of rank $N$ with the above weights is given by
\[
Z_{AD}(C_N,B_N) = \prod_{1\le i \le j \le N} (1+c_i b_{N-j+1}).
\]
\end{thm}

One of the ways this theorem can be proved is via the machinery of \emph{Schur polynomials}, the basics of which we briefly review in the next subsection.

\subsubsection{Schur polynomials and interlacing partitions}
\label{subsubsec:particle_arrays}

A \emph{partition} $\lambda = (\lambda_1,\lambda_2, \lambda_3, \ldots )$ is a non-negative sequence of integers such that $\lambda_1 \ge \lambda_2 \ge \lambda_3 \ge \ldots$. We associate to $\lambda$ its \emph{Young diagram} $D(\lambda) \subseteq \mathbb{Z}\times \mathbb{Z}$, given as
\[
D(\lambda) = \{(i,j) \mid 1 \leq i \leq \ell(\lambda), \; 1 \leq j \leq \lambda_i \}
\]
We draw our diagrams in French notation, in the first quadrant,  as shown in the below example:
\[
\lambda = (4,2,1), \qquad D(\lambda) = 
\ytableausetup{aligntableaux=center}
\begin{ytableau} \\ & \\ & & \bullet & \end{ytableau} 
\]
We refer to the elements in $D(\lambda)$ as {\em cells}. The cell labelled above has coordinates (1,3). The \emph{content} of a cell $u = (i,j)$ in the $i$-th row and $j$-th column of the Young diagram is $c(u) = i-j$. The marked cell above has content $c((1,3))=-2$. The size of a partition is the number of cells in its Young diagram and is denoted by $|\lambda|$.  The above partition has size $|\lambda| =4+2+1=7$. See Figure \ref{fig:YoungDiagramex} for another example.

Given two partitions, $\lambda$ and $\mu$, we say that $\mu$ is \emph{contained in} $\lambda$ and write $\mu \subset \lambda$ if the Young diagram of $\mu$ is contained within $\lambda$. Given that $\mu\subset \lambda$ we can define the \emph{skew diagram} $\lambda/\mu$ as the Young diagram of $\lambda$ with the cells from the Young diagram of $\mu$ removed.

To a partition we can associate an infinite sequence of particles and holes by assigning a particle to every vertical edge on the boundary of its Young diagram, and a hole to every horizontal edge. This is known as the \emph{Maya diagram} of the partition. See Figure \ref{fig:YoungDiagramex}. Note that the Maya diagram has a unique content line such that the number of particles to the right of this line is equal to the number of holes to the left. We call this the \emph{zero-content line} and view it as the center of our Maya diagram. If we place the Maya diagram on $\ZZ + \frac{1}{2}$ centered at zero then the position $x_i$ of the $i$-th particles (counting from right to left) is given by
\[
x_i = \lambda_i- i +\frac{1}{2}.
\]

For every partition $\lambda$ we can associate a second partition $\lambda'$ known as the \emph{conjugate} of $\lambda$. The conjugate $\lambda'$ is defined as the partition whose Young diagram is given by reflecting the Young diagram of $\lambda$ across its zero-content line. For example, $\lambda=(4,3,2,2,1)$, the partition in Fig. \ref{fig:YoungDiagramex},  has conjugate  $\lambda'= (5,4,2,1)$.

Given two partitions $\lambda$ and $\mu$ say that $\lambda$ and $\mu$ \emph{interlace} if
\[
\lambda_1\ge \mu_1 \ge \lambda_2 \ge \mu_2 \ge \ldots
\]
and write $\lambda \succeq \mu $. Say that $\lambda$ and $\mu$ \emph{co-interlace} if their conjugate partitions interlace and write $\lambda \succeq' \mu $.  Note that $\mu \preceq \lambda$ implies that $\mu \subset \lambda$.

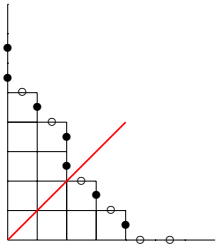
\begin{figure}[ht]
\centering\resizebox{3cm}{!}{
\begin{tikzpicture}[baseline=(current bounding box.center)]
\draw (0,0) grid (2,4); \draw (0,4) grid (1,5); \draw (2,0) grid (3,2); \draw (3,0) grid (4,1);
\node[scale=2] at (0.5,5) {$\circ$}; \node[scale=2] at (1,4.5) {$\bullet$}; \node[scale=2] at (1.5,4) {$\circ$}; \node[scale=2] at (2,3.5) {$\bullet$}; \node[scale=2] at (2,2.5) {$\bullet$}; \node[scale=2] at (2.5,2) {$\circ$}; \node[scale=2] at (3,1.5) {$\bullet$}; \node[scale=2] at (3.5,1) {$\circ$}; \node[scale=2] at (4,0.5) {$\bullet$}; 
\draw (4,0) grid (7,0); \node[scale=2] at (4.5,0) {$\circ$}; \node[scale=2] at (5.5,0) {$\circ$}; 
\draw (0,5) grid (0,8); \node[scale=2] at (0,5.5) {$\bullet$}; \node[scale=2] at (0,6.5) {$\bullet$}; 
\draw[ultra thick, red] (0,0)--(4,4);
\end{tikzpicture} }
\caption{The Young diagram of the partition $\lambda=(4,3,2,2,1)$ along with its Maya diagram: $\ldots \bullet \bullet \enspace \circ\bullet\circ\bullet\bullet\circ\bullet\circ\bullet \enspace \circ \circ \ldots$. The red line indicates the zero-content line of the Young diagram. We have $|\lambda| = 12$.}\label{fig:YoungDiagramex}
\end{figure}

Recall that given two partitions $\mu \subset\lambda$ and variables $X_n = (x_1,\dots, x_n)$ the skew-Schur polynomial is defined as
\[
s_{\lambda/\mu}(X_n) = \sum_{\sigma\in\text{SSYT}(\lambda/\mu)} x^\sigma.
\]
Here the sum is over all \emph{semi-standard Young tableaux} of shape $\lambda/\mu$, where a semi-standard Young tableaux is a filling of the cells of the diagram by the integers $1,\ldots,n$ such that they are weakly increase along the rows and strictly increasing up the columns. 

Here we state some basic identities satisfied by Schur polynomials that will be useful for us (see \cite{macdonald1998symmetric} for more details). For any sets of variables $X = (x_1, \dots, x_l), Y = (y_1, \dots, y_p)$, we have
\begin{prop}[\cite{macdonald1998symmetric}]\label{prop:basicSchur}
Branching rule:
\begin{align*}
    \sum_{\lambda}  s_{\lambda/\nu}(X) s_{\nu/\mu}(Y) 
    &=   s_{\lambda/\mu}(X, Y)
\end{align*}

Dual Skew-Cauchy Identity:
\begin{align*}
     &\sum_{\lambda}  s_{\lambda/\nu}(X) s_{\lambda'/\mu'}(Y) \\
    &= \left( \prod_{i, j} (1+x_i y_j) \right) \sum_{\lambda}  s_{\nu'/\lambda'}(Y) s_{\mu/\lambda}(X)\;\;.
\end{align*}
\end{prop}

\subsubsection{The Aztec diamond and Schur processes}

There is a bijection between tilings of the Aztec diamond of rank $N$ and sequences of interlacing partitions
\begin{equation}\label{eq:interlace}
\emptyset \preceq \lambda^{(1)} \succeq' \mu^{(2)} \preceq \ldots \preceq \lambda^{(N-1)} \succeq' \mu^{(N)} \preceq \lambda^{(N)} \succeq' \emptyset.
\end{equation}
Given a tiling of the Aztec diamond of rank $N$ assign particles and holes to the dominos according to the rules
\[
\centering\resizebox{1cm}{!}{
\begin{tikzpicture}
\draw[lightgray] (0,0) rectangle (1,1);
\draw[lightgray,fill=lightgray] (1,0) rectangle (2,1);
\draw[very thick,fill=white] (0.5,0.5) circle (5pt);
\draw[very thick,fill=white] (1.5,0.5) circle (5pt);
\draw[very thick] (0,0) rectangle (2,1);
\end{tikzpicture}},\;\;\;
\centering
\resizebox{.5cm}{!}{
\begin{tikzpicture}
\draw[lightgray] (0,0) rectangle (1,1);
\draw[lightgray,fill=lightgray] (0,1) rectangle (1,2);
\draw[very thick] (0,0) rectangle (1,2);
\draw[very thick,fill] (0.5,0.5) circle (5pt);
\draw[very thick,fill] (0.5,1.5) circle (5pt);
\end{tikzpicture}},\;\;\;
\centering\resizebox{1cm}{!}{ \begin{tikzpicture}
\draw[lightgray,fill=lightgray] (0,0) rectangle (1,1);
\draw[lightgray] (1,0) rectangle (2,1);
\draw[very thick] (0,0) rectangle (2,1);
\draw[very thick,fill] (0.5,0.5) circle (5pt);
\draw[very thick,fill] (1.5,0.5) circle (5pt);
\end{tikzpicture}},\;\;\;
\centering\resizebox{.5cm}{!}{
\begin{tikzpicture}
\draw[lightgray, fill=lightgray] (0,0) rectangle (1,1);
\draw[lightgray] (0,1) rectangle (1,2);
\draw[very thick,fill=white] (0.5,0.5) circle (5pt);
\draw[very thick,fill=white] (0.5,1.5) circle (5pt);
\draw[very thick] (0,0) rectangle (1,2);
\end{tikzpicture}}.
\]
Along each diagonal slice of the Aztec diamond view the resulting sequence of particles and holes as the Maya diagram of some partition by extending it infinitely to the South-West with particles and infinitely to the North-East by holes. Let us index the slices starting from $0$, such that $\mu^{(i)}$ is the partition along slice $2i-2$ and $\lambda^{(i)}$ is the partition along slice $2i-1$. Note that $\mu^{(1)}=\mu^{(N+1)}=\emptyset$ is forced. Figure \ref{fig:ADandMayaDiagr} gives an example of our notation and this bijection. One can check \cite{steep, johansson2} that the requirement that these partitions come from a valid tiling is exactly the interlacing condition given in Eqn. (\ref{eq:interlace}).

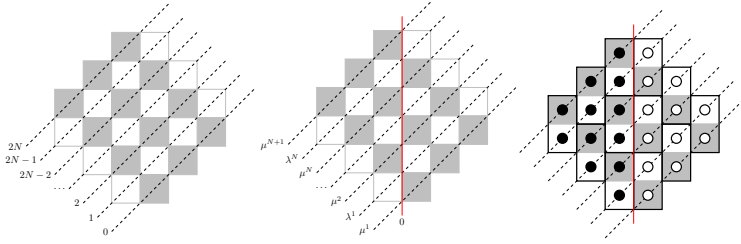
\begin{figure}[ht]
\centering\resizebox{10cm}{!}{
\begin{tikzpicture}
\checkerboard{2}
\draw[dashed] (2,-4)--(7,1); \node[below, left] at (2,-4) {$0$};
\draw[dashed] (1.5,-3.5)--(6.5,1.5); \node[below, left] at (1.5,-3.5) {$1$};
\draw[dashed] (1,-3)--(6,2); \node[below, left] at (1,-3) {$2$};
\draw[dashed] (0.5,-2.5)--(5.5,2.5); \node[below, left] at (0.5,-2.5) {$\ldots$};
\draw[dashed] (0,-2)--(5,3); \node[below, left] at (0,-2) {$2N-2$};
\draw[dashed] (-0.5,-1.5)--(4.5,3.5); \node[below, left] at (-0.5,-1.5) {$2N-1$};
\draw[dashed] (-1,-1)--(4,4); \node[below, left] at (-1,-1) {$2N$};
\end{tikzpicture}
\begin{tikzpicture}
\checkerboard{2}
\draw[dashed] (2,-4)--(7,1); \node[below, left] at (2,-4) {$\mu^1$};
\draw[dashed] (1.5,-3.5)--(6.5,1.5); \node[below, left] at (1.5,-3.5) {$\lambda^1$};
\draw[dashed] (1,-3)--(6,2); \node[below, left] at (1,-3) {$\mu^2$};
\draw[dashed] (0.5,-2.5)--(5.5,2.5); \node[below, left] at (0.5,-2.5) {$\ldots$};
\draw[dashed] (0,-2)--(5,3); \node[below, left] at (0,-2) {$\mu^{N}$};
\draw[dashed] (-0.5,-1.5)--(4.5,3.5); \node[below, left] at (-0.5,-1.5) {$\lambda^{N}$};
\draw[dashed] (-1,-1)--(4,4); \node[below, left] at (-1,-1) {$\mu^{N+1}$};
\draw[red,thick] (3,-3.5)--(3,3.5); \node[below] at (3,-3.5) {$0$};
\end{tikzpicture}
\begin{tikzpicture}
\checkerboard{2}
\draw[dashed] (2,-4)--(7,1);
\draw[dashed] (1.5,-3.5)--(6.5,1.5);
\draw[dashed] (1,-3)--(6,2); 
\draw[dashed] (0.5,-2.5)--(5.5,2.5);
\draw[dashed] (0,-2)--(5,3);
\draw[dashed] (-0.5,-1.5)--(4.5,3.5); 
\draw[dashed] (-1,-1)--(4,4);
\draw[very thick,fill] (0.5,0.5) circle (5pt); 
\draw[very thick,fill] (0.5,-0.5) circle (5pt); 
\draw[very thick,fill] (1.5,0.5) circle (5pt); \draw[very thick,fill] (2.5,0.5) circle (5pt); \draw[very thick,fill=white] (3.5,0.5) circle (5pt); \draw[very thick,fill=white] (4.5,0.5) circle (5pt); \draw[very thick,fill=white] (5.5,0.5) circle (5pt);
\draw[very thick,fill] (1.5,1.5) circle (5pt); \draw[very thick,fill] (2.5,1.5) circle (5pt); \draw[very thick,fill=white] (3.5,1.5) circle (5pt); \draw[very thick,fill=white] (4.5,1.5) circle (5pt);
\draw[very thick,fill] (2.5,2.5) circle (5pt); \draw[very thick,fill=white] (3.5,2.5) circle (5pt);
\draw[very thick,fill] (1.5,-0.5) circle (5pt); \draw[very thick,fill] (2.5,-0.5) circle (5pt); \draw[very thick,fill=white] (3.5,-0.5) circle (5pt); \draw[very thick,fill=white] (4.5,-0.5) circle (5pt); \draw[very thick,fill=white] (5.5,-0.5) circle (5pt);
\draw[very thick,,fill] (1.5,-1.5) circle (5pt); \draw[very thick,fill] (2.5,-1.5) circle (5pt); \draw[very thick,fill=white] (3.5,-1.5) circle (5pt); \draw[very thick,fill=white] (4.5,-1.5) circle (5pt);
\draw[very thick,fill] (2.5,-2.5) circle (5pt); \draw[very thick,fill=white] (3.5,-2.5) circle (5pt);
\draw[very thick] (0,-1) rectangle (1,1); \draw[very thick] (1,0) rectangle (2,2); \draw[very thick] (2,1) rectangle (3,3);
\draw[very thick,] (3,1) rectangle (4,3); \draw[very thick] (4,0) rectangle (5,2);
\draw[very thick] (5,-1) rectangle (6,1);
\draw[very thick] (1,-2) rectangle (2,0); \draw[very thick] (2,-3) rectangle (3,-1); \draw[very thick] (3,-3) rectangle (4,-1); 
\draw[very thick] (4,-2) rectangle (5,0); \draw[very thick] (2,-1) rectangle (3,1); 
\draw[very thick] (3,-1) rectangle (4,1);
\draw[red,thick] (3,-3.5)--(3,3.5); 
\end{tikzpicture}
}
\caption{Left and Center: Assigning partitions to slices of the Aztec diamond. The red line indicates the zero content line for the partitions. Right: The tiling corresponding to all partitions being empty for the Aztec diamond of rank 3.} \label{fig:ADandMayaDiagr}
\end{figure}

Given the bijection from tilings to sequences of interlacing partitions 
\[
\lambda^{(1)}, \mu^{(2)}, \lambda^{(2)}, \dots, \mu^{(N)}, \lambda^{(N)}
\]
described above, one can write the weight of the tiling in terms of Schur polynomials. The weight of a tiling is given by
\[
s_{\lambda^{(1)}}(c_1) s_{(\lambda^{(1)}/\mu^{(2)})'}(b_N) s_{\lambda^{(2)}/\mu^{(2)}}(c_2) \ldots s_{\lambda^{(N)}/\mu^{(N)}}(c_N) s_{(\lambda^{(N)})'}(b_1)
\]
where here we used the notation $(\lambda/\mu)' \coloneqq \lambda'/\mu' $. 

\begin{remark}\label{rmk:schur}
A probability measure on sequences of partitions of the form 
$$\frac{1}{Z}s_{\lambda^{(1)}}(u_1) s_{(\lambda^{(1)}/\mu^{(2)})'}(v_1) s_{\lambda^{(2)}/\mu^{(2)}}(u_2) \ldots s_{\lambda^{(N)}/\mu^{(N)}}(u_N) s_{(\lambda^{(N)})'}(v_N)$$
with $u_1,\dots, u_N, v_1,\dots, v_N \in \mathbb{R}_{>0}$ is a particular case of a \emph{Schur Process}. Schur processes are well-studied. Using them one can derive exact determinantal formulas for correlation functions and study various statistics of random tilings asymptotically as $N \rightarrow \infty$, see \cite{BorodinGorinSPB12} for a survey and see also \cite{okounkov2003correlation}. 
\end{remark}

In particular, we have 
\begin{prop}[\cite{steep}]
The partition function of the Aztec diamond of rank $N$ is given by
\[
\begin{aligned}
&Z_{AD}(C_N,B_N) =\\
& \qquad \sum s_{\lambda^{(1)}}(c_1) s_{(\lambda^{(1)}/\mu^{(2)})'}(b_N) s_{\lambda^{(2)}/\mu^{(2)}}(c_2) \ldots s_{\lambda^{(N)}/\mu^{(N)}}(c_N) s_{(\lambda^{(N)})'}(b_1) \;\;.
\end{aligned}
\]
where the sum is over all tuples of partitions $\lambda^{(1)}, \mu^{(2)}, \lambda^{(2)}, \dots, \mu^{(N)}, \lambda^{(N)}$ satisfying the interlacing condition Eqn. (\ref{eq:interlace}).
\end{prop}
By repeated applications of the identities in Proposition \ref{prop:basicSchur}, the above simplifies to the product in Theorem \ref{thm:ADPF}.

It will often be more convenient to consider only the particle positions. From this point of view, the tiling becomes an array of interlacing particles. Let $x^{(n)} = \{x_1^{(n)} > x^{(n)}_2 > \cdots > x^{(n)}_n \}$ be the position of particles corresponding to $\lambda^{(n)}$, and $y^{(n)} = \{y_1^{(n)} > y^{(n)}_2 > \cdots > y^{(n)}_{n-1} \}$ those corresponding to $\mu^{(n)}$. A set of particle positions corresponds to a tiling if and only if they satisfy the interlacing conditions 
\begin{equation}\label{eq:particleInterlace}
\begin{aligned}
x_i^{(n)} &\geq y^{(n)}_i > x_{i+1}^{(n)} \\
x_i^{(n)} &\geq y^{(n+1)}_i \geq x_{i}^{(n)}-1 
\end{aligned}
\end{equation}
and the bounds
\begin{equation}\label{eq:particleBounds}
\begin{aligned}
    -n + \frac{1}{2} &\leq x_i^{(n)} \leq N - n +\frac{1}{2} \\
-n +1 + \frac{1}{2} &\leq y_i^{(n)} \leq N - n +\frac{1}{2}
\end{aligned}
\end{equation}
for each $n = 1, \dots, N$. See Figure \ref{ex:3til-bijection} for several examples.


\subsection{Coupled tilings and LLT polynomials}\label{sec:couples_tilings}

\subsubsection{$k$-tilings of the Aztec Diamond}
In this section we define the interacting $k$-tilings of the Aztec Diamond. See \cite{LLTaztec} for a more detailed discussion.

Consider an Aztec diamond of rank $N$. A \emph{$k$-tiling} $\bm{T} = (T^{(1)},\ldots,T^{(k)})$ is a $k$-tuple of tilings of the Aztec diamond. We will often draw the different tilings in different colors (see Figure \ref{ex:3til-bijection}) and refer to tiling $T^{(a)}$ as being color $a$. We order the colors so that color $a$ is smaller than color $b$ if $a<b$.

If $C_N = (c_1, \dots, c_N), B_N = (b_1, \dots, b_N)$, each of the tilings $T^{(a)}$ has its own weight $\text{wt}(T^{(a)})$ by giving the horizontal dominos weights $c_m, b_{N-m+1}$ on diagonals $2 m -1 $ and $2m$, respectively, as described in subsection \ref{subsec:ADT}. Now we define an interaction between pairs of tilings. Consider two tilings $T^{(a)}$ and $T^{(b)}$ with $a<b$. Let blue be the smaller color and red the larger color. We define an interaction between the two tilings to be any instance of the local configuration
\[
\resizebox{2cm}{!}{
\begin{tikzpicture}[baseline = (current bounding box).center]
\draw[lightgray] (0,0) rectangle (1,1);
\draw[lightgray,fill=lightgray] (1,0) rectangle (2,1);

\draw[very thick, blue] (0,0) rectangle (2,1);
\draw[very thick, red] (0.05,0.05) rectangle (2.05,1.05);
\end{tikzpicture}
}
\text{, }
\resizebox{3cm}{!}{
\begin{tikzpicture}[baseline = (current bounding box).center]
\draw[lightgray] (0,0) rectangle (1,1);
\draw[lightgray,fill=lightgray] (1,0) rectangle (2,1);
\draw[lightgray,fill=lightgray] (-1,0) rectangle (0,1);
\draw[very thick, blue] (0,0) rectangle (2,1);
\draw[very thick, red] (-1.05,0.05) rectangle (1.05,1.05);
\end{tikzpicture}
}
\text{, }
\resizebox{2cm}{!}{
\begin{tikzpicture}[baseline = (current bounding box).center]
\draw[lightgray] (0,0) rectangle (1,1);
\draw[lightgray,fill=lightgray] (1,0) rectangle (2,1);
\draw[lightgray,fill=lightgray] (0,1) rectangle (1,2);
\draw[very thick, blue] (0,0) rectangle (2,1);
\draw[very thick, red] (0.05,0.05) rectangle (1.05,2.05);
\end{tikzpicture} 
}
\text{, or}
\resizebox{2cm}{!}{
\begin{tikzpicture}[baseline = (current bounding box).center]
\draw[lightgray] (0,0) rectangle (1,1);
\draw[lightgray,fill=lightgray] (0,0) rectangle (1,1);
\draw[lightgray,fill=lightgray] (1,1) rectangle (2,2);
\draw[very thick, red] (0,0) rectangle (2,1);
\draw[very thick, blue] (1.05,0.05) rectangle (2.05,2.05);
\end{tikzpicture}
}
\]
when the two tilings are superimposed on top of one another. 
\begin{remark}
A $k$-tiling with these interactions is called the ``white-pink" model in \cite{LLTaztec}.
\end{remark}
Now we can define the weights we use for $k$-tilings.
\begin{definition}
The weight of a $k$-tiling is 
$$\text{wt}(\bm{T}) \coloneqq t^{\text{\# interactions}}\prod_{i=1}^k \text{wt}(T^{(i)}) \;\;.$$
\end{definition}
As usual we will study the probability measure on $k$-tilings where the probability of each $k$-tiling is proportional to its weight:
$$P(\mathbf{T}) = \frac{\text{wt}(\bm{T})}{Z^{(k)}_{AD}(C_N,B_N;t)} \;\;.$$
The partition function $Z^{(k)}_{AD}(C_N,B_N;t)$ can be written in a simple product form.
\begin{thm}[\cite{LLTaztec}]
The partition function $Z^{(k)}_{AD}(C_N,B_N;t)$ of the $k$-tilings of the Aztec diamond of rank $N$ is given by
\begin{equation} \label{eq:kADpartitionfunction}
    Z^{(k)}_{AD}(C_N,B_N;t) = \prod_{l=0}^{k-1} \prod_{1\le i \le j \le N} (1+c_i b_{N-j+1}t^l).
\end{equation}
\end{thm}
 Note that when $t=1$ the tilings are independent and we have
\[
 Z^{(k)}_{AD}(C_N,B_N;1) = \left(Z_{AD}(C_N,B_N)\right)^k,
\]
that is, the partition function is a product of $k$ copies of the partition function for a single tiling of the Aztec diamond, as one would expect. More surprisingly, when $t=0$ we have
\[
 Z^{(k)}_{AD}(C_N,B_N;0) = Z_{AD}(C_N,B_N),
\]
that is, the partition function of the $k$-tiling is equal to the partition function of a single tiling. See \cite{LLTaztec} for a bijective proof of this fact.

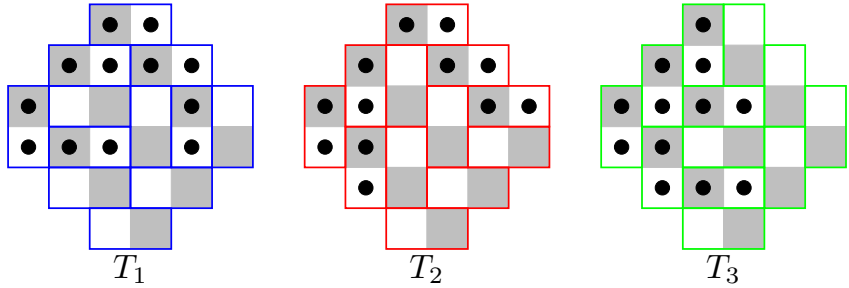
\begin{figure}[ht]
\begin{center}
\resizebox{\textwidth}{!}{
\begin{tabular}{ccc}
\resizebox{0.22\textwidth}{!}{
\begin{tikzpicture}[baseline = (current bounding box).center]
\checkerboard{2}
\draw[fill=black] (4.5,-0.5) circle (5pt);
\draw[fill=black] (4.5,0.5) circle (5pt);
\draw[fill=black] (2.5,-0.5) circle (5pt); \draw[fill=black] (4.5,1.5) circle (5pt);
\draw[fill=black] (1.5,-0.5) circle (5pt); \draw[fill=black] (3.5,1.5) circle (5pt);
\draw[fill=black] (0.5,-0.5) circle (5pt); \draw[fill=black] (2.5,1.5) circle (5pt); \draw[fill=black] (3.5,2.5) circle (5pt);
\draw[fill=black] (0.5,0.5) circle (5pt); \draw[fill=black] (1.5,1.5) circle (5pt); \draw[fill=black] (2.5,2.5) circle (5pt); 

\draw[very thick, blue] (0,-1) rectangle (1,1); \draw[very thick, blue] (1,-2) rectangle (3,-1); \draw[very thick, blue] (1,-1) rectangle (3,0); \draw[very thick, blue] (1,0) rectangle (3,1); \draw[very thick, blue] (1,1) rectangle (3,2); 
\draw[very thick, blue] (2,-3) rectangle (4,-2); \draw[very thick, blue] (2,2) rectangle (4,3);
\draw[very thick, blue] (3,-2) rectangle (5,-1); \draw[very thick, blue] (3,-1) rectangle (4,1); \draw[very thick, blue] (3,1) rectangle (5,2);
\draw[very thick, blue] (4,-1) rectangle (5,1);
\draw[very thick, blue] (5,-1) rectangle (6,1);
\end{tikzpicture}
}
&
\resizebox{0.22\textwidth}{!}{
\begin{tikzpicture}[baseline = (current bounding box).center]
\checkerboard{2}
\draw[fill=black] (5.5,0.5) circle (5pt);
\draw[fill=black] (4.5,0.5) circle (5pt);
\draw[fill=black] (1.5,-1.5) circle (5pt); \draw[fill=black] (4.5,1.5) circle (5pt);
\draw[fill=black] (1.5,-0.5) circle (5pt); \draw[fill=black] (3.5,1.5) circle (5pt);
\draw[fill=black] (0.5,-0.5) circle (5pt); \draw[fill=black] (1.5,0.5) circle (5pt); \draw[fill=black] (3.5,2.5) circle (5pt);
\draw[fill=black] (0.5,0.5) circle (5pt); \draw[fill=black] (1.5,1.5) circle (5pt); \draw[fill=black] (2.5,2.5) circle (5pt); 

\draw[very thick, red] (0,-1) rectangle (1,1); \draw[very thick, red] (1,-2) rectangle (2,0);  \draw[very thick, red] (1,0) rectangle (2,2);  
\draw[very thick, red] (2,-3) rectangle (4,-2); \draw[very thick, red] (2,-2) rectangle (3,0); \draw[very thick, red] (2,0) rectangle (3,2); \draw[very thick, red] (2,2) rectangle (4,3);
\draw[very thick, red] (3,-2) rectangle (5,-1); \draw[very thick, red] (3,-1) rectangle (4,1); \draw[very thick, red] (3,1) rectangle (5,2);
\draw[very thick, red] (4,-1) rectangle (6,0); \draw[very thick, red] (4,0) rectangle (6,1);
\end{tikzpicture}
}
&
\resizebox{0.22\textwidth}{!}{
\begin{tikzpicture}[baseline = (current bounding box).center]
\checkerboard{2}
\draw[fill=black] (3.5,-1.5) circle (5pt);
\draw[fill=black] (2.5,-1.5) circle (5pt);
\draw[fill=black] (1.5,-1.5) circle (5pt); \draw[fill=black] (3.5,0.5) circle (5pt);
\draw[fill=black] (1.5,-0.5) circle (5pt); \draw[fill=black] (2.5,0.5) circle (5pt);
\draw[fill=black] (0.5,-0.5) circle (5pt); \draw[fill=black] (1.5,0.5) circle (5pt); \draw[fill=black] (2.5,1.5) circle (5pt);
\draw[fill=black] (0.5,0.5) circle (5pt); \draw[fill=black] (1.5,1.5) circle (5pt); \draw[fill=black] (2.5,2.5) circle (5pt);

\draw[very thick, green] (0,-1) rectangle (1,1); 
\draw[very thick, green] (1,-2) rectangle (2,0);  \draw[very thick, green] (1,0) rectangle (2,2);  
\draw[very thick, green] (2,-3) rectangle (4,-2); \draw[very thick, green] (2,-2) rectangle (4,-1); \draw[very thick, green] (2,-1) rectangle (4,0); \draw[very thick, green] (2,0) rectangle (4,1); \draw[very thick, green] (2,1) rectangle (3,3);
\draw[very thick, green] (3,1) rectangle (4,3);
\draw[very thick, green] (4,-2) rectangle (5,0); \draw[very thick, green] (4,0) rectangle (5,2);
\draw[very thick, green] (5,-1) rectangle (6,1);
\end{tikzpicture}
}
\\ 
$T_1$ & $T_2$ & $T_3$
\end{tabular}
}
\end{center}
\caption{An example of a 3-tiling of the rank-3 Aztec diamond along with the corresponding particles. This $k$-tiling corresponds to the sequence of interlacing tuples of partitions: $((0),(0),(0))\preceq ((2),(3),(1)) \succeq' ((2),(2),(0))\preceq ((2,1),(2,0),(1,0)) \succeq' ((1,0),(1,0),(0,0)) \preceq ((1,1,0),(1,0,0),(0,0,0)) \succeq'  ((0,0,0),(0,0,0),(0,0,0))$.
{\color{blue} }} 
\label{ex:3til-bijection}
\end{figure}

\subsubsection{LLT polynomials}
While Schur polynomials were valuable in studying the single tilings of the Aztec diamond, for the $k$-tilings it is the LLT polynomials that are useful. 

LLT polynomials are certain symmetric polynomials introduced by Lascoux, Lecler, and Thibon \cite{LLT1997} as the generating functions of semi-standard ribbon tableaux counting a spin statistic.  Recently, a version of these polynomials called coinversion LLT polynomials were constructed as the partition function of a class of integrable lattice models \cite{color1,LLT,CFsuper,GKsuper}. In \cite{LLTaztec}, the authors used the lattice model formulation of the coinversion LLT polynomials to compute the partition function of the interacting $k$-tilings of the Aztec diamond that are the focus of this paper.

Here we collect the relevant definitions and properties of the coinversion LLT polynomials.

Let $\bm{\lambda}/\bm{\mu} = (\lambda^{(1)}/\mu^{(1)},\ldots,\lambda^{(k)}/\mu^{(k)})$ and be a $k$-tuple of skew-partitions. Define $\text{SSYT}(\bm{\lambda}/\bm{\mu}) = \text{SSYT}(\lambda^{(1)}/\mu^{(1)})\times\ldots \times \text{SSYT}(\lambda^{(k)}/\mu^{(k)})$ to be $k$-tuples of semi-standard Young tableaux. Given a $k$-tuple of semi-standard fillings $\bm{\sigma}\in  \text{SSYT}(\bm{\lambda}/\bm{\mu})$, $\bm{\sigma} = (\sigma^{1},\ldots,\sigma^{k})$, we draw them so that Young diagrams are aligned along content lines. See Example \ref{exampleLLT}.
\begin{example}
	Let $\bm{\lambda}/\bm{\mu} = ((3,1), (2,2,2)/(1,1,1), (1), (2,1)/(2))$. Below is one possible semi-standard filling of $\bm{\lambda}/\bm{\mu}$. The top row labels the contents of each diagonal line.
	
	\ytableausetup{nosmalltableaux}
	\ytableausetup{nobaseline}
	\begin{center}
	\begin{ytableau}
		\none & \none & \none & \none & \none & \none & \none & \none & \none & \none[-3] & \none[-2] & \none[-1] & \none[0] & \none[1] & \none[2] \\			
		\none & \none & \none & \none & \none & \none & \none & \none &\none[\diag] &\none[\diag] &\none[\diag] &\none[\diag] & \none[\diag] & \none[\diag] \\
		\none & \none & \none & \none & \none & \none & \none &\none[\diag] &\none[\diag] & 3 &\none[\diag] &\none[\diag] & \none[\diag] & \none[\diag] \\
		\none & \none & \none & \none & \none &\none &\none[\diag] &\none[\diag] & \none[\diag] & *(lightgray) & *(lightgray) &\none[\diag] & \none[\diag] \\
		\none & \none & \none & \none &\none &\none[\diag] &\none[\diag] &\none[\diag] &\none[\diag] &\none[\diag] &\none[\diag] &\none[\diag] & \none \\
		\none & \none & \none &\none &\none[\diag] &\none[\diag] &\none[\diag] & *(yellow) 7 &\none[\diag] &\none[\diag] &\none[\diag] &\none \\
		\none & \none &\none &\none[\diag] &\none[\diag] &\none[\diag] &\none[\diag] &\none[\diag] &\none[\diag] &\none[\diag] &\none &\none \\
		\none & \none &\none[\diag] & *(lightgray) & 6 &\none[\diag] &\none[\diag] &\none[\diag] &\none[\diag] &\none &\none & \none \\
		\none & \none[\diag] & \none[\diag] & *(lightgray) & *(green) 4 &\none[\diag] &\none[\diag] &\none[\diag] &\none &\none & \none & \none \\
		\none[\diag] & \none[\diag] & \none[\diag] & *(lightgray) & 1 &\none[\diag] &\none[\diag] &\none &\none & \none & \none & \none \\
		\none[\diag] & \none[\diag] & \none[\diag] &\none[\diag] &\none[\diag] &\none[\diag] &\none &\none & \none & \none & \none & \none \\
		8 &\none[\diag] &\none[\diag] &\none[\diag] &\none[\diag] &\none &\none & \none & \none & \none & \none & \none \\
		*(yellow) 2 & *(green) 4 & 9 &\none[\diag] &\none &\none & \none & \none & \none & \none & \none & \none \\
	\end{ytableau}	
	\end{center}
        The green cells are an example of a coinversion triple with $x=4$, $y=4$, and $z=\infty$. The yellow cells are an example of an inversion triple with $x=0$, $y=2$, and $z=7$.
	\label{exampleLLT}
\end{example}

Define a coinversion triple to be a triple of entries  of the form
\begin{center}
\ytableausetup{nosmalltableaux}
\ytableausetup{nobaseline}
\begin{ytableau}
\none & x & z \\
\none & \none[\diag]  \\
y
\end{ytableau} 
\end{center}
where 
\begin{enumerate}
    \item  $y$ is an entry of $\sigma^{(a)}$ and $x,z$ are entries in $\sigma^{(b)}$,  with $1\le a < b \le k$.
    \item $y$ and $z$ lie along the same content line.
    \item $x\le y \le z$.
\end{enumerate}
While $y$ must be in the Young diagram of $\lambda^{(a)}/\mu^{(a)}$, we let $x=0$ and $z=\infty$ if they are not in the Young diagram of $\lambda^{(b)}/\mu^{(b)}$.

Define the coinversion LLT polynomial by
\begin{equation}
    \mathcal{L}_{\bm{\lambda}/\bm{\mu}}(X_n;t) = \sum_{\bm{\sigma}\in \text{SSYT}(\bm{\lambda}/\bm{\mu})} x^{\bm{\sigma}} t^{\text{coinv}(\bm{\sigma})}
\end{equation}
where the sum is over $k$-tuples of semi-standard Young tableaux with shapes given by $\bm{\lambda}/\bm{\mu}$ and $\text{coinv}(\bm{\sigma})$ is the number of coinversion triples in the filling.

We will also need a `dual' version of the polynomials. Define $\mathcal{\tilde L}$ by
\begin{equation}
\mathcal{\tilde L}_{\bm{\lambda}/\bm{\mu}}(X_n;t) = \sum_{\bm{\sigma}\in \text{SSYT}'(\bm{\lambda}/\bm{\mu})} x^{\bm{\sigma}} t^{\text{inv}(\bm{\sigma})}
\end{equation}
where the sum is over $k$-tuples of tableaux that are weakly increasing up columns and strictly increasing across rows,
and $\text{inv}(\bm{\sigma})$ counts the number of triples of the form
\begin{center}
\ytableausetup{nosmalltableaux}
\ytableausetup{nobaseline}
\begin{ytableau}
\none & x & z \\
\none & \none[\diag]  \\
y
\end{ytableau} 
\end{center}
where 
\begin{enumerate}
    \item  $y$ is an entry of $\sigma^{(a)}$ and $x,z$ are entries in $\sigma^{(b)}$,  with $1\le a < b \le k$.
    \item $y$ and $z$ lie along the same content line.
    \item $x< y < z$.
\end{enumerate}
As for coinversions, $y$ must be in the diagram but we let $x=0$ and $z=\infty$ if they are not in the diagrams. Example \ref{exampleLLT} gives an example of both a coinversion and inversion triple.

\begin{remark}
Note that here our notation differs slightly from that of \cite{LLTaztec}. Our $\mathcal{\tilde L}$ are the same as their $\mathcal{L}^P$. Also, our $\tilde d(\bm{\lambda}, \bm{\mu})$ in the Cauchy identity, Prop. \ref{prop:Cauchy},  is the same as their $d^P(\bm{\lambda}, \bm{\mu})$.
\end{remark}

\noindent The polynomials $\mathcal{L}$ and $\mathcal{\tilde L}$ satisfy the following properties:
\begin{enumerate}
    \item They are symmetric in the $x_i$.
    \item When $k=1$ we have 
    \[
    \begin{aligned}
    \mathcal{L}_{\bm{\lambda}/\bm{\mu}}(X_n;t) = & s_{\lambda^{(1)}/\mu^{(1)}}(X_n)  \\
    \mathcal{\tilde L}_{\bm{\lambda}/\bm{\mu}}(X_n;t) = & s_{(\lambda^{(1)}/\mu^{(1)})'}(X_n).
    \end{aligned}
    \]
    \item When $t=1$ we have
    \[
    \begin{aligned}
    \mathcal{L}_{\bm{\lambda}/\bm{\mu}}(X_n;1) = & \prod_{i=1}^k s_{\lambda^{(i)}/\mu^{(i)}}(X_n) \\
    \mathcal{\tilde L}_{\bm{\lambda}/\bm{\mu}}(X_n;1) = & \prod_{i=1}^k s_{(\lambda^{(i)}/\mu^{(i)})'}(X_n).
    \end{aligned}
    \]
\end{enumerate}
\noindent In addition, we have the following propositions.
\begin{prop}[Branching rule, \cite{GKsuper}] \label{prop:branching}
The $\mathcal{L}$ and $\mathcal{\tilde L}$ satisfy the branching rules
\begin{equation}
\begin{aligned}
    \mathcal{L}_{\bm{\lambda}}(X_n;t) &= \sum_{\bm{\mu}}\mathcal{L}_{\bm{\lambda}/\bm{\mu}}(x_{m+1}\ldots,x_n;t) \mathcal{ L}_{\bm{\mu}}(x_1,\ldots,x_m;t) \\
    \mathcal{\tilde L}_{\bm{\lambda}}(X_n;t) &= \sum_{\bm{\mu}}\mathcal{\tilde L}_{\bm{\lambda}/\bm{\mu}}(x_{m+1}\ldots,x_n;t) \mathcal{\tilde L}_{\bm{\mu}}(x_1,\ldots,x_m;t)
\end{aligned}
\end{equation}
\end{prop}
\begin{prop}[Cauchy identity, \cite{GKsuper}] \label{prop:Cauchy}
The $\mathcal{L}$ and $\mathcal{\tilde L}$ satisfy the Cauchy identity
\begin{equation}
\begin{aligned}
    &\sum_{\bm{\lambda}} t^{\tilde d(\bm{\lambda},\bm{\nu})} \mathcal{L}_{\bm{\lambda}/\bm{\mu}}(Y_m;t)\mathcal{\tilde L}_{\bm{\lambda}/\bm{\nu}}(X_n;t) \\
    &= \left(\prod_{i,j}\prod_{l=0}^{k-1} (1+x_iy_jt^l) \right) \sum_{\bm{\lambda}} t^{\tilde d(\bm{\mu},\bm{\lambda})} \mathcal{L}_{\bm{\nu}/\bm{\lambda}}(Y_m;t)\mathcal{\tilde L}_{\bm{\bm{\mu}/\lambda}}(X_n;t).
\end{aligned}
\end{equation}
where $\tilde d(\bm{\lambda},\bm{\nu})$ has an explicit formula in terms of the parts of the partitions.

If $\bm{\nu}=\bm{\mu}=\bm{0}$, then this simplifies to
\[
\sum_{\bm{\lambda}} t^{\tilde d(\bm{\lambda},\mathbf{0})} \mathcal{L}_{\bm{\lambda}}(Y_m;t)\mathcal{\tilde L}_{\bm{\lambda}}(X_n;t) =\prod_{i,j}\prod_{l=0}^{k-1} (1+x_iy_jt^l).
\]
\end{prop}
See \cite{GKsuper} for proofs of these properties via integrable vertex models.
\begin{remark}
Note that when $k=1$ this reduces to the dual Cauchy identity for Schur polynomials.
\end{remark}

We will also need the following definitions: The size of a $k$-tuple of partitions is denoted $|\bm{\lambda}|$ and is given by the sum of the sizes of each partition in the tuple. We say that two $k$-tuples of partitions $\bm{\lambda}$ and $\bm{\mu}$ \emph{(co-)interlace} if for each $i=1,\ldots, k$ we have that $\lambda^{(i)}$ and $\mu^{(i)}$ (co-)interlace. We write $\bm{\lambda} \succeq \bm{\mu}$ or $\bm{\lambda} \succeq' \bm{\mu}$, for interlacing and co-interlacing, respectively.

To each $k$-tiling we associate partitions $\bm{\lambda}^{(i)} = (\bm{\lambda}^{(i, 1)}, \dots, \bm{\lambda}^{(i, k)})$, $\bm{\mu}^{(i)} = (\bm{\mu}^{(i, 1)}, \dots, \bm{\mu}^{(i, k)})$, $i= 1, \dots, N$, where $\bm{\lambda}^{(i, a)}$ denotes the Maya diagram along slice $2 i-1$ for color $a$, and  $\bm{\mu}^{(i,a)}$ the Maya diagram along slice $2 i - 2$. In order for a sequence of partitions to correspond to a valid tiling they must satisfy
\[
\bm{0}\preceq\bm{\lambda}^{(1)} \succeq'\bm{\mu}^{(2)}\preceq \ldots \succeq'\bm{\mu}^{(N)}\preceq \bm{\lambda}^{(N)} \succeq' \bm{0}.
\]
See Figure \ref{ex:3til-bijection} for an example.

\begin{prop}[\cite{LLTaztec}]\label{prop:LLTaztec}
In terms of the corresponding Maya diagrams 
\[
\bm{\lambda}^{(1)},\bm{\mu}^{(2)}, \ldots, \bm{\mu}^{(N)}, \bm{\lambda}^{(N)},
\]
the weight of a $k$-tiling of the Aztec diamond of rank $N$ can be written as
\begin{equation} 
\begin{aligned}
         &
        \mathcal{L}_{\bm{\lambda}^{(1)}}(c_1;t) 
        t^{\tilde d(\bm{\lambda}^{(1)},\bm{\mu}^{(2)})} \mathcal{\tilde L}_{\bm{\lambda}^{(1)}/\bm{\mu}^{(2)}}(b_N;t) \\
       &  \times \mathcal{L}_{\bm{\lambda}^{(2)}/\bm{\mu}^{(2)}}(c_2;t) 
         t^{\tilde d(\bm{\lambda}^{(2)},\bm{\mu}^{(3)})} \mathcal{\tilde L}_{\bm{\lambda}^{(2)}/\bm{\mu}^{(3)}}(b_{N-1};t) \\
        & \ldots \\     
        & \times \mathcal{L}_{\bm{\lambda}^{(N)}/\bm{\mu}^{(N)}}(c_N;t) 
        t^{\tilde d(\bm{\lambda}^{(N)},\mathbf{0})}\mathcal{\tilde L}_{\bm{\lambda}^{(N)}}(b_1;t)
    \end{aligned}
\end{equation}
\end{prop}

Note that when have only a single variable, the LLT polynomials can be written
\begin{equation}\label{eqn:monomials}
\begin{aligned}
    \mathcal{L}_{\bm{\lambda}/\bm{\mu}}(x;t) &\;= x^{|\bm{\lambda}/\bm{\mu}|} t^{\text{coinv}(\bm{\lambda}/\bm{\mu})}\\
    t^{\tilde d(\bm{\lambda},\bm{\mu})} \mathcal{\tilde L}_{\bm{\lambda}/\bm{\mu}}(y;t) &\;= y^{|\bm{\lambda}/\bm{\mu}|} t^{\text{inv}(\bm{\lambda}/\bm{\mu}) + \tilde d(\bm{\lambda},\bm{\mu})}.
\end{aligned}
\end{equation}
In particular, they are monomials. It will be useful in the proof of Prop. \ref{prop:probs} to know precisely how each partition contributes to the powers of $t$ in each of these monomials.


\begin{lem}\label{lem:tpow}
Fix two colors $b>a$ and consider the $i$-th row of color $a$ in $\bm{\lambda}$. 
\begin{enumerate}
\item There is an interaction between this row and row $j$ of color $b$ whenever
\[
\min(\bm{\lambda}_j^{(b)}-j,\bm{\lambda}_i^{(a)}-i) - \max(\bm{\mu}_j^{(b)}-j,\bm{\mu}_i^{(a)}-i) \ge 0
\]
in which case it contributes 
\begin{equation}\label{eq:tHor}
\begin{aligned}
&\min(\bm{\lambda}_j^{(b)}-j,\bm{\lambda}_i^{(a)}-i) - \max(\bm{\mu}_j^{(b)}-j,\bm{\mu}_i^{(a)}-i) \\
& \qquad \qquad + 
1(\bm{\lambda}_j^{(b)}-j< \bm{\lambda}_i^{(a)}-i )
\end{aligned}
\end{equation}
to the total power of $t$ in $\mathcal{L}_{\bm{\lambda}/\bm{\mu}}(x;t)$.

\item There is an interaction between this row and row $j$ of color $b$ whenever
\begin{equation}\label{eq:tVert}
\bm{\lambda}_j^{(b)}-j = \bm{\mu}_j^{(b)}-j+1 = \bm{\mu}_i^{(a)}-i = \bm{\lambda}_i^{(a)}-i 
\end{equation}
in which case its contributes a single power of $t$ to the total power of $t$ in $t^{\tilde d(\bm{\lambda},\bm{\mu})} \mathcal{\tilde L}_{\bm{\lambda}/\bm{\mu}}(y;t)$.
\end{enumerate}
\end{lem}
\begin{proof}
This follows from the discussion in Section 4.4 of \cite{LLTaztec}.
\end{proof}

As with the single tiling of the Aztec diamond, we can associate to each $k$-tiling a particle configuration. In this case, we have particles of $k$ colors, one for each tiling. For each color, the particles must satisfy the required interlacing conditions (\ref{eq:particleInterlace}) and (\ref{eq:particleBounds}). Again, see Figure \ref{ex:3til-bijection} for an example.

\section{Markov chains on colored interlacing arrays}\label{sec:Markov}

In this section we define a Markov chain on colored interlacing particle arrays which will be equivalent to the shuffling algorithm under the identification between particle arrays and $k$-tilings described in Section \ref{sec:background}. The essential ingredients are the Cauchy identities and branching rule for LLT polynomials. Using these, we apply a construction whose original form was introduced in \cite{DiaconisFill1990}, and which was further developed in the case of random tilings in \cite{Ferrari2008, borodin2015random}. We first elaborate on the (well-known) construction in the one color Schur case, and then describe the corresponding generalization to the LLT case.

\subsection{Markov Chains on Schur Processes}\label{sec:MarkovSchur}

Recall that under the bijection to interlacing partitions, the probability measure on domino tilings becomes
$$\frac{1}{Z} s_{\lambda^{(1)}}(c_1) s_{(\lambda^{(1)}/\mu^{(2)})'}(b_N) s_{\lambda^{(2)}/\mu^{(2)}}(c_2) \ldots s_{\lambda^{(N)}/\mu^{(N)}}(c_N) s_{(\lambda^{(N)})'}(b_1) \;\;.$$
As mentioned in Remark \ref{rmk:schur}, this is an example of a Schur process. Now we introduce transition kernels which are building blocks for Markov chains that map Schur processes to other Schur processes. For generalizations of this construction and more details, see \cite{BorodinGorinSPB12, borodin2015random} and references therein.

Suppose that $b, c_1,c_2,\dots$ are positive real numbers, and denote the tuple~$c_{[1,k]} = (c_1,\dots, c_k)$. Define transition probabilities by
\[
\begin{aligned}
p^{\uparrow}_{\lambda\to\mu}(c_{[1,k]}|b) =& \frac{1}{Z(c_{[1,k]}|b)}\frac{s_{\mu}(c_{[1,k]})}{s_{\lambda}(c_{[1,k]})} s_{\mu'/\lambda'}(b) \\
p^{\downarrow}_{\lambda\to\mu}(c_{[1,k-1]}|c_k) =& \frac{s_{\mu}(c_{[1,k-1]})}{s_{\lambda}(c_{[1,k]})} s_{\lambda/\mu}(c_k) 
\end{aligned}
\]
where 
$$Z(c_{[1,k]}|b) \coloneqq \prod_{i=1}^k (1 + b c_i)$$ 
which can be computed by the Cauchy identity.

We have
\[
\begin{aligned}
\sum_{\mu} p^{\uparrow}_{\lambda\to\mu}(c_{[1,k]}|b) = & 1 \\
\sum_{\mu} p^{\downarrow}_{\lambda\to\mu}(c_{[1,k-1]}|c_k) = & 1 
\end{aligned}
\]
where the second equality follows from the branching rule. Note that $p^{\uparrow}_{\lambda\to\mu}$ is nonzero only if $\lambda\subset \mu$ and we view  this as the partition growing from $\lambda$ to $\mu$. Similarly, $p^{\downarrow}_{\lambda\to\mu}$ is nonzero only if $\mu \subset \lambda$ and we view this as the partition shrinking.

A key property of the transition kernels $p^{\uparrow}, p^{\downarrow}$ is their commutation:
\begin{equation}
\label{eqn:pcom}
\begin{aligned}
&\sum_{\mu} p^{\uparrow}_{\lambda\to\mu}(c_{[1,k]}|b) p^{\downarrow}_{\mu\to\nu}(c_{[1,k-1]}|c_{k}) \\
& \qquad \qquad \qquad =
\sum_{\mu} p^{\downarrow}_{\lambda\to\mu}(c_{[1,k-1]}|c_k) p^{\uparrow}_{\mu\to\nu}(c_{[1,k-1]}|b)\;\;.
\end{aligned}
\end{equation}
This property also follows from the skew-Cauchy and branching identities.

Using these, we will define transition probabilities out of which our Markov chain will be built. Now, let $c_1,\dots, c_N, b_1,\dots, b_N$ be positive real numbers. Define $c_{[i, j]} \coloneqq (c_i,\dots, c_j)$. Given partitions $\lambda$ and $\nu$, define the transition probabilities to a new partition $\tilde{\lambda}$ by 
\[
\tilde{P}_k(\tilde{\lambda} | \lambda,\nu) = \frac{ p^{\uparrow}_{\lambda\to \tilde{\lambda}}(c_{[1,k]} |b_{N-k+2}) p^{\downarrow}_{\tilde{\lambda} \to\nu}(c_{[1,k-1]}|c_{k})}{ \sum_{\mu} p^{\uparrow}_{\lambda\to\mu}(c_{[1,k]}|b_{N-k+2}) p^{\downarrow}_{\mu\to\nu}(c_{[1,k-1]}|c_{k})}.
\]
With this we can describe a Markov chain acting on the set of sequences of partitions $$\emptyset \preceq \lambda^{(1)} \succeq' \mu^{(2)} \preceq \lambda^{(2)} \succeq' \cdots \preceq \lambda^{(N)} \succeq' \mu^{(N+1)} \preceq \lambda^{(N+1)} \succeq'  \emptyset \;\;.$$
\begin{definition}[Schur parallel update~\cite{borodin2015random}]
Given a sequence of partitions interlaced as above, the transition probabilities for updates $\mu^{(k)} \rightarrow \tilde{\mu}^{(k)}$ and $\lambda^{(k)} \rightarrow \tilde{\lambda}^{(k)}$ are defined as follows:
\begin{enumerate}
    \item For each $k  =2, \dots, N+1$, set $\tilde{\mu}^{(k)} = \lambda^{(k-1)}$.
    \item For each $k  =1, \dots, N+1$, update $\lambda^{(k)} \rightarrow \tilde{\lambda}^{(k)}$ with transition probabilities 
    $$\tilde{P}_k(\tilde{\lambda}^{(k)} | \lambda^{(k)},\tilde{\mu}^{(k)})\;\;.$$
\end{enumerate}
\end{definition}

Now we note that we may write the Schur process corresponding to a rank-$N$ domino tiling of the Aztec diamond as 
\begin{equation}\label{eq:SchurProcAlt}
\begin{aligned}
 &\frac{1}{Z} s_{\lambda^{(1)}}(c_1) s_{(\lambda^{(1)}/\mu^{(2)})'}(b_N) s_{\lambda^{(2)}/\mu^{(2)}}(c_2)\ldots s_{\lambda^{(N)}/\mu^{(N)}}(c_N)s_{(\lambda^{(N)})'}(b_1) =  \\
&p^{\uparrow}_{\emptyset \rightarrow \lambda^{(N)}}(c_{[1,N]} | b_1) p^{\downarrow}_{\lambda^{(N)} \rightarrow \mu^{(N)}}(c_{[1,N-1]}|c_N)  p^{\uparrow}_{\mu^{(N)} \rightarrow \lambda^{(N-1)}}(c_{[1,N-1]} | b_{2}) \; \times \cdots \\
&\qquad \qquad \cdots \times  p^{\downarrow}_{\lambda^{(2)} \rightarrow \mu^{(2)}}(c_{1}|c_2)  p^{\uparrow}_{\mu^{(2)} \rightarrow \lambda^{(1)}}(c_{1} | b_{N}) 
\end{aligned}
\end{equation}
This form is useful in the following proposition. 

\begin{prop}[\cite{borodin2015random}]\label{prop:Schurcase}
Suppose $\lambda^{(1)},\mu^{(2)}, \lambda^{(2)}, \ldots,\mu^{(N)},\lambda^{(N)}$ are sampled from the Schur process
\begin{align*}
 \frac{1}{Z} s_{\lambda^{(1)}}(c_1) s_{(\lambda^{(1)}/\mu^{(2)})'}(b_N) s_{\lambda^{(2)}/\mu^{(2)}}(c_2)\ldots s_{\lambda^{(N)}/\mu^{(N)}}(c_N)s_{(\lambda^{(N)})'}(b_1)
\end{align*}
and suppose $\mu^{(N+1)} = \lambda^{(N+1)} = \emptyset$ deterministically.

Then after the Schur parallel update, the updated partitions
$$\tilde{\lambda}^{(1)},\tilde{\mu}^{(2)}, \tilde{\lambda}^{(2)}, \ldots,\tilde{\mu}^{(N)},\tilde{\lambda}^{(N)},\tilde{\mu}^{(N+1)},\tilde{\lambda}^{(N+1)}$$
are distributed according to the Schur process 
\begin{align*}
&\frac{1}{Z} s_{\tilde{\lambda}^{(1)}}(c_1) s_{(\tilde{\lambda}^{(1)}/\tilde{\mu}^{(2)})'}(b_{N+1}) \times\cdots \\
&\qquad\qquad\cdots \times s_{(\tilde{\lambda}^{(N)}/\tilde{\mu}^{(N+1)})'}(b_2) s_{\tilde{\lambda}^{(N+1)}/\tilde{\mu}^{(N+1)}}(c_{N+1}) s_{(\tilde{\lambda}^{(N+1)})'}(b_1) \;\;.
\end{align*}
\end{prop}

\begin{remark}
    This proposition is a special case of well-known facts (again see e.g. \cite{BorodinGorinSPB12, borodin2015random}), but we include a proof for clarity of exposition, and in particular because this proof generalizes immediately to the multi-colored case.
\end{remark}

\begin{proof}
Using Eqn. (\ref{eq:SchurProcAlt}), we must show 
\begin{align*}
 &\sum_{\{\mu^{(i)}\}} p^{\uparrow}_{\emptyset \rightarrow \lambda^{(N)}}(c_{[1,N]} | b_1) p^{\downarrow}_{\lambda^{(N)} \rightarrow \mu^{(N)}}(c_1,\dots, c_{N-1}|c_N) \times \cdots \\
 & \qquad \cdots \times  p^{\downarrow}_{\lambda^{(2)} \rightarrow \mu^{(2)}}(c_{1}|c_2)  p^{\uparrow}_{\mu^{(2)} \rightarrow \lambda^{(1)}}(c_{1} | b_{N})  \prod_{n=1}^{N+1} \tilde{P}_n(\tilde{\lambda}^{(n)} | \lambda^{(n)},\tilde{\mu}^{(n)}) =\\
 &
 p^{\uparrow}_{\emptyset \rightarrow \tilde{\lambda}^{(N+1)}}(c_{[1,N+1]} | b_1) p^{\downarrow}_{\tilde{\lambda}^{(N+1)} \rightarrow \tilde{\mu}^{(N+1)}}(c_{[1,N]}|c_{N+1})  p^{\uparrow}_{\tilde{\mu}^{(N+1)} \rightarrow \tilde{\lambda}^{(N)}}(c_{[1,N]} | b_{2}) \; \times \cdots \\
& \qquad \cdots \times  p^{\downarrow}_{\tilde{\lambda}^{(2)} \rightarrow \tilde{\mu}^{(2)}}(c_{1}|c_2)  p^{\uparrow}_{\tilde{\mu}^{(2)} \rightarrow \tilde{\lambda}^{(1)}}(c_{1} | b_{N+1}) \;\;.
\end{align*}

Indeed, first note that since $\tilde{\mu}^{(i)}=\lambda^{(i-1)}$ deterministically, we are not summing over the $\lambda^{(i)}$. In particular, the numerators of the transition probability  $\prod_{n=1}^{N+1} \tilde{P}_n(\tilde{\lambda}^{(n)} | \lambda^{(n)},\tilde{\mu}^{(n)})$
are 
$$\prod_{n=1}^{N+1} p^{\uparrow}_{\lambda^{(n)} \to \tilde{\lambda}^{(n)}}(c_{[1,n]} |b_{N-n+2}) p^{\downarrow}_{\tilde{\lambda}^{(n)} \to \tilde{\mu}^{(n)}}(c_{[1,n-1]}|c_{n})$$
and thus can be pulled out of the sum. Moreover, this product is exactly the Schur process we desire. We are left to show that the denominator of the transition probabilities cancels with what remains in the numerator.

Consider the sum over $\mu^{(2)}$. The relevant term is
$$\sum_{\mu^{(2)}} p^{\downarrow}_{\lambda^{(2)} \rightarrow \mu^{(2)}}(c_{1}|c_2)  p^{\uparrow}_{\mu^{(2)} \rightarrow \lambda^{(1)}}(c_{1} | b_{N}). $$
Using the commutation relation (\ref{eqn:pcom}) we have
$$
\sum_{\mu^{(2)}} p^{\downarrow}_{\lambda^{(2)} \rightarrow \mu^{(2)}}(c_{1}|c_2)  p^{\uparrow}_{\mu^{(2)} \rightarrow \lambda^{(1)}}(c_{1} | b_{N}) = 
\sum_{\kappa} p^{\uparrow}_{\lambda^{(2)} \to\kappa}(c_{[1,2]}|b_{N}) p^{\downarrow}_{\kappa\to\lambda^{(1)}}(c_{1}|c_{2}).
$$
This is precisely the denominator of $\tilde{P}_2$ and we see that the terms cancel. 

More generally, for $k = 2, \dots, N+1$, one can see that the sum over $\mu^{(k)}$ cancels with the denominator of $\tilde{P}_k$ through the commutation relation
\begin{align*}
   & \sum_{\mu^{(k)}} p^{\downarrow}_{\lambda^{(k)} \rightarrow \mu^{(k)}}(c_1,\dots, c_{k-1} | c_k) p^{\uparrow}_{\mu^{(k)} \rightarrow \lambda^{(k-1)}}(c_1,\dots, c_{k-1} | b_{N-k+2}) \\
    &= \sum_{\kappa} p^{\uparrow}_{\lambda^{(k)} \rightarrow \kappa}(c_1,\dots, c_{k-1},c_k | b_{N-k+2}) p^{\downarrow}_{\kappa \rightarrow \lambda^{(k-1)}} (c_1,\dots, c_{k-1} | c_k) \;\;.
\end{align*}

\end{proof}

 Computing the transition probabilities in terms of particle positions 
 \[
 x^{(n)}, \tilde{x}^{(n)}, y^{(n)}, \tilde{y}^{(n)} \qquad \text{corresponding to} \qquad  \lambda^{(n)},\tilde{\lambda}^{(n)}, \mu^{(n)}, \tilde{\mu}^{(n)},
 \]
 respectively, one observes that we can define the Schur parallel update in terms of the particles as follows:
\begin{enumerate}
    \item For each $n =2, \dots, N+1$, set $\tilde{y}^{(n)} = x^{(n-1)}$.
    \item For each $n  =1, \dots, N+1$, update $x^{(n)} \rightarrow \tilde{x}^{(n)}$ according to the rules: 
    \begin{itemize}
    \item If either $\tilde{x}_i^{(n)} = x_i^{(n)}$ or $\tilde{x}_i^{(n)} = x_i^{(n)}+1$ is forced in order to preserve interlacing with $\tilde{y}^{(n)}$, then transition accordingly with probability $1$. 
    \item Otherwise, $$ 
   \tilde{x}^{(n)}_i = 
   \begin{cases}
   x^{(n)}_i + 1 & 
    \text{ with probability } \frac{c_k b_{N-k+2}}{1 + c_k b_{N-k+2}} \\
   x^{(n)}_i  & \text{ otherwise }
   \end{cases}.
   $$
   \end{itemize}
   These transitions are independent for~$1 \leq i \leq n \leq N+1$.
\end{enumerate}

\begin{remark}
That the Markov chain has this form follows from the discussion of the Aztec diamond shuffling algorithm in \cite{borodin2015random}. It will also follow from our generalization to multiple colors in the next section.
\end{remark}

Now we have the following proposition, which follows directly from Proposition~\ref{prop:Schurcase}.

\begin{prop}
Let $x^{(1)}, y^{(1)}, \dots, x^{(N-1)}, y^{(N)}, x^{(N)}, x^{(N+1)}, y^{(N+1)}$ be a random array of interlaced particles such that $x^{(1)}, y^{(1)}, \dots, x^{(N-1)}, y^{(N)}, x^{(N)}$ are sampled according to the Schur process
$$\frac{1}{Z} s_{\lambda^{(1)}}(c_1) s_{(\lambda^{(1)}/\mu^{(2)})'}(b_N) s_{\lambda^{(2)}/\mu^{(2)}}(c_2)\ldots s_{\lambda^{(N)}/\mu^{(N)}}(c_N)s_{(\lambda^{(N)})'}(b_1)$$
and  
\begin{align*}
x^{(N+1)} &= (-N-\frac{1}{2}, -N+\frac{1}{2}, \dots, -\frac{1}{2}) \\
y^{(N+1)} &= (-N+\frac{1}{2}, -N+\frac{3}{2}, \dots, -\frac{1}{2})
\end{align*}
deterministically. Then after the Schur parallel update, the particles $\tilde{x}^{(n)}$, $\tilde{y}^{(n)}$, $n= 1,\dots, N+1$, are distributed according to 
\[
\begin{aligned}
&\frac{1}{Z} s_{\tilde{\lambda}^{(1)}}(c_1) s_{(\tilde{\lambda}^{(1)}/\tilde{\mu}^{(2)})'}(b_{N+1}) \times \ldots \\ 
&\qquad \qquad \ldots \times s_{(\tilde{\lambda}^{(N)}/\tilde{\mu}^{(N+1)})'}(b_2) s_{\tilde{\lambda}^{(N+1)}/\tilde{\mu}^{(N+1)}}(c_{N+1}) s_{(\tilde{\lambda}^{(N+1)})'}(b_1) \;\;.
\end{aligned}
\]
\end{prop}

\subsection{Markov Chains on LLT Processes}
Now we generalize the Markov chain above to a Markov chain on $k$-tuples of interlacing partitions. While the definition in terms of interlacing partitions will follow directly from machinery of LLT polynomials, the interpretation as particle dynamics will require careful computation.

Let $c_1,\dots, c_N, b_1,\dots, b_N \in \mathbb{R}_{> 0}$.
Define the \emph{LLT process} to be a probability measure on arrays of tuples of partitions given by
\begin{multline}
P(\bm{\lambda}^{(1)}, \bm{\mu}^{(2)}, \bm{\lambda}^{(2)}, \dots, \bm{\mu}^{(N)}, \bm{\lambda}^{(N)}) \label{eqn:LLTprocess}\\
 = \frac{1}{Z} \mathcal{L}_{\bm{\lambda}^{(1)}}(c_1;t) 
        t^{\tilde d(\bm{\lambda}^{(1)},\bm{\mu}^{(2)})} \mathcal{\tilde L}_{\bm{\lambda}^{(1)}/\bm{\mu}^{(2)}}(b_N;t)   \\
         \qquad \qquad \times \mathcal{L}_{\bm{\lambda}^{(2)}/\bm{\mu}^{(2)}}(c_2;t)
         t^{\tilde d(\bm{\lambda}^{(2)},\bm{\mu}^{(3)})} \mathcal{\tilde L}_{\bm{\lambda}^{(2)}/\bm{\mu}^{(3)}}(b_{N-1};t)  \\
         \times \cdots \times    
        \mathcal{L}_{\bm{\lambda}^{(N)}/\bm{\mu}^{(N)}}(c_N;t) 
        t^{\tilde d(\bm{\lambda}^{(N)},\mathbf{0})}\mathcal{\tilde L}_{\bm{\lambda}^{(N)}}(b_1;t)  \;\;. 
\end{multline}
Recall that this probability measure describes random $k$-tilings (Proposition \ref{prop:LLTaztec}).

The transition kernels from which we will build the Markov chain are, for $c = (c_1,\dots, c_l)$ and $b, c' \in \mathbb{R}$,
\begin{align*}
    p^{\uparrow}_{\bm{\lambda}\to\bm{\mu}}(c|b) =& \frac{t^{\tilde{d}(\bm{\mu},\bm{\lambda})}}{Z(c|b)}\frac{\mathcal{L}_{\bm{\mu}}(c;t)}{\mathcal{L}_{\bm{\lambda}}(c;t)} \tilde{\mathcal{L}}_{\bm{\mu}/\bm{\lambda}}(b;t) \\
p^{\downarrow}_{\bm{\lambda}\to\bm{\mu}}(c|c') =& \frac{\mathcal{L}_{\bm{\mu}}(c';t)}{\mathcal{L}_{\bm{\lambda}}(c,c';t)} \mathcal{L}_{\bm{\lambda}/\bm{\mu}}(c';t) \;\;. 
\end{align*}

These transition kernels satisfy a similar commutation relation to Equation \eqref{eqn:pcom}:
\begin{prop}
\begin{equation} \label{eqn:LLTcom}
\begin{aligned}
&\sum_{\bm{\mu}} p^{\uparrow}_{\bm{\lambda}\to\bm{\mu}}(c_{[1,l]}|b) p^{\downarrow}_{\bm{\mu}\to \bm{\nu}}(c_{[1,l-1]}|c_l) \\
& \qquad \qquad = \sum_{\bm{\mu}} p^{\downarrow}_{\bm{\lambda}\to\bm{\mu}}(c_{[1,l-1]}|c_l) p^{\uparrow}_{\bm{\mu}\to\bm{\nu}}(c_{[1,l-1]}|b)\;\;.
\end{aligned}
\end{equation}
\end{prop}

In the exact same way as in the Schur case, we can write the LLT process as 
\begin{equation}\label{eqn:LLTProcAlt}
\begin{aligned}
    &P(\bm{\lambda}^{(1)}, \bm{\mu}^{(2)}, \bm{\lambda}^{(2)}, \dots, \bm{\mu}^{(N)}, \bm{\lambda}^{(N)}) =\\
    & p^{\uparrow}_{\emptyset \rightarrow \bm{\lambda}^{(N)}}(c_{[1,N]} | b_1) p^{\downarrow}_{\bm{\lambda}^{(N)} \rightarrow \bm{\mu}^{(N)}}(c_1,\dots, c_{N-1}|c_N)  p^{\uparrow}_{\bm{\mu}^{(N)} \rightarrow \bm{\lambda}^{(N-1)}}(c_{[1,N-1]} | b_{2}) \; \times \cdots \\
&\qquad \qquad \cdots \times  p^{\downarrow}_{\bm{\lambda}^{(2)} \rightarrow \bm{\mu}^{(2)}}(c_{1}|c_2)  p^{\uparrow}_{\bm{\mu}^{(2)} \rightarrow \bm{\lambda}^{(1)}}(c_{1} | b_{N}) \;\;.
\end{aligned}
\end{equation}

We define the following update step that we will use for transitioning from rank $N$ to rank $(N + 1)$. 
\begin{definition}[LLT parallel update] \label{def:LLTprocupdate}
Suppose we are given a k-tuple of sequences of interlaced partitions
\[
\bm{0}\preceq\bm{\lambda}^{(1)} \succeq'\bm{\mu}^{(2)}\preceq \ldots \succeq'\bm{\mu}^{(N)}\preceq \bm{\lambda}^{(N)}  \succeq' \bm{\mu}^{(N+1)} \preceq \bm{\lambda}^{(N+1)}  \succeq' \bm{0} \;\;.
\]
The transition probabilities for updates $\bm{\mu}^{(k)} \rightarrow \tilde{\bm{\mu}}^{(k)}$ and $\bm{\lambda}^{(k)} \rightarrow \tilde{\bm{\lambda}}^{(k)}$ are defined as follows:

\begin{enumerate}
    \item For each $n  =2, \dots, N+1$, set $\tilde{\bm{\mu}}^{(n)} = \bm{\lambda}^{(n-1)}$.
    \item For each $n =1, \dots, N+1$, update $\bm{\lambda}^{(n)} \rightarrow \tilde{\bm{\lambda}}^{(n)}$ with transition probabilities 
    $$\tilde{P}_n(\tilde{\bm{\lambda}} | \bm{\lambda},\tilde{\bm{\mu}}) \coloneqq \frac{ p^{\uparrow}_{\bm{\lambda}\to \tilde{\bm{\lambda}}}(c_{[1,n]} |b_{N-n+2}) p^{\downarrow}_{\tilde{\bm{\lambda}} \to \tilde{\bm{\mu}}}(c_{[1,n-1]}|c_{n})}{ \sum_{\kappa} p^{\uparrow}_{\bm{\lambda}\to \bm{\kappa}}(c_{[1,n]}|b_{N-n+2}) p^{\downarrow}_{\bm{\kappa} \to \tilde{\bm{\mu}}}(c_{[1,n-1]}|c_{n})}\;\;.$$
\end{enumerate}
\end{definition}

As in the Schur case we have 

\begin{prop}\label{prop:LLTcase}
Suppose $\bm{\lambda}^{(1)},\bm{\mu}^{(2)}, \bm{\lambda}^{(2)}, \ldots,\bm{\mu}^{(N)},\bm{\lambda}^{(N)}$ are sampled from the LLT process
\begin{align*}
&\frac{1}{Z} \mathcal{L}_{\bm{\lambda}^{(1)}}(c_1;t) 
        t^{\tilde d(\bm{\lambda}^{(1)},\bm{\mu}^{(2)})} \mathcal{\tilde L}_{\bm{\lambda}^{(1)}/\bm{\mu}^{(2)}}(b_N;t)  \\
        & \times \mathcal{L}_{\bm{\lambda}^{(2)}/\bm{\mu}^{(2)}}(c_2;t) 
         t^{\tilde d(\bm{\lambda}^{(2)},\bm{\mu}^{(3)})} \mathcal{\tilde L}_{\bm{\lambda}^{(2)}/\bm{\mu}^{(3)}}(b_{N-1};t) \\
        & \times \cdots \times    
        \mathcal{L}_{\bm{\lambda}^{(N)}/\bm{\mu}^{(N)}}(c_N;t) 
        t^{\tilde d(\bm{\lambda}^{(N)},\mathbf{0})}\mathcal{\tilde L}_{\bm{\lambda}^{(N)}}(b_1;t)  
\end{align*}
and suppose $\bm{\mu}^{(N+1)} = \bm{\lambda}^{(N+1)} = \mathbf{0}$ deterministically.

Then after the LLT parallel update the new partitions
$$\tilde{\bm{\lambda}}^{(1)},\tilde{\bm{\mu}}^{(2)},\tilde{\bm{\lambda}}^{(2)}, \ldots,\tilde{\bm{\mu}}^{(N)},\tilde{\bm{\lambda}}^{(N)},\tilde{\bm{\mu}}^{(N+1)},\tilde{\bm{\lambda}}^{(N+1)}$$
are distributed according to the LLT process 
\begin{align*}
&\frac{1}{Z} \mathcal{L}_{\tilde{\bm{\lambda}}^{(1)}}(c_1;t) 
        t^{\tilde d(\tilde{\bm{\lambda}}^{(1)},\tilde{\bm{\mu}}^{(2)})} \mathcal{\tilde L}_{\tilde{\bm{\lambda}}^{(1)}/\tilde{\bm{\mu}}^{(2)}}(b_{N+1};t)  \\
        & \times \mathcal{L}_{\tilde{\bm{\lambda}}^{(2)}/\tilde{\bm{\mu}}^{(2)}}(c_2;t) 
         t^{\tilde d(\tilde{\bm{\lambda}}^{(2)},\tilde{\bm{\mu}}^{(3)})} \mathcal{\tilde L}_{\tilde{\bm{\lambda}}^{(2)}/\tilde{\bm{\mu}}^{(3)}}(b_{N};t) \\
        & \times \cdots \times    
        \mathcal{L}_{\tilde{\bm{\lambda}}^{(N+1)}/\tilde{\bm{\mu}}^{(N+1)}}(c_{N+1};t) 
        t^{\tilde d(\tilde{\bm{\lambda}}^{(N+1)},\mathbf{0})}\mathcal{\tilde L}_{\tilde{\bm{\lambda}}^{(N+1)}}(b_1;t) \;\;.
\end{align*}
\end{prop}

\begin{proof}
The calculation is exactly as in the proof of Proposition \ref{prop:Schurcase} by writing the LLT process as in Eqn. (\ref{eqn:LLTProcAlt}) and using the commutation relations (\ref{eqn:LLTcom}).
\end{proof}

Through our bijection between interlacing tuples of partitions and multi-colored interlacing particle arrays we can view the Markov chain on LLT processes as dynamics on our particles. In the following proposition we compute the transition probabilities in this chain in terms of the corresponding interlacing particle arrays. See Figure \ref{fig:jumpprobs_ex} for an example of the particle jump probabilities along one row and see Figure \ref{fig:example} for an example of the full particle update.

\begin{prop}\label{prop:probs}
Let 
\[
\bm{\lambda}^{(n)} = (\bm{\lambda}^{(n,1)}, \dots, \bm{\lambda}^{(n,k)}) \qquad \text{and} \qquad \bm{\mu}^{(n)} = (\bm{\mu}^{(n,1)}, \dots, \bm{\mu}^{(n,k)}), \qquad n=1,\dots, N
\]
 be a random sequence of tuples of partitions sampled from the LLT process~\eqref{eqn:LLTprocess}, and choose 
$
\bm{\lambda}^{(N+1)} = \bm{\mu}^{(N+1)} = \bm{0}
$
deterministically. These correspond to particles at positions
\[
\mathbf{x}^{(n)} = (\bm{x}^{(n,1)}, \dots, \bm{x}^{(n,k)}) \qquad \text{and} \qquad \mathbf{y}^{(n)} = (\mathbf{y}^{(n, 1)},\dots, \mathbf{y}^{(n, k)}), \qquad n=1,\dots, N+1.
\]
Let 
\[
\tilde{\bm{\lambda}}^{(n)} = (\tilde{\bm{\lambda}}^{(n,1)}, \dots, \tilde{\bm{\lambda}}^{(n,k)}) \qquad \text{and} \qquad \tilde{\bm{\mu}}^{(n)} = (\tilde{\bm{\mu}}^{(n,1)}, \dots, \tilde{\bm{\mu}}^{(n,k)}), \qquad n=1,\dots, N+1
\]
be the sequence of tuples of partitions after the update, corresponding to particle positions
\[
\tilde{\mathbf{x}}^{(n)} = (\tilde{\bm{x}}^{(n,1)}, \dots, \tilde{\bm{x}}^{(n,k)}) \qquad \text{and} \qquad \tilde{\mathbf{y}}^{(n)} = (\tilde{\mathbf{y}}^{(n, 1)},\dots, \tilde{\mathbf{y}}^{(n, k)}), \qquad n=1,\dots, N+1.
\]
Then the dynamics on the multi-colored interlacing particle arrays defined below is equivalent to the LLT parallel update defined in Definition~\ref{def:LLTprocupdate} above.

For each $n = 1, \dots, N+1$, update the particles as follows:

\begin{enumerate}
    \item If $n \geq 2$, set $\tilde{\mathbf{y}}^{(n)} = \mathbf{x}^{(n-1)}$ deterministically.
    \item Given $ \bm{x}^{(n)}$ and $\tilde{\mathbf{y}}^{(n)}$, $\tilde{\mathbf{x}}^{(n)}$ 
can be sampled according to the following rules: \\
For each $l = 1,\dots, k$
\begin{itemize}
    \item If 
    $\bm{x}^{(n, l)}_i = \tilde{\mathbf{y}}^{(n, l)}_i-1$
    then
   $$ \tilde{\bm{x}}^{(n, l)}_i = \bm{x}^{(n, l)}_i+1$$
   with probability $1$.
   \item If 
    $\bm{x}^{(n, l)}_i = \tilde{\mathbf{y}}^{(n, l)}_{i-1}-1$
    then
   $$ \tilde{\bm{x}}^{(n, l)}_i = \bm{x}^{(n, l)}_i$$
   with probability $1$.
   \item Otherwise, $$ 
   \tilde{\bm{x}}^{(n, l)}_i = 
   \begin{cases}
   \bm{x}^{(n, l)}_i + 1 & 
    \text{ with probability } \frac{c_n b_{N-n+2}t^{\#_i(n,l)}}{1 + c_n b_{N-n+2}t^{\#_i(n,l)}} \\
   \bm{x}^{(n, l)}_i  & \text{ otherwise }
   \end{cases}
   $$
\end{itemize}
where 
\begin{equation}\label{eqn:tpow}
\resizebox{0.8\textwidth}{!}{$
\begin{aligned}
\#_i(n,l)=& |\{m>l|\text{there exists $j$ s.t. }\tilde{\bm{y}}_j^{(n, m)}\le \bm{x}_i^{(n, l)} \le \bm{x}_j^{(n, m)} \} | \\
& +| \{m<l|\text{there exists $j$ s.t. }\tilde{\bm{y}}_j^{(n, m)}\le \bm{x}_i^{(n, l)}+1 \le \bm{x}_j^{(n, m)}\} |.
\end{aligned}
$}
\end{equation}
These transitions are independent for each~$1 \leq i \leq n \leq N+1$, $1 \leq l \leq k$. In~\eqref{eqn:tpow}, we allow~$j = 1,\dots, n$, and we use the convention that~$\bm{\tilde{y}}_n^{(n, m)} = -n + 1/2$.
\end{enumerate}
\end{prop}
\begin{proof}
First let us recall that for the Markov chain on LLT processes, at level $n$ the transition probabilities are proportional to the product of the transition kernels
\begin{align*}
    p^{\uparrow}_{\bm{\lambda}\to\bm{\mu}}(c_{[1,n]}|b) =& \frac{t^{\tilde{d}(\bm{\mu},\bm{\lambda})}}{Z(c_{[1,n]}|b)}\frac{\mathcal{L}_{\bm{\mu}}(c_{[1,n]};t)}{\mathcal{L}_{\bm{\lambda}}(c_{[1,n]};t)} \tilde{\mathcal{L}}_{\bm{\mu}/\bm{\lambda}}(b;t) \\
p^{\downarrow}_{\bm{\lambda}\to\bm{\mu}}(c_{[1,n-1]}|c) =& \frac{\mathcal{L}_{\bm{\mu}}(c_{[1,n-1]};t)}{\mathcal{L}_{\bm{\lambda}}(c_{[1,n]};t)} \mathcal{L}_{\bm{\lambda}/\bm{\mu}}(c;t) \;\;. 
\end{align*}

To simplify notation, in what follows we let $c = c_n, b = b_{N-n+2}$, and $\tilde{\bm{\lambda}} =\bm{\tilde{\lambda}}^{(n)}, \bm{\lambda} = \bm{\lambda}^{(n)},\tilde{\bm{\mu}} = \tilde{\bm{\mu}}^{(n)}$. We see that the probability of a particular $\tilde{\bm{\lambda}} $ under $\tilde{P}_n(\tilde{\bm{\lambda}} | \bm{\lambda},\tilde{\bm{\mu}}) $ is proportional to 

\begin{align*}
    \frac{t^{\tilde{d}(\tilde{\bm{\lambda}},\bm{\lambda})}}{Z(c_{[1,n]}|b)}\frac{\mathcal{L}_{\tilde{\bm{\lambda}}}(c_{[1,n]};t)}{\mathcal{L}_{\bm{\lambda}}(c_{[1,n]};t)} \tilde{\mathcal{L}}_{\tilde{\bm{\lambda}}/\bm{\lambda}}(b;t) & \cdot
    \frac{\mathcal{L}_{\tilde{\bm{\mu}}}(c_{[1,n-1]};t)}{\mathcal{L}_{\tilde{\bm{\lambda}}}(c_{[1,n]};t) } \mathcal{L}_{\tilde{\bm{\lambda}}/\tilde{\bm{\mu}}}(c;t) \\
    &\propto t^{\tilde{d}(\tilde{\bm{\lambda}},\bm{\lambda})} \tilde{\mathcal{L}}_{\tilde{\bm{\lambda}}/\bm{\lambda}}(b;t)  \mathcal{L}_{\tilde{\bm{\lambda}}/\tilde{\bm{\mu}}}(c;t) \;\;
\end{align*}
where we have only kept factors depending on $\tilde{\bm{\lambda}}$.

It is not hard to see from the definition of LLT polynomials that the quantity~$\tilde{\mathcal{L}}_{\tilde{\bm{\lambda}}/\bm{\lambda}}(b;t)$ will be~$0$ unless~$\bm{\tilde{\lambda}}$ corresponds to a configuration where all particles jump by at most~$1$. Furthermore, if~$\bm{\lambda}, \bm{\tilde{\mu}}$ correspond to a particle configuration where there are particles forced to jump or stay, then~$\mathcal{L}_{\tilde{\bm{\lambda}}/\tilde{\bm{\mu}}}(c;t)$ will be~$0$
unless~$\tilde{\bm{\lambda}}$ corresponds to a configuration where all of these particles do in fact jump or stay. Therefore, the possible~$k$-tuples~$\bm{\tilde{\lambda}}$ correspond exactly to the possible outcomes of particle jumps in the proposition. 

Thus, to prove the proposition, it suffices to show that for these~$k$-tuples~$\bm{\tilde{\lambda}}$, the ratios of their transition probabilities are equal to the ratios of the particle transition probabilities described in the proposition. It is enough to compare the ratio for pairs of $\tilde{\bm{x}}$ which differ by a single non-forced particle jump, as any ratio of the particle transition probabilities can be written as a product of such simple ratios. For concreteness, suppose the jump was made by particle $i$ of color $l$.  

Two particle configurations differing by a single particle jump are equivalent to two tuples of partitions differing by the corresponding single cell in one of their Young diagrams. Define $\bm{\delta}(l, i) = (\bm{\delta}(l, i)^{(1)}, \dots, \bm{\delta}(l, i)^{(k)})$, where
$$\bm{\delta}(l, i)^{(c)}_j = \bm{1}_{c = l, j = i} \;\;.$$
 From the above discussion, we see that we must show that
\begin{equation}\label{eqn:ratios}
    \frac{t^{\tilde{d}(\tilde{\bm{\lambda}} +\bm{\delta}(l,i),\bm{\lambda})}\tilde{\mathcal{L}}_{(\tilde{\bm{\lambda}} + \bm{\delta}(l,i))/\bm{\lambda}}(b;t)  \mathcal{L}_{(\tilde{\bm{\lambda}}+\bm{\delta}(l, i))/\tilde{\bm{\mu}}}(c;t)}{t^{\tilde{d}(\tilde{\bm{\lambda}},\bm{\lambda})}\tilde{\mathcal{L}}_{\tilde{\bm{\lambda}}/\bm{\lambda}}(b;t)  \mathcal{L}_{\tilde{\bm{\lambda}}/\tilde{\bm{\mu}}}(c;t)} = 
    c b t^{\#_i(n,l)} \;\;
\end{equation}
where $\#_i(n,l)$ is defined in (\ref{eqn:tpow}), and~$\tilde{ \bm{\lambda}}^{(l)}_i = \bm{\lambda}^{(l)}_i$.

To show that the powers of $c$ and $b$ are correct in the equality (\ref{eqn:ratios}), recall Eqn. (\ref{eqn:monomials}): for LLT polynomials with a single variable we have
\[
\begin{aligned}
\mathcal{L}_{\bm{\kappa}/\bm{\nu}}(c;t) = c^{|\bm{\kappa}/\bm{\nu}|} t^{\text{coinv}(\bm{\kappa}/\bm{\nu})} \\
\mathcal{\tilde L}_{\bm{\kappa}/\bm{\nu}}(b;t) = b^{|\bm{\kappa}/\bm{\nu}|} t^{\text{inv}(\bm{\kappa}/\bm{\nu})}.
\end{aligned}
\]
Since the Young diagram of $(\tilde{\bm{\lambda}} + \bm{\delta}(l, i))/\bm{\lambda}$ has exactly one more cell than that of $\tilde{\bm{\lambda}} /\bm{\lambda}$, the numerator will have exactly one extra factor of $b$, and similarly one extra factor of $c$.

We are left to show the powers of $t$ match on both sides of (\ref{eqn:ratios}). The powers of $t$ on the LHS of (\ref{eqn:ratios}) come from two sources: interactions between color $l$ and colors $m>l$, and interactions between color $l$ and colors $m<l$. We show in detail that the contributions when the color $m$ is larger give exactly the first term in (\ref{eqn:tpow}). A similar analysis can be done to show that the contributions when the  color $m$ is smaller give the second term in (\ref{eqn:tpow}). 

 Looking at the first term of (\ref{eqn:tpow}), note the particle position inequalities
\[
\tilde{\bm{y}}_j^{(n, m)}\le \bm{x}_i^{(n, l)} \le \bm{x}_j^{(n, m)}
\]
can be written in terms of the parts of the partitions as
\[
\tilde{\bm{\mu}}_j^{(m)}-j\le \bm{\lambda}^{(l)}_i-i\le \bm{\lambda}^{(m)}_j-j.
\]
We will show that in the LHS of (\ref{eqn:ratios}) we get an extra power of $t$ in the numerator that is not present in the denominator exactly when there exists an $m$ and $j$ where the above inequality holds.

To do so we do an exhaustive check over all possible relative positions of the corresponding particles and in each case use Lemma \ref{lem:tpow} to determine the powers of $t$ on the LHS of~\eqref{eqn:ratios}. This casework is listed below:

\begin{enumerate}
    \item If $\tilde{\bm{\lambda}}_j^{(m)}-j > \tilde{\bm{\lambda}}_i^{(l)}-i$ and  $\tilde{\bm{\mu}}_j^{(m)}-j \ge \tilde {\bm{\mu}}_i^{(l)}-i$ there are three subcases for the contribution from Eqn. (\ref{eq:tHor}):
    \begin{enumerate}
        \item If $\tilde{\bm{\lambda}}_i^{(l)}-i \ge \tilde{\bm{\mu}}_j^{(m)}-j$ then there is a power of $t^{\tilde{\bm{\lambda}}_i^{(l)}-i +1 - \tilde{\bm{\mu}}_j^{(m)}+j}$ in the numerator and $t^{\tilde{\bm{\lambda}}_i^{(l)}-i - \tilde{\bm{\mu}}_j^{(m)}+j}$ in the denominator.
        \item  If $\tilde{\bm{\lambda}}_i^{(l)}-i +1 
        = \tilde{\bm{\mu}}_j^{(m)}-j$ then there is a power of $t^0$ in the numerator and no contribution to the denominator.
        \item If $\tilde{\bm{\lambda}}_i^{(l)}-i + 1 < 
         \tilde{\bm{\mu}}_j^{(m)}-j$ there is no contribution to either the numerator or the denominator.
    \end{enumerate}
    For the contribution of Eqn. (\ref{eq:tVert}), note there is no contribution to the denominator since $\tilde{\bm{\lambda}}_j^{(m)}-j \ne \tilde{\bm{\lambda}}_i^{(l)}-i$. In the numerator there is no contribution since $\bm{\lambda}_i^{(l)}-i \ne \tilde{\bm{\lambda}}_i^{(l)}-i + 1$.

    Overall, we get a net power of $t$ in the ratio in case (a) and none in the other cases.

    \item If $\tilde{\bm{\lambda}}_j^{(m)}-j > \tilde {\bm{\lambda}}_i^{(l)}-i$ and  $\tilde{\bm{\mu}}_j^{(m)}-j < \tilde {\bm{\mu}}_i^{(l)}-i$, note that we must have $\tilde {\bm{\lambda}}_i^{(l)}-i \ge \tilde {\bm{\mu}}_i^{(l)}-i$. Then from Eqn. (\ref{eq:tHor}) we get a $t^{\tilde {\bm{\lambda}}_i^{(l)} - i - \tilde {\bm{\mu}}_i^{(l)}+i}$ in the denominator and a $t^{\tilde {\bm{\lambda}}_i^{(l)}+1 - i - \tilde {\bm{\mu}}_i^{(l)}+i}$ in the numerator.
    
    Eqn. (\ref{eq:tVert}) does not contribute to either for the same reason as in case 1.

    Overall, we have a net power of $t$ in the ratio.
    
    \item If $\tilde{\bm{\lambda}}_j^{(m)}-j < \tilde {\bm{\lambda}}_i^{(l)}-i$ and  $\tilde{\bm{\mu}}_j^{(m)}-j \ge \tilde {\bm{\mu}}_i^{(l)}-i$ then from Eqn. (\ref{eq:tHor}) we get a $t^{\tilde {\bm{\lambda}}_j^{(m)} - j - \tilde {\bm{\mu}}_j^{(m)}+j + 1}$ in both the numerator and the denominator. 
    
    Eqn. (\ref{eq:tVert}) does not contribute to either, again, for the same reason as in case 1.

    Overall, there is no net power of $t$ in the ratio.

    \item If $\tilde{\bm{\lambda}}_j^{(m)}-j < \tilde {\bm{\lambda}}_i^{(l)}-i$ and  $\tilde{\bm{\mu}}_j^{(m)}-j < \tilde {\bm{\mu}}_i^{(l)}-i$ then for the contribution from Eqn. (\ref{eq:tHor}) we have two subcases:
    \begin{enumerate}
        \item If $\tilde{\bm{\lambda}}_j^{(m)}-j\ge \tilde {\bm{\mu}}_i^{(l)}-i$ then we get a $t^{\tilde {\bm{\lambda}}_j^{(m)} - j - \tilde {\bm{\mu}}_i^{(l)}+i +1}$ in both the numerator and the denominator.
        \item Otherwise, there is no contribution to either.
    \end{enumerate}
    
    Eqn. (\ref{eq:tVert}) still does not contribute to either for the same reasons as above.

    Overall there is no net power of $t$ in the ratio.

    \item If $\tilde{\bm{\lambda}}_j^{(m)}-j = \tilde {\bm{\lambda}}_i^{(l)}-i$ and  $\tilde{\bm{\mu}}_j^{(m)}-j \ge \tilde {\bm{\mu}}_i^{(l)}-i$ then from Eqn. (\ref{eq:tHor}) we get a $t^{\tilde {\bm{\lambda}}_j^{(m)} - j - \tilde {\bm{\mu}}_j^{(m)}+j }$ in the denominator and a $t^{\tilde {\bm{\lambda}}_j^{(m)} - j - \tilde {\bm{\mu}}_j^{(m)}+j +1}$ in the numerator. Here the extra power of $t$ in the numerator comes from the indicator $1(\tilde{\bm{\lambda}}_j^{(m)}-j < \tilde {\bm{\lambda}}_i^{(l)}-i+1)$ in (\ref{eq:tHor}). 
    
    Note that (\ref{eq:tVert}) does not contribute to the numerator as $\tilde{\bm{\lambda}}_j^{(m)}-j \ne \tilde {\bm{\lambda}}_i^{(l)}-i+1$. The contribution from Eqn. (\ref{eq:tVert}) to the denominator can be split into  two subcases:
    \begin{enumerate}
        \item  If $\tilde {\bm{\lambda}}_j^{(m)} - j = \bm{\lambda}_j^{(m)}-j+1$, it contributes a single power of $t$ to the denominator.
        \item Otherwise, it contributes nothing to the denominator.
    \end{enumerate}

    Overall, there is no net power of $t$ in the ratio if $\tilde {\bm{\lambda}}_j^{(m)} - j = \bm{\lambda}_j^{(m)}-j+1$, and a net a power of $t$ otherwise.

    \item If $\tilde{\bm{\lambda}}_j^{(m)}-j = \tilde {\bm{\lambda}}_i^{(l)}-i$ and  $\tilde{\bm{\mu}}_j^{(m)}-j < \tilde {\bm{\mu}}_i^{(l)}-i$, note that we must have $\tilde{\bm{\lambda}}_j^{(m)}-j \ge \tilde {\bm{\mu}}_i^{(l)}-i$.  Then for the contribution from Eqn. (\ref{eq:tHor}) we get a $t^{\tilde {\bm{\lambda}}_j^{(m)} - j - \tilde {\bm{\mu}}_i^{(l)}+i }=1$ in the denominator and a $t^{\tilde {\bm{\lambda}}_j^{(m)} - j - \tilde {\bm{\mu}}_i^{(l)}+i +1}=t$ in the numerator.
    
      Eqn. (\ref{eq:tVert}) does not contribute to the numerator for the same reason as in case 5. The contribution to the denominator from Eqn. (\ref{eq:tVert}) can be split into two subcases:
    \begin{enumerate}
        \item If $\tilde {\bm{\lambda}}_j^{(m)} - j = \bm{\lambda}_j^{(m)}-j+1$, it contributes a single power of $t$ to the denominator.
        \item Otherwise, it contributes nothing to the denominator.
    \end{enumerate} 
    Overall, there is no net power of $t$ in the ratio if $\tilde {\bm{\lambda}}_j^{(m)} - j = \bm{\lambda}_j^{(m)}-j+1$, and a net a power of $t$ otherwise. 
    
\end{enumerate}
One can check the ratio of the powers of $t$ is given by
\[
\begin{cases}
1, &\text{ if } \tilde {\bm{\mu}}_j^{(m)}-j \le \bm{\lambda}_i^{(l)}-i \leq \bm{\lambda}_j^{(m)}-j \text{  (Cases 1 (a), 2, 5 (b), 6 (b))} \\
0, &\text{ otherwise (Cases 1 (b), 1 (c), 3, 4, 5 (a), 6 (a))}
\end{cases}
\]
exactly as we desired. Summing over all $m>l$ and all $j$ gives the first term in Eqn. (\ref{eqn:tpow}).

\end{proof}

\begin{figure}
         \centering
         \includegraphics[scale=.7]{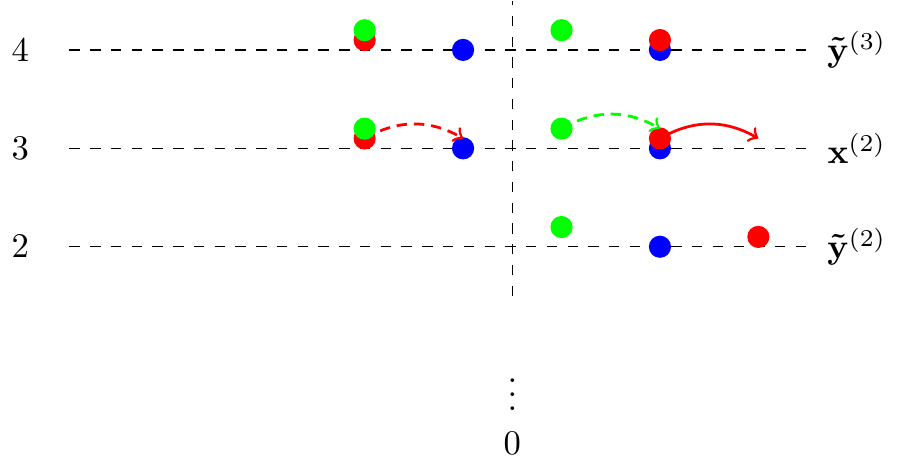}
        \caption{Shown is a possible configuration of particles of the Maya diagrams on diagonals $2, 3$ and $4$, after step one of the update. We use the convention blue $ = 1 <$ red $= 2 <$ green $ = 3$. The solid red arrow denotes a forced jump for the red particle to preserve the interlacing of the red particles. Supposing we have~$c_i = b_i \equiv 1$, the dashed red arrow denotes a jump which will happen with probability $\frac{t^2}{1 + t^2}$, because~$\bm{\tilde{y}}_2^{(2, 1)} \leq \bm{x}_2^{(2, 2)}+1 \leq \bm{x}_2^{(2, 1)}$ and~$\bm{\tilde{y}}_2^{(2, 3)} \leq \bm{x}_2^{(2, 2)} \leq \bm{x}_2^{(2, 3)}$. Similarly, the dashed green arrow denotes a jump which has probability $\frac{t}{1 + t}$. }
        \label{fig:jumpprobs_ex}
\end{figure}

As a simple example of the particle dynamics, one may set all parameters $b_i$,$c_i$ equal to $1$ and consider $\bm{x}^{(1)}$, that is, the bottom level of particles. Along this row we have exactly one particle of each color, and the marginal evolution of these particles is itself a Markov chain. At each step each particle independently either stays in place or jumps by $1$ to the right. The probability for the particle of color $l$, at position $\bm{x}_1^{(1, l)}$, to jump is 
$$\frac{t^{\#_1(1,l)}}{1 + t^{\#_1(1,l)}}$$
where 
\begin{align*}
\#_1(1,l)= &|\{m>l|  \bm{x}_1^{(1, l)} \leq \bm{x}_1^{(1, m)} \} | \\
& +| \{m<l|\bm{x}_1^{(1, l)} < \bm{x}_1^{(1, m)}  \} | \;\;.
\end{align*}
In other words, if at time $T$ particles are ordered $0,1,\dots, n-1$ from right to left, breaking ties by putting larger colors first, the jump probability of particle $i$ is $\frac{t^i}{1 + t^i}$.  When $t<1$, we see that if a particle falls behind the other it becomes discouraged and moves more slowly, while for $t>1$ it becomes determined to catch up and moves more quickly.


\begin{figure}
     \centering
     \begin{subfigure}[a]{1\linewidth}
         \centering
         \includegraphics[scale=.35]{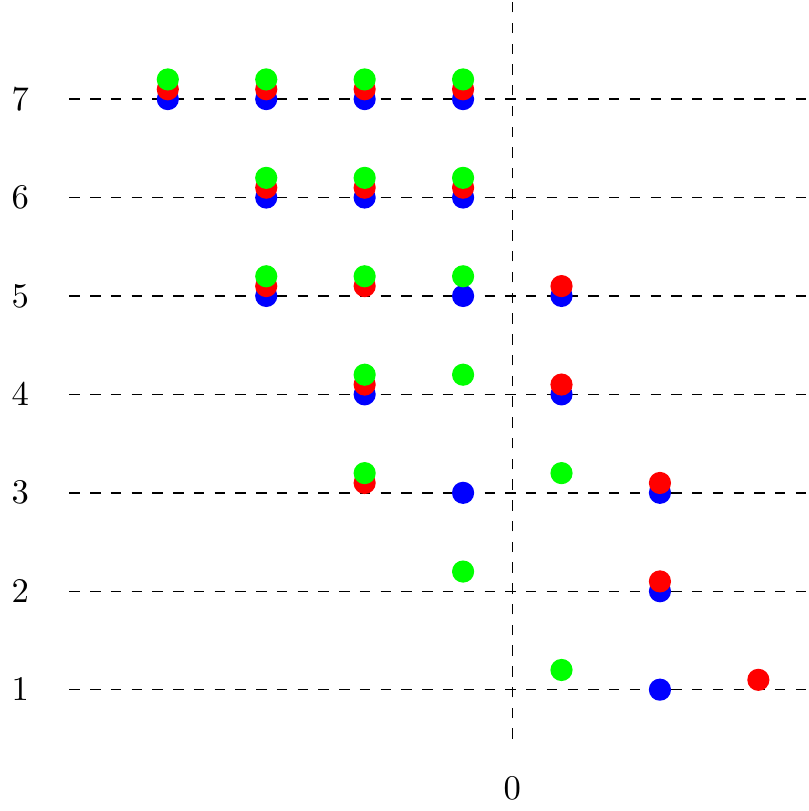}\hfill
         \includegraphics[scale=.4]{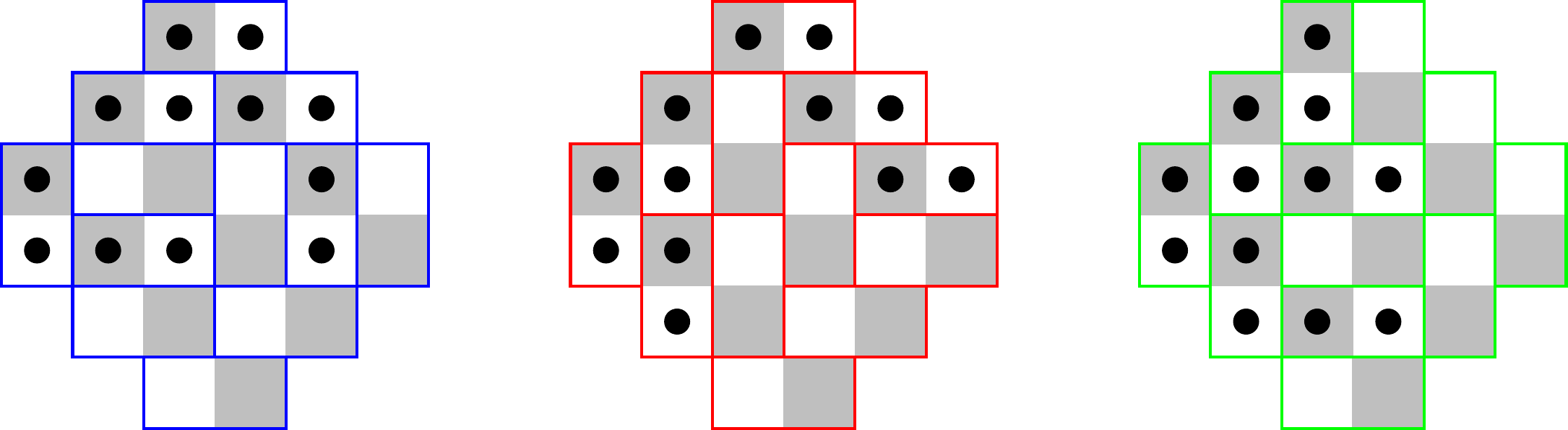}
         \caption{Initial configuration.}
         \label{fig:init}
     \end{subfigure}\\
     \vfill
     \begin{subfigure}[b]{1\linewidth}
         \centering
         \includegraphics[scale=.35]{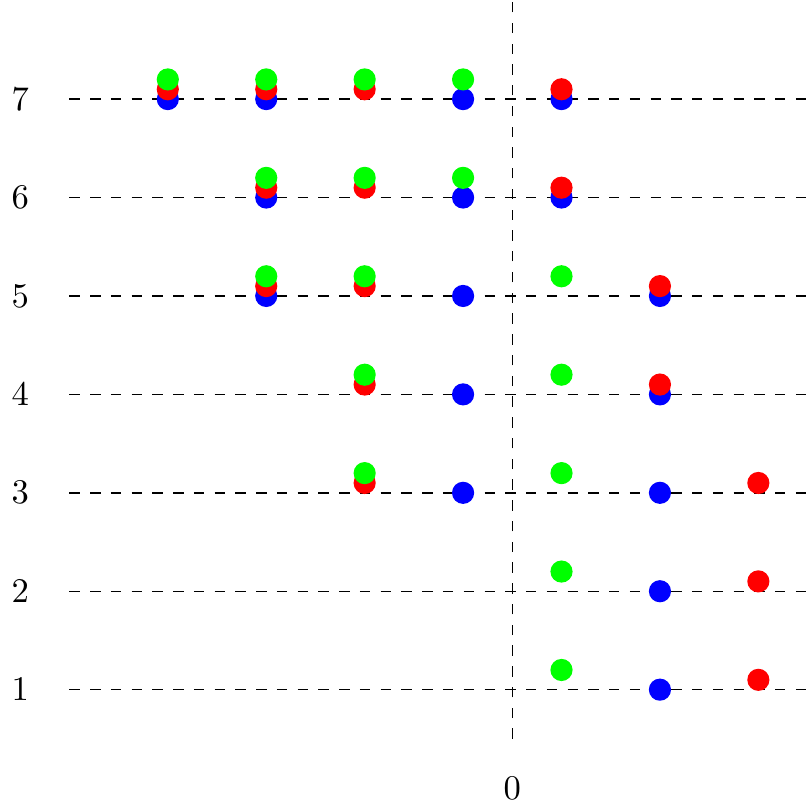}\hfill
         \includegraphics[scale=.3]{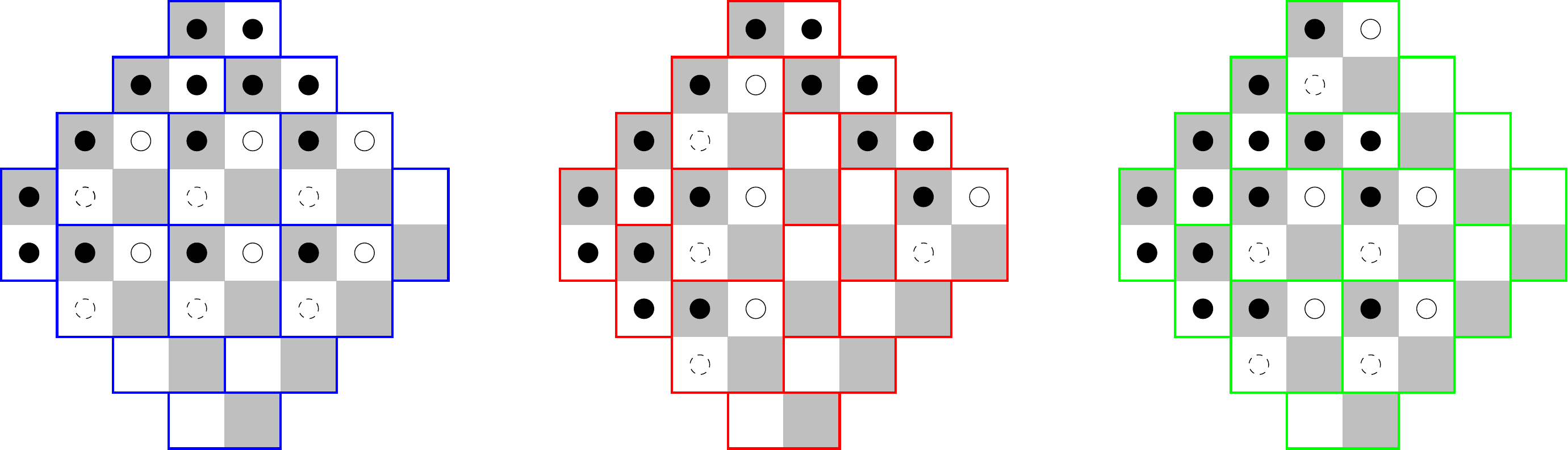}
         \caption{After step 1 and forced jumps.}
         \label{fig:after_step1}
     \end{subfigure}
     
     \begin{subfigure}[b]{1\linewidth}
         \centering
         \includegraphics[scale=.35]{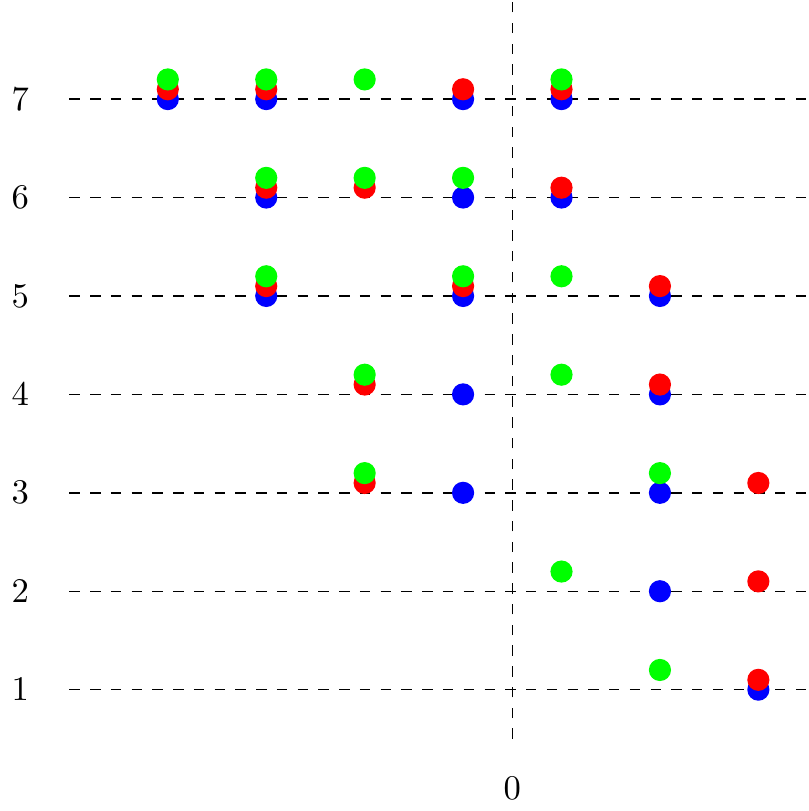}\hfill
         \includegraphics[scale=.3]{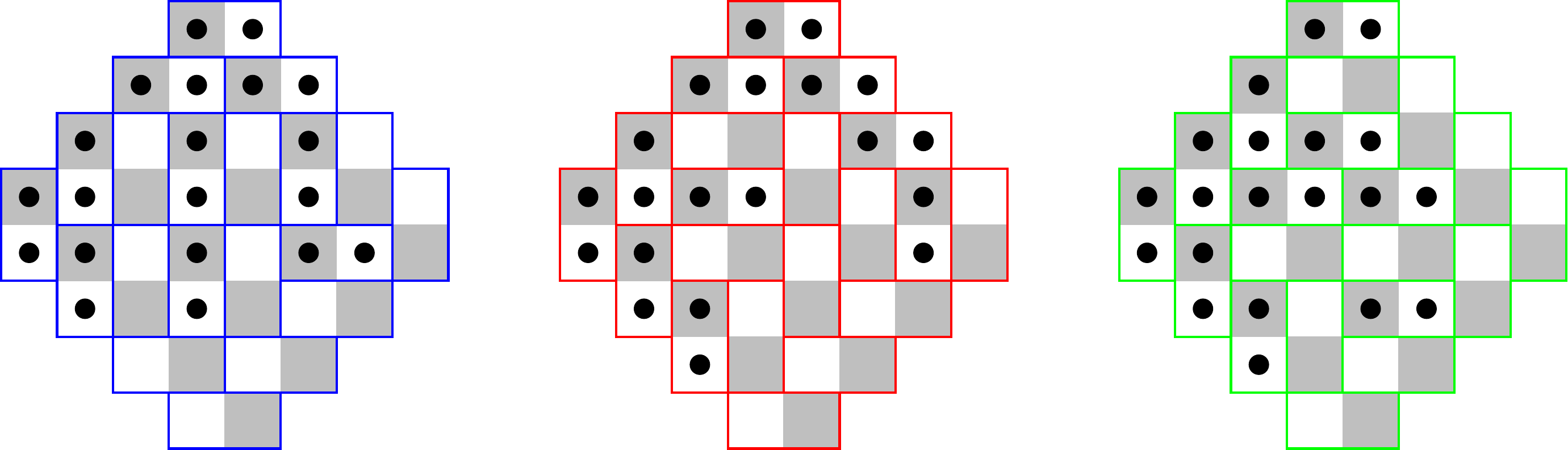}
         \caption{After step 2.}
         \label{fig:after_step2}
     \end{subfigure}
        \caption{Here we illustrate one possibility for the random update of the particle configuration of the $k = 3$-tiling of size $N = 3$ shown in Figure \ref{ex:3til-bijection}. First, shown is the initial configuration, augmented with an extra empty tuple of Maya diagrams. Second, the particles are shown after both step 1 and the forced jumps of step 2. Finally, one possible outcome of step 2 is shown. }
        \label{fig:example}
\end{figure}

\section{Shuffling}\label{sec:shuffling}
In this section we show how the particle dynamics defined in Section \ref{sec:Markov} can be described as an algorithm which acts directly on the tilings by sliding, destroying, and creating dominos. In the one color case, this is known as the domino shuffling algorithm~\cite{aztec0, shuffling1}.

\subsection{Interpretation in terms of Dominos: Schur Case}\label{subsec:SchurShuffle}
Domino shuffling is a sampling algorithm to generate a random rank-$N$ domino tiling. Via the bijection between tilings and interlacing particle arrays, the particle dynamics defined in Section \ref{sec:MarkovSchur} coincides with the shuffling algorithm. Here we will review the shuffling algorithm, and refer the reader to Propp \cite{shuffling1} for more details.

 Recall that we assign weights $c_i$ to the horizontal dominos on the diagonal slice~$2 i - 1$, and $b_{N-i+1}$ to the horizontal dominos on slice~$2 i$, respectively, for~$i = 1,\dots, N$. Take two additional numbers $c_{N+1}, b_{N+1}$. We now describe a way to randomly sample a tiling of rank~$N+1$ given one of rank~$N$, such that if the original one is sampled with weights $(c_1, \dots, c_N), (b_N, \dots, b_1)$, the one obtained from the algorithm will be sampled with weights $(c_1, \dots, c_{N+1}), (b_{N+1}, \dots, b_1)$. 
 
 Given the checkerboard coloring of the squares of rank-$N$ Aztec diamond, recall that there are four types of dominos which can appear in a tiling. They are shown in Figure \ref{fig:ADex}. Label these four types of dominos as S (South), W (West), N (North), E (East), respectively. Given a tiling $T$ of rank~$N$, following three steps will lead to a tiling $T'$ of rank~$N+1$:
\begin{enumerate}
    \item (Sliding) S dominos slide one unit South, W dominos slide one unit West, N dominos slide one unit North, and E dominos slide one unit East
    \item (Destruction) If two dominos cross each other's path as they slide, both are destroyed and deleted from the tiling.
    \item (Creation) What remains will be a partial tiling of rank~$N+1$. The un-tiled portion will be a disjoint union of $2 \times 2$ blocks (in a unique way). Fill in each $2 \times 2$ block independently with either a vertical pair or a horizontal pair of dominos, with probabilities  
    \begin{align*}
        \text{ vertical: } & \;\; \frac{1}{1 + c_i b_{N-i+2} }\\
        \text{ horizontal: } & \;\; \frac{c_i b_{N-i+2}}{1 + c_i b_{N-i+2}}
    \end{align*}
    where here $2 i - 1$ is the diagonal of the lower left square in the block, with respect to the indexing of diagonals on the rank-$(N+1)$ Aztec diamond.
    
\end{enumerate}

The correspondence of the shuffling algorithm to the particle process is well known~\cite{borodin2015random}. The following facts, from which the equivalence of shuffling and the particle transition probabilities can be deduced, will be useful for us in the next section:
\begin{itemize}

    \item Slides west correspond to a particle being forced to stay. 
    \item Slides north correspond to a particle being forced to jump.
    \item Creations correspond to a particle which can either stay or jump.
\end{itemize}
For an example, see the red tiling and red particles in Figure \ref{fig:example}.

\subsection{Interpretation in terms of Dominos: LLT Case}
We now state the analogous shuffling algorithm corresponding to the LLT particle process.
\begin{thm}\label{thm:main}
The following algorithm generates a random $k$-tiling of the rank-$N$ Aztec diamond with probability proportional to its weight. 

\bigskip
\noindent \textbf{Algorithm}: Start with a rank-0 Aztec diamond. To get from a rank-$(T-1)$ to rank-$T$ $k$-tiling,
\begin{enumerate}
    \item Slide and destroy as in the normal domino shuffle, independently for each color.
    \item Fill in empty $2\times 2$ squares according to the rule:
    \begin{enumerate}
        \item For the smallest color put a horizontal pair of dominoes with probability 
        $$\frac{c_i b_{N-i+2} t^{\#_1(1)}}{1+c_i b_{N-i+2} t^{\#_1(1)}}$$ 
        where 
        \[
        \#_1(l) = \# \text{ colors $m>l$ that locally have } 
\resizebox{0.08\textwidth}{!}{
\begin{tikzpicture}[baseline = (current bounding box).center]
\draw[] (0,0) rectangle (1,1); \draw[] (1,1) rectangle (2,2);
\draw[fill=lightgray] (0,1) rectangle (1,2); \draw[fill=lightgray] (1,0) rectangle (2,1);  
\draw[very thick, red] (0.05,0.05) rectangle (1.05,2.05);
\end{tikzpicture}}
\text{, } 
\resizebox{0.08\textwidth}{!}{
\begin{tikzpicture}[baseline = (current bounding box).center]
\draw[] (0,0) rectangle (1,1); \draw[] (1,1) rectangle (2,2);
\draw[fill=lightgray] (0,1) rectangle (1,2); \draw[fill=lightgray] (1,0) rectangle (2,1);  
\draw[very thick, red] (0.05,0.05) rectangle (2.05,1.05);
\end{tikzpicture}}
\text{, or creation.}
        \]
        where here $2 i$ is the diagonal of the lower left square in the block, otherwise put a vertical pair. Do all of these first.
        \item Now do all the larger colors from smallest to largest. For color $l>1$, put a horizontal pair of dominoes with probability 
        $$\frac{c_i b_{N-i+2} t^{\#_1(l)+\#_2(l)}}{1+c_i b_{N-i+2} t^{\#_1(l)+\#_2(l)}}$$
        where 
        \[
        \#_2(l) = \# \text{ colors $m<l$ that locally have } 
\resizebox{0.12\textwidth}{!}{
\begin{tikzpicture}[baseline = (current bounding box).center]
\draw[] (0,0) rectangle (1,1); \draw[] (1,1) rectangle (2,2);
\draw[fill=lightgray] (0,1) rectangle (1,2); \draw[fill=lightgray] (1,0) rectangle (2,1);  
\draw[very thick, blue] (1,1) rectangle (3,2);
\end{tikzpicture}}
\text{ or } 
\resizebox{0.08\textwidth}{!}{
\begin{tikzpicture}[baseline = (current bounding box).center]
\draw[] (0,0) rectangle (1,1); \draw[] (1,1) rectangle (2,2);
\draw[fill=lightgray] (0,1) rectangle (1,2); \draw[fill=lightgray] (1,0) rectangle (2,1);  
\draw[very thick, blue] (1,1) rectangle (2,3);
\end{tikzpicture}}
        \]
    and $\#_1(l)$ is as in part (a), otherwise put a vertical pair.
    \end{enumerate}

    This gives a $k$-tiling of the Aztec diamond whose rank has increased by one. 
\end{enumerate}
Repeat steps (2) and (3) until you get a rank-$N$ Aztec diamond.
\end{thm}
\begin{proof}
It suffices to show that the update step corresponds to the update of tuples of interlacing arrays of particles described in Proposition \ref{prop:probs}. For each color, the transition probabilities only differ from those of the usual domino shuffling algorithm through the creation probabilities. As in the one color domino shuffling, creations correspond to particles which can jump or stay, with creation of a horizontal pair corresponding to a particle jumping by $1$ and creation of a vertical pair corresponding to staying. It is enough to show that the power of $t$ in the creation step described above corresponds to the power of $t$ in the jump probability of the particle process.

Consider a color $l$, and a larger color $m$, which we represent by blue and red, respectively. Suppose that we are doing a creation for color $l$ and in that $2 \times 2$ box the red configuration looks like one of the three configurations given in the definition of $\#_1$ in the theorem statement. Given these domino configurations, we can give the possibilities for the corresponding red and blue particle configurations. These are shown in Figure \ref{fig:jumpcases1}. Note that in each case the particle configuration contributes a power of $t$ to the blue particle's jump probability in \eqref{eqn:tpow}. On the other hand, it is also clear from Figure~\ref{fig:jumpcases1} that these are all of the possible cases in which red particles contribute a power of~$t$ to the blue particle's jump probability in \eqref{eqn:tpow}.

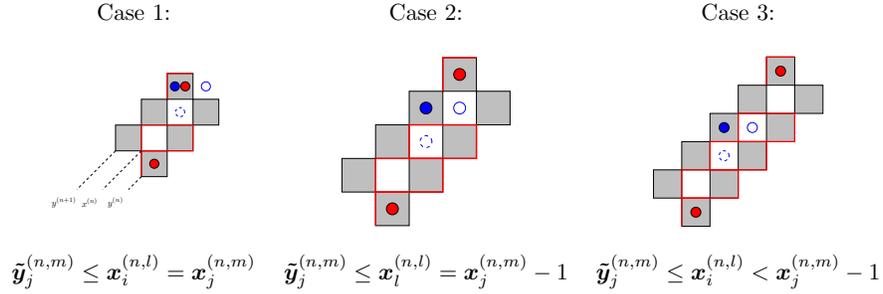
\begin{figure}[!htb]
\begin{center}
\resizebox{\textwidth}{!}{
\begin{tabular}{ccc}
Case 1: & Case 2: & Case 3:
\\
\\
\resizebox{0.22\textwidth}{!}{
\begin{tikzpicture}[baseline = (current bounding box).center]
\draw[fill=lightgray] (0,0) rectangle (1,1); 
\draw[fill=lightgray] (-1,1) rectangle (0,2); \draw[fill=lightgray] (1,1) rectangle (2,2);
\draw[fill=lightgray] (0,2) rectangle (1,3); \draw[fill=lightgray] (2,2) rectangle (3,3);
\draw[fill=lightgray] (1,3) rectangle (2,4); 
\draw[fill=red] (0.5,0.5) circle (5pt);
\draw[fill=blue] (1.3,3.5) circle (5pt);
\draw[fill=red] (1.7,3.5) circle (5pt);
\draw[blue,dashed,fill=none] (1.5,2.5) circle (5pt);\draw[blue,fill=none] (2.5,3.5) circle (5pt);
\draw[red, very thick] (0,0)--(0,1)--(1,1);
\draw[red, very thick] (0,1) rectangle (2,2);
\draw[red, very thick] (1,3)--(1,4)--(2,4);
\draw[dashed] (0,0)--(-0.5,-0.5); \draw[dashed] (0,1)--(-1.5,-0.5); \draw[dashed] (-1,1)--(-2.5,-0.5);
\node at (-1,-1) {$y^{(n)}$}; \node at (-2,-1) {$x^{(n)}$}; \node at (-3,-1) {$y^{(n+1)}$};
\end{tikzpicture}
} &
\resizebox{0.22\textwidth}{!}{
\begin{tikzpicture}[baseline = (current bounding box).center]
\draw[fill=lightgray] (0,0) rectangle (1,1); 
\draw[fill=lightgray] (-1,1) rectangle (0,2); \draw[fill=lightgray] (1,1) rectangle (2,2);
\draw[fill=lightgray] (0,2) rectangle (1,3); \draw[fill=lightgray] (2,2) rectangle (3,3);
\draw[fill=lightgray] (1,3) rectangle (2,4); \draw[fill=lightgray] (3,3) rectangle (4,4);
\draw[fill=lightgray] (2,4) rectangle (3,5); 
\draw[fill=red] (0.5,0.5) circle (5pt);
\draw[fill=blue] (1.5,3.5) circle (5pt);
\draw[fill=red] (2.5,4.5) circle (5pt);
\draw[blue,dashed,fill=none] (1.5,2.5) circle (5pt);\draw[blue,fill=none] (2.5,3.5) circle (5pt);
\draw[red, very thick] (0,0)--(0,1)--(1,1);
\draw[red, very thick] (0,1) rectangle (2,2); \draw[red, very thick] (1,2) rectangle (3,3);
\draw[red, very thick] (2,4)--(2,5)--(3,5);
\end{tikzpicture}
} &
\resizebox{0.22\textwidth}{!}{
\begin{tikzpicture}[baseline = (current bounding box).center]
\draw[fill=lightgray] (0,0) rectangle (1,1); 
\draw[fill=lightgray] (-1,1) rectangle (0,2); \draw[fill=lightgray] (1,1) rectangle (2,2);
\draw[fill=lightgray] (0,2) rectangle (1,3); \draw[fill=lightgray] (2,2) rectangle (3,3);
\draw[fill=lightgray] (1,3) rectangle (2,4); \draw[fill=lightgray] (3,3) rectangle (4,4);
\draw[fill=lightgray] (2,4) rectangle (3,5); \draw[fill=lightgray] (4,4) rectangle (5,5);
\draw[fill=lightgray] (3,5) rectangle (4,6);
\draw[fill=red] (0.5,0.5) circle (5pt);
\draw[fill=blue] (1.5,3.5) circle (5pt);
\draw[fill=red] (3.5,5.5) circle (5pt);
\draw[blue,dashed,fill=none] (1.5,2.5) circle (5pt);\draw[blue,fill=none] (2.5,3.5) circle (5pt);
\draw[red, very thick] (0,0)--(0,1)--(1,1);
\draw[red, very thick] (0,1) rectangle (2,2); \draw[red, very thick] (1,2) rectangle (3,3); \draw[red, very thick] (2,3) rectangle (4,4);
\draw[red, very thick] (3,5)--(3,6)--(4,6);
\end{tikzpicture}
}
\\
\\
$\bm{\tilde{y}}^{(n, m)}_{j}\le \bm{x}^{(n, l)}_{i} = \bm{x}^{(n, m)}_{j}$ & $\bm{\tilde{y}}^{(n,m )}_{j} \le \bm{x}^{(n, l)}_{l} = \bm{x}^{(n, m)}_{j}-1$ & $\bm{\tilde{y}}^{(n, m)}_{j} \le \bm{x}^{(n, l)}_{i} < \bm{x}^{(n, m)}_{j}-1$
\end{tabular}
}
\end{center}
\caption{Three cases in which the dominos of a larger color $m$ (in red) contribute a factor of $t$ to the creation probabilities of a smaller color $l$ (in blue). In each case we also show the relative positions of the blue particle, which is jumping, with the nearby red particles whose positions lead to the contribution of the power of~$t$. We use the same notation as in Proposition~\ref{prop:probs}; in particular,~$\bm{x}^{(n, l)}_{i}$ and~$\bm{x}^{(n, m)}_{j}$ denote the blue and red particles in the~$x^{(n)}$ diagonal before the update. By the update rules for $\bm{y}$ particles (c.f. Proposition~\ref{prop:probs}), these agree with the particle positions~$\bm{\tilde{y}}^{(n+1,l)}_{i}$ and 
$ \bm{\tilde{y}}^{(n+1,m)}_{j}$ in the~$y^{(n+1)}$ diagonal after the update, which are shown in the figure. Note that the figures are still valid if~$j = n$; in this case the bottom most red particle would be off of the Aztec diamond. The dashed blue circle indicates where the blue particle would be if it does not jump, the solid blue circle indicates where the particle would be if it jumped.} 
\label{fig:jumpcases1}
\end{figure}

Now consider a color $l$, and a smaller color $m$, which we represent by red and blue, respectively (so again blue is the smaller color, but now red is the color whose transition probability we are discussing). Suppose that we are doing a creation for the red dominos in a $2 \times 2$ box, and that the blue dominos there look like one of the two configurations in the definition of $\#_2$ in the theorem statement. Then the possibilities for the corresponding red and blue particle configurations are shown in Figure \ref{fig:jumpcases2}. Again we see that one of these cases occurs if and only if the red particle's transition probability obtains a power of~$t$ from blue in \eqref{eqn:tpow}.

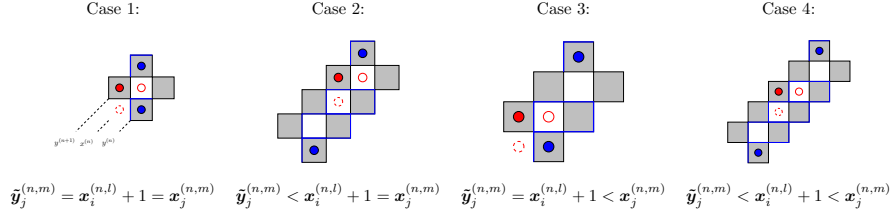
\begin{figure}[!htb]
\begin{center}
\resizebox{\textwidth}{!}{
\begin{tabular}{cccc}
Case 1: & Case 2: & Case 3: & Case 4:
\\
\\
\resizebox{0.22\textwidth}{!}{
\begin{tikzpicture}[baseline = (current bounding box).center]
\draw[fill=lightgray] (0,0) rectangle (1,1); 
\draw[fill=lightgray] (-1,1) rectangle (0,2); \draw[fill=lightgray] (1,1) rectangle (2,2);
\draw[fill=lightgray] (0,2) rectangle (1,3); 
\draw[fill=blue] (0.5,0.5) circle (5pt);
\draw[fill=blue] (0.5,2.5) circle (5pt);
\draw[fill=red] (-0.5,1.5) circle (5pt);
\draw[red,dashed,fill=none] (-0.5,0.5) circle (5pt);\draw[red,fill=none] (0.5,1.5) circle (5pt);
\draw[blue, ultra thick] (0,0)--(0,1)--(1,1);
\draw[blue, ultra thick] (0,2)--(0,3)--(1,3);
\draw[dashed] (0,0)--(-0.5,-0.5); \draw[dashed] (-1,0)--(-1.5,-0.5); \draw[dashed] (-1,1)--(-2.5,-0.5);
\node at (-1,-1) {$y^{(n)}$}; \node at (-2,-1) {$x^{(n)}$}; \node at (-3,-1) {$y^{(n+1)}$};
\end{tikzpicture}
} &
\resizebox{0.22\textwidth}{!}{
\begin{tikzpicture}[baseline = (current bounding box).center]
\draw[fill=lightgray] (0,0) rectangle (1,1); 
\draw[fill=lightgray] (-1,1) rectangle (0,2); \draw[fill=lightgray] (1,1) rectangle (2,2);
\draw[fill=lightgray] (0,2) rectangle (1,3); \draw[fill=lightgray] (2,2) rectangle (3,3);
\draw[fill=lightgray] (1,3) rectangle (2,4); \draw[fill=lightgray] (3,3) rectangle (4,4);
\draw[fill=lightgray] (2,4) rectangle (3,5); 
\draw[fill=blue] (0.5,0.5) circle (5pt);
\draw[fill=red] (1.5,3.5) circle (5pt);
\draw[fill=blue] (2.5,4.5) circle (5pt);
\draw[red,dashed,fill=none] (1.5,2.5) circle (5pt);\draw[red,fill=none] (2.5,3.5) circle (5pt);
\draw[blue, very thick] (0,0)--(0,1)--(1,1);
\draw[blue, very thick] (0,1) rectangle (2,2); \draw[blue, very thick] (1,2) rectangle (3,3);
\draw[blue, very thick] (2,4)--(2,5)--(3,5);
\end{tikzpicture}
} &
\resizebox{0.22\textwidth}{!}{
\begin{tikzpicture}[baseline = (current bounding box).center]
\draw[fill=lightgray] (0,0) rectangle (1,1); 
\draw[fill=lightgray] (-1,1) rectangle (0,2); \draw[fill=lightgray] (1,1) rectangle (2,2);
\draw[fill=lightgray] (0,2) rectangle (1,3); \draw[fill=lightgray] (2,2) rectangle (3,3);
\draw[fill=lightgray] (1,3) rectangle (2,4);
\draw[fill=blue] (0.5,0.5) circle (5pt);
\draw[fill=red] (-0.5,1.5) circle (5pt);
\draw[fill=blue] (1.5,3.5) circle (5pt);
\draw[red,dashed,fill=none] (-0.5,0.5) circle (5pt);\draw[red,fill=none] (0.5,1.5) circle (5pt);
\draw[blue, very thick] (0,0)--(0,1)--(1,1);
\draw[blue, very thick] (0,1) rectangle (2,2);
\draw[blue, very thick] (1,3)--(1,4)--(2,4);
\end{tikzpicture}
} &
\resizebox{0.22\textwidth}{!}{
\begin{tikzpicture}[baseline = (current bounding box).center]
\draw[fill=lightgray] (0,0) rectangle (1,1); 
\draw[fill=lightgray] (-1,1) rectangle (0,2); \draw[fill=lightgray] (1,1) rectangle (2,2);
\draw[fill=lightgray] (0,2) rectangle (1,3); \draw[fill=lightgray] (2,2) rectangle (3,3);
\draw[fill=lightgray] (1,3) rectangle (2,4); \draw[fill=lightgray] (3,3) rectangle (4,4);
\draw[fill=lightgray] (2,4) rectangle (3,5); \draw[fill=lightgray] (4,4) rectangle (5,5);
\draw[fill=lightgray] (3,5) rectangle (4,6);
\draw[fill=blue] (0.5,0.5) circle (5pt);
\draw[fill=red] (1.5,3.5) circle (5pt);
\draw[fill=blue] (3.5,5.5) circle (5pt);
\draw[red,dashed,fill=none] (1.5,2.5) circle (5pt);\draw[red,fill=none] (2.5,3.5) circle (5pt);
\draw[blue, very thick] (0,0)--(0,1)--(1,1);
\draw[blue, very thick] (0,1) rectangle (2,2); \draw[blue, very thick] (1,2) rectangle (3,3); \draw[blue, very thick] (2,3) rectangle (4,4);
\draw[blue, very thick] (3,5)--(3,6)--(4,6);
\end{tikzpicture}
}
\\
\\
$\bm{\tilde{y}}^{(n, m)}_{j} = \bm{x}^{(n, l)}_{i}+1 = \bm{x}^{(n, m)}_{j}$ & $\bm{\tilde{y}}^{(n, m)}_{j} < \bm{x}^{(n, l)}_{i}+1 = \bm{x}^{(n, m)}_{j}$ & $\bm{\tilde{y}}^{(n, m)}_{j} = \bm{x}^{(n, l)}_{i}+1 < \bm{x}^{(n, m)}_{j}$ & $\bm{\tilde{y}}^{(n, m)}_{j} < \bm{x}^{(n, l)}_{i}+1 < \bm{x}^{(n, m)}_{j}$
\end{tabular}
}
\end{center}
\caption{Four cases in which the smaller color $m$ (in blue) contributes a factor of $t$ to the creation probability of a larger color $l$ (in red). Similarly to Figure~\ref{fig:jumpcases1}, we use~$\bm{x}_i^{(n,l)}, \bm{x}_j^{(n,m)}$ to denote particle positions of the $x^{(n)}$ diagonal at the previous time step, which agree with the particle positions~$\bm{\tilde{y}}^{(n+1)}$ of the~$y^{(n+1)}$ diagonal at the next time step. The dashed red circle indicates where the red particle would be if it does not jump, the solid red circle indicates where the particle would be if it jumped. } 
\label{fig:jumpcases2}
\end{figure}

As a result, the powers of $t$ that we pick up from $\#_1(l) + \#_2(l)$ are exactly the powers of $t$ described in Proposition \ref{prop:probs}.

\end{proof}

\section{Conclusion}\label{sec:conclusion}
In this article we generalize the domino shuffling algorithm for tilings of the Aztec diamond to shuffling algorithm for interacting $k$-tilings studied in \cite{LLTaztec}. We present this algorithm in a variety of ways:
\begin{enumerate}
    \item As a Markov chain on LLT processes, which generalize the well-studied Schur processes.
    \item As a sequence of local moves on the dominos themselves.
    \item And through a generalization of the `spider move' on the underlying dimer models (see Appendix \ref{sec:genShuf}).
\end{enumerate}
It is intriguing that despite the increased complexity of the $k$-tilings compared to standard tilings of the Aztec diamond, the shuffling algorithm remains very familiar. In terms of moves on the dominos, the increase in complexity is confined only to how new pairs of dominos are created. This has interesting combinitorial consequences. In fact, it can be used to give a alternate proof that the partition function of the $k$-tilings of rank $N$ is given by $Z^{(k)}_{AD} = \left(2(1+t)\right)^{\binom{N+1}{2}}$.

Using this shuffling algorithm allows fast sampling of the $k$-tilings.  In simulations one finds that these $k$-tilings appear to show very interesting asymptotic features, including the presence of arctic curves and limit shapes. See Appendix \ref{sec:sim}. One might hope that, just as for the standard tilings of the Aztec diamond, the shuffling algorithm may in the future be useful in the study of this asymptotic behavior.


\bibliographystyle{abbrv}
\bibliography{shuffle}

\appendix
\addcontentsline{toc}{section}{Appendices}

\section{Generalized spider move} \label{sec:genShuf}

\begin{figure}[ht!]
    \centering
    \includegraphics[scale=.8]{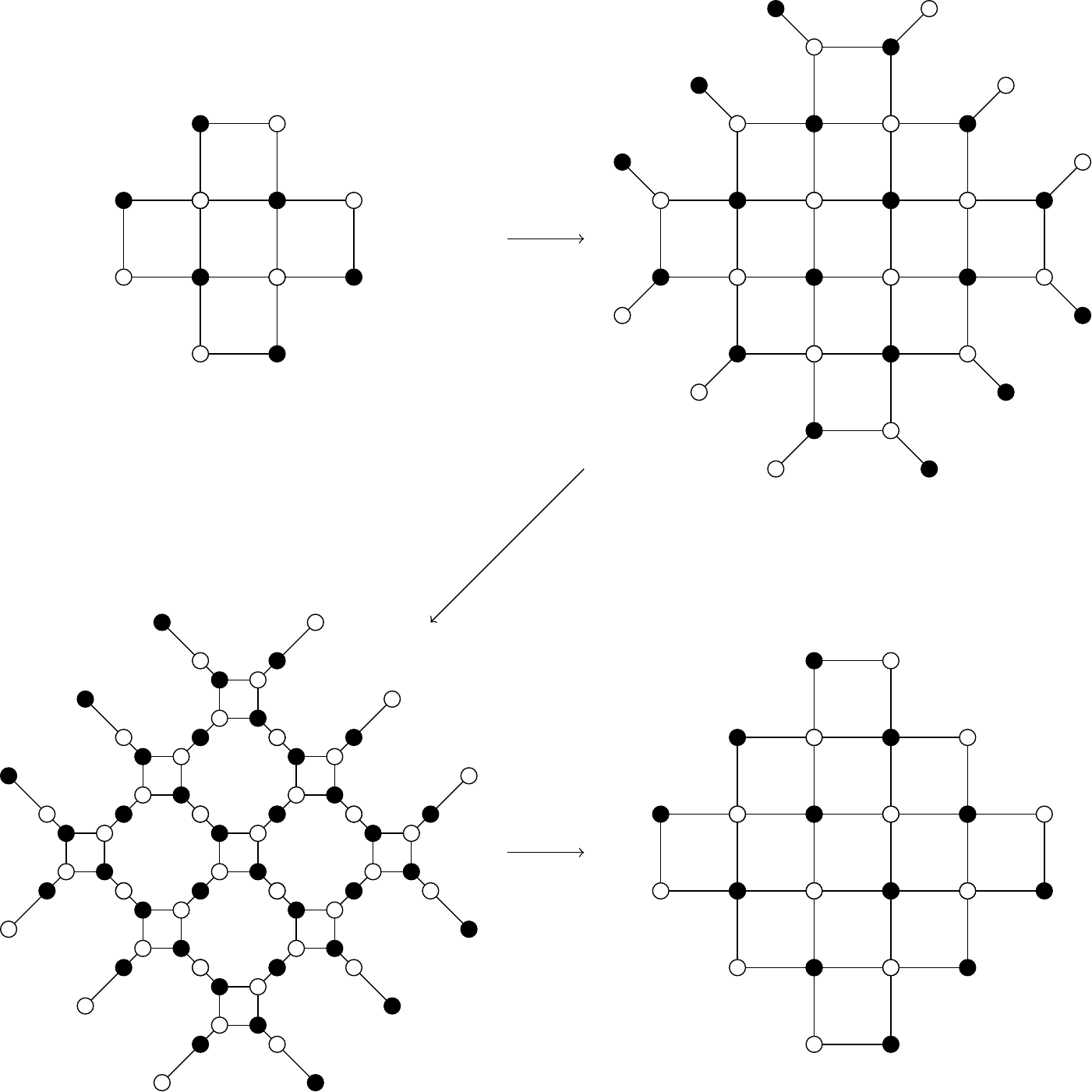}
    \caption{Starting with an Aztec diamond of rank~$2$ the above local moves construct an Aztec diamond of rank~$3$. The first transformation is a boundary decoration, the second consists of a collection of spider moves, and the final transformation is the contraction of degree two vertices. The randomization of these local moves leads to the Markov chain on dimer configurations known as the shuffling algorithm.}
    \label{fig:square_shuffling}
\end{figure}

\begin{figure}[ht!]
    \centering
    \includegraphics[scale=.8]{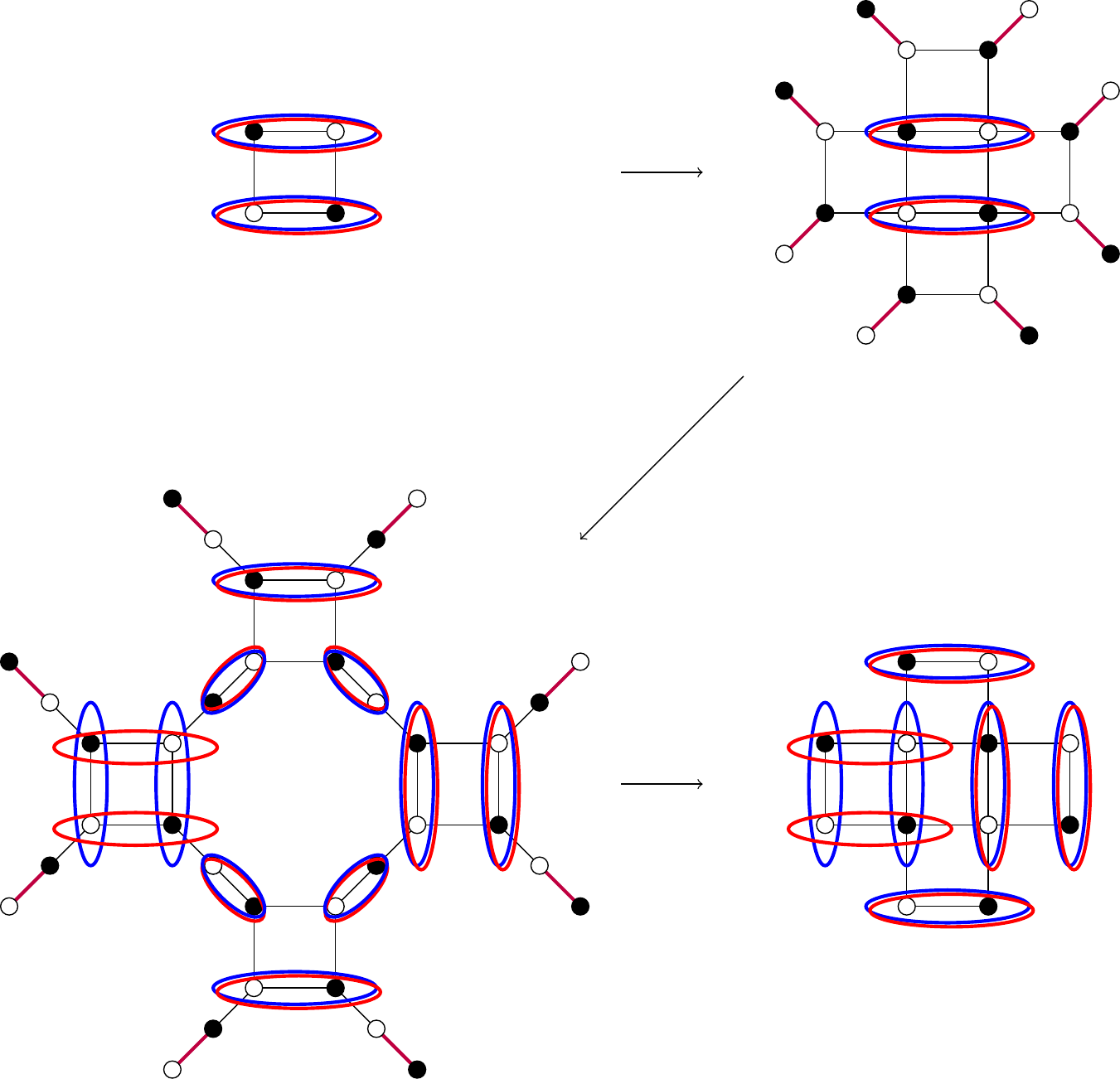}
    \caption{Starting with a double dimer configuration on the Aztec diamond of rank~$1$, the shuffling algorithm produces one on an Aztec diamond of rank~$2$. The weight of the configuration before the shuffling can be computed as the product over cells of cell weights in the top right picture, using the local interactions listed in~\eqref{eqn:local_before}. This gives a weight of~$t$. The weight of the configuration after the shuffling can be computed as the product of the local interactions listed in~\eqref{eqn:local_after} over the cells in the bottom left picture; this also gives a total weight of~$t$. One can check that for both the~$N = 1$ and~$N =2$ domino tiling corresponding to the top left and bottom right configurations, the weight is indeed~$t$, c.f. the domino interactions in Section~\ref{sec:couples_tilings}. }
    \label{fig:square_shuffling2}
\end{figure}

Here we will sketch an alternate proof of the correctness of the shuffling algorithm, Theorem~\ref{thm:basic}, using a generalization of the `spider move' on the underlying dimer model. See \cite{shuffling1, GoncharovKenyon2011DimersClusterIntSys} for details on the shuffling algorithm and the spider move in the one color case. Throughout this section we will assume the reader is familiar with the bijection between domino tilings and dimer covers, see~\cite{kenyon2003introduction}. For simplicity we will only consider the case where we have two colors, blue and red, with blue smaller than red.

We call a~\emph{cell} a face of the Aztec diamond which has a white vertex in the top left, as shown below
\begin{center}
\begin{tikzpicture}[baseline = (current bounding box).center]
\draw (0,0) rectangle (1,1);
\draw[dashed] (0,0)--(-0.5,-0.5);
\draw[dashed] (1,0)--(1.5,-0.5);
\draw[dashed] (0,1)--(-0.5,1.5);
\draw[dashed] (1,1)--(1.5,1.5);
\draw[fill=black] (0,0) circle (3pt);
\draw[fill=black] (1,1) circle (3pt);
\draw[fill=white] (1,0) circle (3pt);
\draw[fill=white] (0,1) circle (3pt);
\node[above] at (0.5,1) {$a$};
\node[right] at (1,0.5) {$b$};
\node[below] at (0.5,0) {$c$};
\node[left] at (0,0.5) {$d$};
\end{tikzpicture}.
\end{center}
In what follows, when we show a small patch of a graph, dashed lines mean the patch is a portion of a larger graph. The labels indicate the weight assigned to a dimer that occupies the corresponding edge. When it is clear from the context, we sometimes also use the term `cell' to refer to both the patch of graph and the dimers occupying edges of that face in a dimer configuration. As we are considering double dimer configurations, we will color the dimers red and blue to distinguish the two configurations.  

Define interactions to be local configurations of the form
\begin{equation}\label{eqn:local_before}
\begin{tabular}{ccc}
\begin{tikzpicture}[baseline = (current bounding box).center]
\draw (0,0) rectangle (1,1);
\draw[dashed] (0,0)--(-0.5,-0.5);
\draw[dashed] (1,0)--(1.5,-0.5);
\draw[dashed] (0,1)--(-0.5,1.5);
\draw[dashed] (1,1)--(1.5,1.5);
\draw[fill=black] (0,0) circle (3pt);
\draw[fill=black] (1,1) circle (3pt);
\draw[fill=white] (1,0) circle (3pt);
\draw[fill=white] (0,1) circle (3pt);
\draw[very thick,blue] (0.5,1) ellipse (1cm and 0.2cm);
\draw[very thick,red] (0.55,0.95) ellipse (1cm and 0.2cm);
\end{tikzpicture}
&
\begin{tikzpicture}[baseline = (current bounding box).center]
\draw (0,0) rectangle (1,1);
\draw[dashed] (0,0)--(-0.5,-0.5);
\draw[dashed] (1,0)--(1.5,-0.5);
\draw[dashed] (0,1)--(-0.5,1.5);
\draw[dashed] (1,1)--(1.5,1.5);
\draw[fill=black] (0,0) circle (3pt);
\draw[fill=black] (1,1) circle (3pt);
\draw[fill=white] (1,0) circle (3pt);
\draw[fill=white] (0,1) circle (3pt);
\draw[very thick,blue] (1,0.5) ellipse (0.2cm and 1cm);
\draw[very thick,red] (0.55,-0.05) ellipse (1cm and 0.2cm);
\end{tikzpicture}
&
\begin{tikzpicture}[baseline = (current bounding box).center]
\draw (0,0) rectangle (1,1);
\draw[dashed] (0,0)--(-0.5,-0.5);
\draw[dashed] (1,0)--(1.5,-0.5);
\draw[dashed] (0,1)--(-0.5,1.5);
\draw[dashed] (1,1)--(1.5,1.5);
\draw[fill=black] (0,0) circle (3pt);
\draw[fill=black] (1,1) circle (3pt);
\draw[fill=white] (1,0) circle (3pt);
\draw[fill=white] (0,1) circle (3pt);
\draw[very thick,blue] (0.5,1) ellipse (1cm and 0.2cm);
\draw[very thick,red,shift={(-0.25,1.25)},rotate=-45] (0,0) ellipse (0.5cm and 0.2cm);
\end{tikzpicture}
\end{tabular}
\end{equation}
A dimer occupying the dashed edge adjacent to a vertex $v$ denotes that a dimer occupies one of the other edges adjacent $v$ that are not part of the cell.

Note that if we consider the underlying dimer configuration of a $2$-tiling of the Aztec diamond, the product over all cells of~$t^{\text{\# interactions}}$ agrees with the product of the domino interactions defined in Section~\ref{sec:couples_tilings}. See Figure~\ref{fig:square_shuffling2}, top, for an example.

Now consider applying the spider move to the cell above. This results in a new patch of graph with different edge weights
\begin{center}
\begin{tikzpicture}[baseline = (current bounding box).center]
\draw (0,0) rectangle (1,1);
\draw[] (0,0)--(-0.5,-0.5);
\draw[] (1,0)--(1.5,-0.5);
\draw[] (0,1)--(-0.5,1.5);
\draw[] (1,1)--(1.5,1.5);
\draw[dashed] (-0.5,-0.5)--(-1,-1);
\draw[dashed] (1.5,-0.5)--(2,-1);
\draw[dashed] (-0.5,1.5)--(-1,2);
\draw[dashed] (1.5,1.5)--(2,2);
\draw[fill=white] (0,0) circle (3pt);
\draw[fill=white] (1,1) circle (3pt);
\draw[fill=black] (1,0) circle (3pt);
\draw[fill=black] (0,1) circle (3pt);
\draw[fill=black] (-0.5,-0.5) circle (3pt);
\draw[fill=black] (1.5,1.5) circle (3pt);
\draw[fill=white] (1.5,-0.5) circle (3pt);
\draw[fill=white] (-0.5,1.5) circle (3pt);
\node[above] at (0.5,1) {$a'$};
\node[right] at (1,0.5) {$b'$};
\node[below] at (0.5,0) {$c'$};
\node[left] at (0,0.5) {$d'$};
\end{tikzpicture}
\end{center}
where unlabelled edges have weight 1. After doing this local transformation to each cell in the Aztec diamond and then contracting all valence two vertices, one gets an Aztec diamond of one larger rank. We depict this process in Figure~\ref{fig:square_shuffling}. We will also refer to these patches as cells. 

After the spider moves, in the double dimer configuration we now count interactions of the form
\begin{equation}\label{eqn:local_after}
\begin{tabular}{ccc}
\begin{tikzpicture}[baseline = (current bounding box).center]
\draw (0,0) rectangle (1,1);
\draw[] (0,0)--(-0.5,-0.5);
\draw[] (1,0)--(1.5,-0.5);
\draw[] (0,1)--(-0.5,1.5);
\draw[] (1,1)--(1.5,1.5);
\draw[dashed] (-0.5,-0.5)--(-1,-1);
\draw[dashed] (1.5,-0.5)--(2,-1);
\draw[dashed] (-0.5,1.5)--(-1,2);
\draw[dashed] (1.5,1.5)--(2,2);
\draw[fill=white] (0,0) circle (3pt);
\draw[fill=white] (1,1) circle (3pt);
\draw[fill=black] (1,0) circle (3pt);
\draw[fill=black] (0,1) circle (3pt);
\draw[fill=black] (-0.5,-0.5) circle (3pt);
\draw[fill=black] (1.5,1.5) circle (3pt);
\draw[fill=white] (1.5,-0.5) circle (3pt);
\draw[fill=white] (-0.5,1.5) circle (3pt);
\draw[very thick,blue] (0.5,0) ellipse (1cm and 0.2cm);
\draw[very thick,red] (0.55,-0.05) ellipse (1cm and 0.2cm);
\end{tikzpicture}
&
\begin{tikzpicture}[baseline = (current bounding box).center]
\draw (0,0) rectangle (1,1);
\draw[] (0,0)--(-0.5,-0.5);
\draw[] (1,0)--(1.5,-0.5);
\draw[] (0,1)--(-0.5,1.5);
\draw[] (1,1)--(1.5,1.5);
\draw[dashed] (-0.5,-0.5)--(-1,-1);
\draw[dashed] (1.5,-0.5)--(2,-1);
\draw[dashed] (-0.5,1.5)--(-1,2);
\draw[dashed] (1.5,1.5)--(2,2);
\draw[fill=white] (0,0) circle (3pt);
\draw[fill=white] (1,1) circle (3pt);
\draw[fill=black] (1,0) circle (3pt);
\draw[fill=black] (0,1) circle (3pt);
\draw[fill=black] (-0.5,-0.5) circle (3pt);
\draw[fill=black] (1.5,1.5) circle (3pt);
\draw[fill=white] (1.5,-0.5) circle (3pt);
\draw[fill=white] (-0.5,1.5) circle (3pt);
\draw[very thick,blue] (0.5,0) ellipse (1cm and 0.2cm);
\draw[very thick,red] (-0.05,0.55) ellipse (0.2cm and 1cm);
\end{tikzpicture}
&
\begin{tikzpicture}[baseline = (current bounding box).center]
\draw (0,0) rectangle (1,1);
\draw[] (0,0)--(-0.5,-0.5);
\draw[] (1,0)--(1.5,-0.5);
\draw[] (0,1)--(-0.5,1.5);
\draw[] (1,1)--(1.5,1.5);
\draw[dashed] (-0.5,-0.5)--(-1,-1);
\draw[dashed] (1.5,-0.5)--(2,-1);
\draw[dashed] (-0.5,1.5)--(-1,2);
\draw[dashed] (1.5,1.5)--(2,2);
\draw[fill=white] (0,0) circle (3pt);
\draw[fill=white] (1,1) circle (3pt);
\draw[fill=black] (1,0) circle (3pt);
\draw[fill=black] (0,1) circle (3pt);
\draw[fill=black] (-0.5,-0.5) circle (3pt);
\draw[fill=black] (1.5,1.5) circle (3pt);
\draw[fill=white] (1.5,-0.5) circle (3pt);
\draw[fill=white] (-0.5,1.5) circle (3pt);
\draw[very thick,blue,shift={(1.25,1.25)},rotate=45] (0,0) ellipse (0.5cm and 0.2cm);
\draw[very thick,red] (0.55,0.95) ellipse (1cm and 0.2cm);
\end{tikzpicture}
\end{tabular}
\end{equation}
since the product over cells of these interactions will agree with the domino interactions we get after the contraction of degree two vertices. For an example compare the interactions in the configurations in the bottom row of Figure~\ref{fig:square_shuffling2}.

Note that for the usual dimer model, if
\[
a' = \frac{c}{ac+bd}, \qquad b' = \frac{d}{ac+bd}, \qquad c' = \frac{a}{ac+bd}, \qquad d' = \frac{b}{ac+bd}
\]
then for each choice of boundary condition for the patch, there is an equality of partition functions on the cell before and after the spider move, up to an overall factor of $\Delta = ac+bd$. To set some notation we list these equations below. Note that if~$x$ is a cell with a configuration of dimers, $w(x)$ denotes the product of weights of occupied edges, and the weights are implicitly updated as described above after the spider move.
\begin{center}
\resizebox{0.6\textheight}{!}{
\begin{tabular}{cc}
$w\left(\begin{tikzpicture}[baseline = (current bounding box).center]
\draw (0,0) rectangle (1,1);
\draw[dashed] (0,0)--(-0.5,-0.5);
\draw[dashed] (1,0)--(1.5,-0.5);
\draw[dashed] (0,1)--(-0.5,1.5);
\draw[dashed] (1,1)--(1.5,1.5);
\draw[fill=black] (0,0) circle (3pt);
\draw[fill=black] (1,1) circle (3pt);
\draw[fill=white] (1,0) circle (3pt);
\draw[fill=white] (0,1) circle (3pt);
\draw[very thick,red] (0.05,0.45) ellipse (0.2cm and 1cm);
\draw[very thick,red,shift={(1.25,1.25)},rotate=45] (0,0) ellipse (0.5cm and 0.2cm);
\draw[very thick,red,shift={(1.25,-0.25)},rotate=-45] (0,0) ellipse (0.5cm and 0.2cm);
\end{tikzpicture}\right)  = \Delta \times w\left(\begin{tikzpicture}[baseline = (current bounding box).center]
\draw (0,0) rectangle (1,1);
\draw[] (0,0)--(-0.5,-0.5);
\draw[] (1,0)--(1.5,-0.5);
\draw[] (0,1)--(-0.5,1.5);
\draw[] (1,1)--(1.5,1.5);
\draw[dashed] (-0.5,-0.5)--(-1,-1);
\draw[dashed] (1.5,-0.5)--(2,-1);
\draw[dashed] (-0.5,1.5)--(-1,2);
\draw[dashed] (1.5,1.5)--(2,2);
\draw[fill=white] (0,0) circle (3pt);
\draw[fill=white] (1,1) circle (3pt);
\draw[fill=black] (1,0) circle (3pt);
\draw[fill=black] (0,1) circle (3pt);
\draw[fill=black] (-0.5,-0.5) circle (3pt);
\draw[fill=black] (1.5,1.5) circle (3pt);
\draw[fill=white] (1.5,-0.5) circle (3pt);
\draw[fill=white] (-0.5,1.5) circle (3pt);
\draw[very thick,red] (0.95,0.55) ellipse (0.2cm and 1cm);
\draw[very thick,red,shift={(1.75,1.75)},rotate=45] (0,0) ellipse (0.5cm and 0.2cm);
\draw[very thick,red,shift={(1.75,-0.75)},rotate=-45] (0,0) ellipse (0.5cm and 0.2cm);
\draw[very thick,red,shift={(-0.25,-0.25)},rotate=45] (0,0) ellipse (0.5cm and 0.2cm);
\draw[very thick,red,shift={(-0.25,1.25)},rotate=-45] (0,0) ellipse (0.5cm and 0.2cm);
\end{tikzpicture} \right)$ & ``Right" \\
$w\left(\begin{tikzpicture}[baseline = (current bounding box).center]
\draw (0,0) rectangle (1,1);
\draw[dashed] (0,0)--(-0.5,-0.5);
\draw[dashed] (1,0)--(1.5,-0.5);
\draw[dashed] (0,1)--(-0.5,1.5);
\draw[dashed] (1,1)--(1.5,1.5);
\draw[fill=black] (0,0) circle (3pt);
\draw[fill=black] (1,1) circle (3pt);
\draw[fill=white] (1,0) circle (3pt);
\draw[fill=white] (0,1) circle (3pt);
\draw[very thick,red] (1.05,0.45) ellipse (0.2cm and 1cm);
\draw[very thick,red,shift={(-0.25,-0.25)},rotate=45] (0,0) ellipse (0.5cm and 0.2cm);
\draw[very thick,red,shift={(-0.25,1.25)},rotate=-45] (0,0) ellipse (0.5cm and 0.2cm);
\end{tikzpicture}\right) = \Delta \times w\left(\begin{tikzpicture}[baseline = (current bounding box).center]
\draw (0,0) rectangle (1,1);
\draw[] (0,0)--(-0.5,-0.5);
\draw[] (1,0)--(1.5,-0.5);
\draw[] (0,1)--(-0.5,1.5);
\draw[] (1,1)--(1.5,1.5);
\draw[dashed] (-0.5,-0.5)--(-1,-1);
\draw[dashed] (1.5,-0.5)--(2,-1);
\draw[dashed] (-0.5,1.5)--(-1,2);
\draw[dashed] (1.5,1.5)--(2,2);
\draw[fill=white] (0,0) circle (3pt);
\draw[fill=white] (1,1) circle (3pt);
\draw[fill=black] (1,0) circle (3pt);
\draw[fill=black] (0,1) circle (3pt);
\draw[fill=black] (-0.5,-0.5) circle (3pt);
\draw[fill=black] (1.5,1.5) circle (3pt);
\draw[fill=white] (1.5,-0.5) circle (3pt);
\draw[fill=white] (-0.5,1.5) circle (3pt);
\draw[very thick,red] (-0.05,0.55) ellipse (0.2cm and 1cm);
\draw[very thick,red,shift={(1.25,1.25)},rotate=45] (0,0) ellipse (0.5cm and 0.2cm);
\draw[very thick,red,shift={(1.25,-0.25)},rotate=-45] (0,0) ellipse (0.5cm and 0.2cm);
\draw[very thick,red,shift={(-0.75,-0.75)},rotate=45] (0,0) ellipse (0.5cm and 0.2cm);
\draw[very thick,red,shift={(-0.75,1.75)},rotate=-45] (0,0) ellipse (0.5cm and 0.2cm);
\end{tikzpicture} \right)$ & ``Left" \\
$w\left(\begin{tikzpicture}[baseline = (current bounding box).center]
\draw (0,0) rectangle (1,1);
\draw[dashed] (0,0)--(-0.5,-0.5);
\draw[dashed] (1,0)--(1.5,-0.5);
\draw[dashed] (0,1)--(-0.5,1.5);
\draw[dashed] (1,1)--(1.5,1.5);
\draw[fill=black] (0,0) circle (3pt);
\draw[fill=black] (1,1) circle (3pt);
\draw[fill=white] (1,0) circle (3pt);
\draw[fill=white] (0,1) circle (3pt);
\draw[very thick,red] (0.55,-0.05) ellipse (1cm and 0.2cm);
\draw[very thick,red,shift={(-0.25,1.25)},rotate=-45] (0,0) ellipse (0.5cm and 0.2cm);
\draw[very thick,red,shift={(1.25,1.25)},rotate=45] (0,0) ellipse (0.5cm and 0.2cm);
\end{tikzpicture}\right) = \Delta \times w\left(\begin{tikzpicture}[baseline = (current bounding box).center]
\draw (0,0) rectangle (1,1);
\draw[] (0,0)--(-0.5,-0.5);
\draw[] (1,0)--(1.5,-0.5);
\draw[] (0,1)--(-0.5,1.5);
\draw[] (1,1)--(1.5,1.5);
\draw[dashed] (-0.5,-0.5)--(-1,-1);
\draw[dashed] (1.5,-0.5)--(2,-1);
\draw[dashed] (-0.5,1.5)--(-1,2);
\draw[dashed] (1.5,1.5)--(2,2);
\draw[fill=white] (0,0) circle (3pt);
\draw[fill=white] (1,1) circle (3pt);
\draw[fill=black] (1,0) circle (3pt);
\draw[fill=black] (0,1) circle (3pt);
\draw[fill=black] (-0.5,-0.5) circle (3pt);
\draw[fill=black] (1.5,1.5) circle (3pt);
\draw[fill=white] (1.5,-0.5) circle (3pt);
\draw[fill=white] (-0.5,1.5) circle (3pt);
\draw[very thick,red] (0.55,0.95) ellipse (1cm and 0.2cm);
\draw[very thick,red,shift={(1.75,1.75)},rotate=45] (0,0) ellipse (0.5cm and 0.2cm);
\draw[very thick,red,shift={(-0.75,1.75)},rotate=-45] (0,0) ellipse (0.5cm and 0.2cm);
\draw[very thick,red,shift={(1.25,-0.25)},rotate=-45] (0,0) ellipse (0.5cm and 0.2cm);
\draw[very thick,red,shift={(-0.25,-0.25)},rotate=45] (0,0) ellipse (0.5cm and 0.2cm);
\end{tikzpicture} \right)$ & ``Up"\\
$w\left(\begin{tikzpicture}[baseline = (current bounding box).center]
\draw (0,0) rectangle (1,1);
\draw[dashed] (0,0)--(-0.5,-0.5);
\draw[dashed] (1,0)--(1.5,-0.5);
\draw[dashed] (0,1)--(-0.5,1.5);
\draw[dashed] (1,1)--(1.5,1.5);
\draw[fill=black] (0,0) circle (3pt);
\draw[fill=black] (1,1) circle (3pt);
\draw[fill=white] (1,0) circle (3pt);
\draw[fill=white] (0,1) circle (3pt);
\draw[very thick,red] (0.55,0.95) ellipse (1cm and 0.2cm);
\draw[very thick,red,shift={(1.25,-0.25)},rotate=-45] (0,0) ellipse (0.5cm and 0.2cm);
\draw[very thick,red,shift={(-0.25,-0.25)},rotate=45] (0,0) ellipse (0.5cm and 0.2cm);
\end{tikzpicture}\right) = \Delta \times w\left(\begin{tikzpicture}[baseline = (current bounding box).center]
\draw (0,0) rectangle (1,1);
\draw[] (0,0)--(-0.5,-0.5);
\draw[] (1,0)--(1.5,-0.5);
\draw[] (0,1)--(-0.5,1.5);
\draw[] (1,1)--(1.5,1.5);
\draw[dashed] (-0.5,-0.5)--(-1,-1);
\draw[dashed] (1.5,-0.5)--(2,-1);
\draw[dashed] (-0.5,1.5)--(-1,2);
\draw[dashed] (1.5,1.5)--(2,2);
\draw[fill=white] (0,0) circle (3pt);
\draw[fill=white] (1,1) circle (3pt);
\draw[fill=black] (1,0) circle (3pt);
\draw[fill=black] (0,1) circle (3pt);
\draw[fill=black] (-0.5,-0.5) circle (3pt);
\draw[fill=black] (1.5,1.5) circle (3pt);
\draw[fill=white] (1.5,-0.5) circle (3pt);
\draw[fill=white] (-0.5,1.5) circle (3pt);
\draw[very thick,red] (0.55,-0.05) ellipse (1cm and 0.2cm);
\draw[very thick,red,shift={(1.25,1.25)},rotate=45] (0,0) ellipse (0.5cm and 0.2cm);
\draw[very thick,red,shift={(-0.25,1.25)},rotate=-45] (0,0) ellipse (0.5cm and 0.2cm);
\draw[very thick,red,shift={(1.75,-0.75)},rotate=-45] (0,0) ellipse (0.5cm and 0.2cm);
\draw[very thick,red,shift={(-0.75,-0.75)},rotate=45] (0,0) ellipse (0.5cm and 0.2cm);
\end{tikzpicture} \right)$ & ``Down"\\
$w\left(\begin{tikzpicture}[baseline = (current bounding box).center]
\draw (0,0) rectangle (1,1);
\draw[dashed] (0,0)--(-0.5,-0.5);
\draw[dashed] (1,0)--(1.5,-0.5);
\draw[dashed] (0,1)--(-0.5,1.5);
\draw[dashed] (1,1)--(1.5,1.5);
\draw[fill=black] (0,0) circle (3pt);
\draw[fill=black] (1,1) circle (3pt);
\draw[fill=white] (1,0) circle (3pt);
\draw[fill=white] (0,1) circle (3pt);
\draw[very thick,red,shift={(-0.25,1.25)},rotate=-45] (0,0) ellipse (0.5cm and 0.2cm);
\draw[very thick,red,shift={(1.25,1.25)},rotate=45] (0,0) ellipse (0.5cm and 0.2cm);
\draw[very thick,red,shift={(1.25,-0.25)},rotate=-45] (0,0) ellipse (0.5cm and 0.2cm);
\draw[very thick,red,shift={(-0.25,-0.25)},rotate=45] (0,0) ellipse (0.5cm and 0.2cm);
\end{tikzpicture}\right) = \Delta \times \left(w\left(\begin{tikzpicture}[baseline = (current bounding box).center]
\draw (0,0) rectangle (1,1);
\draw[] (0,0)--(-0.5,-0.5);
\draw[] (1,0)--(1.5,-0.5);
\draw[] (0,1)--(-0.5,1.5);
\draw[] (1,1)--(1.5,1.5);
\draw[dashed] (-0.5,-0.5)--(-1,-1);
\draw[dashed] (1.5,-0.5)--(2,-1);
\draw[dashed] (-0.5,1.5)--(-1,2);
\draw[dashed] (1.5,1.5)--(2,2);
\draw[fill=white] (0,0) circle (3pt);
\draw[fill=white] (1,1) circle (3pt);
\draw[fill=black] (1,0) circle (3pt);
\draw[fill=black] (0,1) circle (3pt);
\draw[fill=black] (-0.5,-0.5) circle (3pt);
\draw[fill=black] (1.5,1.5) circle (3pt);
\draw[fill=white] (1.5,-0.5) circle (3pt);
\draw[fill=white] (-0.5,1.5) circle (3pt);
\draw[very thick,red] (0.55,-0.05) ellipse (1cm and 0.2cm);
\draw[very thick,red] (0.55,0.95) ellipse (1cm and 0.2cm);
\draw[very thick,red,shift={(1.75,1.75)},rotate=45] (0,0) ellipse (0.5cm and 0.2cm);
\draw[very thick,red,shift={(-0.75,1.75)},rotate=-45] (0,0) ellipse (0.5cm and 0.2cm);
\draw[very thick,red,shift={(1.75,-0.75)},rotate=-45] (0,0) ellipse (0.5cm and 0.2cm);
\draw[very thick,red,shift={(-0.75,-0.75)},rotate=45] (0,0) ellipse (0.5cm and 0.2cm);
\end{tikzpicture} \right) + w\left(\begin{tikzpicture}[baseline = (current bounding box).center]
\draw (0,0) rectangle (1,1);
\draw[] (0,0)--(-0.5,-0.5);
\draw[] (1,0)--(1.5,-0.5);
\draw[] (0,1)--(-0.5,1.5);
\draw[] (1,1)--(1.5,1.5);
\draw[dashed] (-0.5,-0.5)--(-1,-1);
\draw[dashed] (1.5,-0.5)--(2,-1);
\draw[dashed] (-0.5,1.5)--(-1,2);
\draw[dashed] (1.5,1.5)--(2,2);
\draw[fill=white] (0,0) circle (3pt);
\draw[fill=white] (1,1) circle (3pt);
\draw[fill=black] (1,0) circle (3pt);
\draw[fill=black] (0,1) circle (3pt);
\draw[fill=black] (-0.5,-0.5) circle (3pt);
\draw[fill=black] (1.5,1.5) circle (3pt);
\draw[fill=white] (1.5,-0.5) circle (3pt);
\draw[fill=white] (-0.5,1.5) circle (3pt);
\draw[very thick,red] (-0.05,0.55) ellipse (0.2cm and 1cm);
\draw[very thick,red] (0.95,0.55) ellipse (0.2cm and 1cm);
\draw[very thick,red,shift={(1.75,1.75)},rotate=45] (0,0) ellipse (0.5cm and 0.2cm);
\draw[very thick,red,shift={(-0.75,1.75)},rotate=-45] (0,0) ellipse (0.5cm and 0.2cm);
\draw[very thick,red,shift={(1.75,-0.75)},rotate=-45] (0,0) ellipse (0.5cm and 0.2cm);
\draw[very thick,red,shift={(-0.75,-0.75)},rotate=45] (0,0) ellipse (0.5cm and 0.2cm);
\end{tikzpicture} \right)\right)$ & ``Creation"\\
$w\left(\begin{tikzpicture}[baseline = (current bounding box).center]
\draw (0,0) rectangle (1,1);
\draw[dashed] (0,0)--(-0.5,-0.5);
\draw[dashed] (1,0)--(1.5,-0.5);
\draw[dashed] (0,1)--(-0.5,1.5);
\draw[dashed] (1,1)--(1.5,1.5);
\draw[fill=black] (0,0) circle (3pt);
\draw[fill=black] (1,1) circle (3pt);
\draw[fill=white] (1,0) circle (3pt);
\draw[fill=white] (0,1) circle (3pt);
\draw[very thick,red] (0.55,0.95) ellipse (1cm and 0.2cm);
\draw[very thick,red] (0.55,-0.05) ellipse (1cm and 0.2cm);
\end{tikzpicture}\right) + w\left(\begin{tikzpicture}[baseline = (current bounding box).center]
\draw (0,0) rectangle (1,1);
\draw[dashed] (0,0)--(-0.5,-0.5);
\draw[dashed] (1,0)--(1.5,-0.5);
\draw[dashed] (0,1)--(-0.5,1.5);
\draw[dashed] (1,1)--(1.5,1.5);
\draw[fill=black] (0,0) circle (3pt);
\draw[fill=black] (1,1) circle (3pt);
\draw[fill=white] (1,0) circle (3pt);
\draw[fill=white] (0,1) circle (3pt);
\draw[very thick,red] (0.05,0.45) ellipse (0.2cm and 1cm);
\draw[very thick,red] (1.05,0.45) ellipse (0.2cm and 1cm);
\end{tikzpicture}\right)= \Delta \times w\left(\begin{tikzpicture}[baseline = (current bounding box).center]
\draw (0,0) rectangle (1,1);
\draw[] (0,0)--(-0.5,-0.5);
\draw[] (1,0)--(1.5,-0.5);
\draw[] (0,1)--(-0.5,1.5);
\draw[] (1,1)--(1.5,1.5);
\draw[dashed] (-0.5,-0.5)--(-1,-1);
\draw[dashed] (1.5,-0.5)--(2,-1);
\draw[dashed] (-0.5,1.5)--(-1,2);
\draw[dashed] (1.5,1.5)--(2,2);
\draw[fill=white] (0,0) circle (3pt);
\draw[fill=white] (1,1) circle (3pt);
\draw[fill=black] (1,0) circle (3pt);
\draw[fill=black] (0,1) circle (3pt);
\draw[fill=black] (-0.5,-0.5) circle (3pt);
\draw[fill=black] (1.5,1.5) circle (3pt);
\draw[fill=white] (1.5,-0.5) circle (3pt);
\draw[fill=white] (-0.5,1.5) circle (3pt);
\draw[very thick,red,shift={(1.25,1.25)},rotate=45] (0,0) ellipse (0.5cm and 0.2cm);
\draw[very thick,red,shift={(-0.25,1.25)},rotate=-45] (0,0) ellipse (0.5cm and 0.2cm);
\draw[very thick,red,shift={(1.25,-0.25)},rotate=-45] (0,0) ellipse (0.5cm and 0.2cm);
\draw[very thick,red,shift={(-0.25,-0.25)},rotate=45] (0,0) ellipse (0.5cm and 0.2cm);
\end{tikzpicture} \right)$ & ``Destruction"\\
\end{tabular}
}
\end{center}
We label the boundary condition by the type of move that it corresponds in the shuffling when going from the domain on the LHS to the domain on the RHS. In the equations that follow, we will label a boundary condition for a cell by an arrow to indicate left/right/up/down, a `$c$' for creation, or a `$d$' for destruction.

For the double dimer model the situation is more complicated. A boundary condition~$(\alpha \beta)$ for a cell consists of a boundary condition~$\alpha$ for the blue configuration and a boundary condition~$\beta$ for the red one. Define~$C$ as the set of boundary conditions~$(\alpha \beta)$ for a cell such that 
\begin{itemize}
    \item $\alpha = c$ and~$\beta \in \{c, \leftarrow, \downarrow\} $ or
    \item  $\alpha  \in \{c, \leftarrow, \downarrow\}$ and~$\beta = c $
\end{itemize}
and define~$D$ as the set of boundary conditions~$(\alpha \beta)$ such that
\begin{itemize}
    \item $\alpha = d$ and~$\beta \in \{d, \leftarrow, \downarrow\} $ or
    \item  $\alpha  \in \{d, \leftarrow, \downarrow\}$ and~$\beta = d $.
\end{itemize}
We have the following relation between partition functions before and after the spider move.
\begin{lem} \label{lem:2spider}
    Let  the weights after the spider move be
    \[
    a' = \frac{c}{ac+bd}, \qquad b' = \frac{d}{ac+bd}, \qquad c' = \frac{a}{ac+bd}, \qquad d' = \frac{b}{ac+bd}.
    \]
    For any pair of boundary conditions~$\alpha, \beta \in \{c,d,\rightarrow,\leftarrow,\uparrow,\downarrow\}$ for each color, denote by~$Z_{\alpha\beta}$ the partition function of the domain with these boundary conditions before the spider move, and~$Z'_{\alpha \beta}$ that of the domain after the spider move. Then
    \begin{align}
         Z_{\alpha \beta} =& \Delta^2 \Gamma Z_{\alpha \beta}', \qquad (\alpha \beta) \in C \\
         Z_{\alpha \beta} =& \Delta^2 \Gamma^{-1} Z_{\alpha \beta}' \qquad (\alpha \beta) \in D \\
        Z_{\alpha\beta} =& \Delta^2 Z_{\alpha\beta}' \qquad \text{o.w.}
    \end{align}
   where $\Delta = ac+bd $ and $\Gamma = \frac{ac+bd}{act+bd}$.
\end{lem}
\begin{proof}
    As there are only 36 choices of boundary condition, one can check (by hand or via computer) that the required 36 equations are satisfied. 
\end{proof}
As an example, $Z_{d\downarrow}$ means the partition function for the domain where the smaller color has the ``Destruction" boundary condition while the larger color has the ``Down" boundary condition. For boundary conditions of type $(d,\downarrow)$ one can check
\[
\begin{aligned}
w\left(
\resizebox{2cm}{!}{
\begin{tikzpicture}[baseline = (current bounding box).center]
\draw (0,0) rectangle (1,1);
\draw[dashed] (0,0)--(-0.5,-0.5);
\draw[dashed] (1,0)--(1.5,-0.5);
\draw[dashed] (0,1)--(-0.5,1.5);
\draw[dashed] (1,1)--(1.5,1.5);
\draw[fill=black] (0,0) circle (3pt);
\draw[fill=black] (1,1) circle (3pt);
\draw[fill=white] (1,0) circle (3pt);
\draw[fill=white] (0,1) circle (3pt);
\draw[very thick,blue] (0.5,1) ellipse (1cm and 0.2cm);
\draw[very thick,blue] (0.5,0) ellipse (1cm and 0.2cm);
\draw[very thick,red] (0.55,0.95) ellipse (1cm and 0.2cm);
\draw[very thick,red,shift={(1.25,-0.25)},rotate=-45] (0,0) ellipse (0.5cm and 0.2cm);
\draw[very thick,red,shift={(-0.25,-0.25)},rotate=45] (0,0) ellipse (0.5cm and 0.2cm);
\end{tikzpicture}
}
\right) + w\left(
\resizebox{2cm}{!}{
\begin{tikzpicture}[baseline = (current bounding box).center]
\draw (0,0) rectangle (1,1);
\draw[dashed] (0,0)--(-0.5,-0.5);
\draw[dashed] (1,0)--(1.5,-0.5);
\draw[dashed] (0,1)--(-0.5,1.5);
\draw[dashed] (1,1)--(1.5,1.5);
\draw[fill=black] (0,0) circle (3pt);
\draw[fill=black] (1,1) circle (3pt);
\draw[fill=white] (1,0) circle (3pt);
\draw[fill=white] (0,1) circle (3pt);
\draw[very thick,blue] (0,0.5) ellipse (0.2cm and 1cm);
\draw[very thick,blue] (1,0.5) ellipse (0.2cm and 1cm);
\draw[very thick,red] (0.55,0.95) ellipse (1cm and 0.2cm);
\draw[very thick,red,shift={(1.25,-0.25)},rotate=-45] (0,0) ellipse (0.5cm and 0.2cm);
\draw[very thick,red,shift={(-0.25,-0.25)},rotate=45] (0,0) ellipse (0.5cm and 0.2cm);
\end{tikzpicture}
}
\right)=& \frac{\Delta^2}{\Gamma} \times w\left(
\resizebox{2cm}{!}{
\begin{tikzpicture}[baseline = (current bounding box).center]
\draw (0,0) rectangle (1,1);
\draw[] (0,0)--(-0.5,-0.5);
\draw[] (1,0)--(1.5,-0.5);
\draw[] (0,1)--(-0.5,1.5);
\draw[] (1,1)--(1.5,1.5);
\draw[dashed] (-0.5,-0.5)--(-1,-1);
\draw[dashed] (1.5,-0.5)--(2,-1);
\draw[dashed] (-0.5,1.5)--(-1,2);
\draw[dashed] (1.5,1.5)--(2,2);
\draw[fill=white] (0,0) circle (3pt);
\draw[fill=white] (1,1) circle (3pt);
\draw[fill=black] (1,0) circle (3pt);
\draw[fill=black] (0,1) circle (3pt);
\draw[fill=black] (-0.5,-0.5) circle (3pt);
\draw[fill=black] (1.5,1.5) circle (3pt);
\draw[fill=white] (1.5,-0.5) circle (3pt);
\draw[fill=white] (-0.5,1.5) circle (3pt);
\draw[very thick,blue,shift={(1.25,1.25)},rotate=45] (0,0) ellipse (0.5cm and 0.2cm);
\draw[very thick,blue,shift={(-0.25,1.25)},rotate=-45] (0,0) ellipse (0.5cm and 0.2cm);
\draw[very thick,blue,shift={(1.25,-0.25)},rotate=-45] (0,0) ellipse (0.5cm and 0.2cm);
\draw[very thick,blue,shift={(-0.25,-0.25)},rotate=45] (0,0) ellipse (0.5cm and 0.2cm);
\draw[very thick,red] (0.55,-0.05) ellipse (1cm and 0.2cm);
\draw[very thick,red,shift={(1.3,1.2)},rotate=45] (0,0) ellipse (0.5cm and 0.2cm);
\draw[very thick,red,shift={(-0.3,1.2)},rotate=-45] (0,0) ellipse (0.5cm and 0.2cm);
\draw[very thick,red,shift={(1.75,-0.75)},rotate=-45] (0,0) ellipse (0.5cm and 0.2cm);
\draw[very thick,red,shift={(-0.75,-0.75)},rotate=45] (0,0) ellipse (0.5cm and 0.2cm);
\end{tikzpicture} 
}
\right) \\
\implies a^2ct + abd =& \frac{(ac+bd)^2(act+bd)}{ac+bd} c'
\end{aligned}
\]
which is consistent as $c' = \frac{a}{ac+bd}$.

The following combinatorial lemma, whose proof we omit, will be useful.
\begin{lem} \label{lem:gamcount}
   For any double dimer configuration on the Aztec diamond of rank~$N$, along each SW-NE diagonal of cells the difference between the number of cells with local boundary condition of type $(\alpha \beta) \in C$ and those of type $(\alpha \beta ) \in D$ is equal to $1$.
\end{lem}

Suppose that in a double dimer configuration, we have a cell~$x$ with boundary condition~$(\alpha \beta)$. We define the transition probability~$\mathbb{P}(x \rightarrow x')$ from~$x$ to a double dimer configuration~$x'$ in the cell after performing the spider move as
$$ \mathbb{P}(x \rightarrow x') = 
\begin{cases}
\frac{w(x')}{Z_{\alpha \beta}'} & \text{if~$x'$ has boundary condition } (\alpha \beta) \\
0 & \text{otherwise}
\end{cases} $$
where~$Z_{\alpha \beta}'$ denotes the partition function in this cell after the spider move with the boundary conditions~$(\alpha \beta)$, as in Lemma~\ref{lem:2spider}. 

Now we define the shuffling algorithm of Theorem~\ref{thm:basic} in terms of local moves; we will apply the spider move to each cell of the Aztec diamond and re-sample the double dimer configuration in each cell. In more detail, suppose that~$T^{(N+1)}$ is a double dimer configuration of the rank-$(N+1)$ Aztec diamond. Note that~$T^{(N+1)}$ corresponds uniquely to a double dimer configuration on the graph obtained from the rank-$N$ Aztec diamond by decorating the boundary and performing spider moves at each cell (see Figure~\ref{fig:square_shuffling}, bottom left). We denote by~$x^{(N+1)} $ the local configuration in each cell of this graph. For a double dimer cover~$T^{(N)}$ of the rank-$N$ Aztec diamond, we define 
$$\mathbb{P}(T^{(N)}\to T^{(N+1)}) = \prod_{\text{cells x}} \mathbb{P}(x^{(N)} \rightarrow x^{(N+1)})$$
where~$x^{(N)}$ is the configuration in cell~$x$ in~$T^{(N)}$ and~$x^{(N+1)}$ is as described above. If all weights are constant, which is the uniform case, the transition probabilities defined above coincide with those of Theorem~\ref{thm:basic}.

\begin{proof}[Alternate proof of Thm. \ref{thm:basic} for the case of two colors]
Consider a random $2$-tiling $T^{(N)}$ of the Aztec diamond of rank $N$. The tiling is made up of gluing together dimer configurations in each cell. The weight of a $2$-tiling is given by product over cells of the weight in each cell. Let $Z^{(N)}$ be the partition function for these $2$-tilings. 

By applying the spider move to each cell and then contracting all valence two vertices we get a $2$-tiling of the rank-$(N+1)$ Aztec diamond (see Fig. \ref{fig:square_shuffling}). The probability to obtain a specific $2$-tiling $T^{(N+1)}$, whose configuration in cell~$x$ before the contraction step is denoted~$x^{(N+1)}$, is given by
\begin{equation*}
\begin{aligned}
    \mathbb{P}(T^{(N+1)}) = & \sum_{T^{(N)}}\mathbb{P}(T^{(N)})\mathbb{P}(T^{(N)}\to T^{(N+1)}) \\
    =& \sum_{T^{(N)}} \frac{w(T^{(N)})}{Z^{(N)}}\mathbb{P}(T^{(N)}\to T^{(N+1)}) \\
    =& \frac{1}{Z^{(N)}} \sum_{T^{(N)}} \prod_{\text{cells $x$ of $T^{(N)}$}} w(x^{(N)}) \mathbb{P}(T^{(N)}\to T^{(N+1)}) \\
    =& \frac{1}{Z^{(N)}}  \prod_{\text{cells $x$}}\left( \sum_{\substack{\text{local configs in } x \\
    \text{ consistent with $T^{(N+1)}$}}}  w(x^{(N)}) \mathbb{P}(x^{(N)} \to x^{(N+1)}) \right)  \\
    =& \frac{1}{Z^{(N)}} \prod_{\text{cells $x$}} \Delta^2 \;  w(x^{(N+1)}) \prod_{\substack{x \text{ with} \\ \text{ b.c. in } D}} \Gamma^{-1}(x) \prod_{\substack{x \text{ with} \\ \text{ b.c. in } C}} \Gamma(x).
\end{aligned}
\end{equation*}
In the final line~$\Gamma(x)$ denotes the value of~$\Gamma$ for the weights~$a,b,c,d$ of the cell~$x$, and we have implicitly updated the weights in the last line. In the last equality we use the relations from Lemma \ref{lem:2spider}.

If $\Gamma$ is constant, that is, $\Gamma(x) = r$ for all cells $x$, then using Lemma \ref{lem:gamcount} we see that
\begin{equation}
    \mathbb{P}(T^{(N+1)}) = \frac{r^{N+1}}{Z^{(N)}} \prod_{\text{cells $x$}} \Delta^2 \;  w(x^{(N+1)}) 
\end{equation}
that is, the probability we get a $2$-tiling $T^{(N+1)}$ via our shuffling algorithm is proportional to the weight of that $2$-tiling, $w(T^{(N+1)}) = \prod_{\text{cells $x$}} w(x^{(N+1)})$. 

Note that if we choose uniform weights, so that in each cell $a=b=c=d$, then they remain uniform after each step of the shuffling. Furthermore, for this choice of weights, $\Gamma = \frac{2}{1+t}$. It follows that the shuffling algorithm works in this case.
\end{proof}

\begin{remark}
    While for the standard ($1$-tiling) domino shuffling algorithm the above argument allows for arbitrary edge weights, in the $2$-tiling case there are extra restrictions on the weights. Note that for the shuffling to work, it is sufficient for the $\Gamma(x)$ to be constant along each diagonal at each step of the algorithm. That is, after using the shuffling to go from rank $N$ to rank $N+1$, the updated $\Gamma(x)$'s, defined using the updated weights, should still remain constant as the cell~$x$ ranges along a SW-NE diagonal. This imposes much stricter constraints on the edge weights.
\end{remark}

\section{Simulations} \label{sec:sim}
In this section, we present several simulations of the coupled tilings made using using the shuffling algorithm.

Notice that the tilings appear to exhibit a limit shape phenomenon, with regions of frozen dominos separated from a disorder region by an arctic curve. When $t<1$ the arctic curves appear to have a cusp along the SE boundary, for example see Figures \ref{fig:t0.2b1}, \ref{fig:t0.2b0.5}, \ref{fig:t0.2b2}, \ref{fig:t0.2b2c3}. When $t>1$ the cusp appears along the NW boundary instead, see Figures \ref{fig:t5b2} and \ref{fig:t10000b1}. Note when $t=0$, southern region of frozen horizontal dominos vanishes and the frozen region in the cusp becomes large, see Figure \ref{fig:t0b5}. When $t\to \infty$ the western region of frozen vertical dominos vanishes and is replaced by the frozen region from the cusp, see Figure \ref{fig:t10000b1}. In simulations with three colors there seems to be two cusps appearing, Figure \ref{fig:t0.2b2c3}. The arctic curves for $t=0,\infty$ can in fact be computed, similarly to what was done in Theorem 5.4 of \cite{LLTaztec}.

Surprisingly, the large scale features, such as arctic curves, always seem to look the same for each color, despite the fact that the coupling is not symmetric between the colors. Of course for $t=1$, Figure~\ref{fig:t1b1c3}, this symmetry follows from the fact that the tilings are independent. For $t=0,\infty$, Figures \ref{fig:t0b5} and \ref{fig:t10000b1}, arguments similar to those in Theorem 5.4 and Lemma 6.1 of \cite{LLTaztec} show that asymptotically the arctic curves are in fact the same. It should be possible to strengthen the cited theorem to show that the limiting height functions are also the same. However, for the intermediate values of $t$ it is unclear why there should exist such a symmetry. 

Comparing Figures \ref{fig:t0.2b0.5} and \ref{fig:t5b2} we see that there appears to be a symmetry with respect to taking $t,b,c\mapsto \frac{1}{t},\frac{1}{b},\frac{1}{c}$. This follows from a similar argument to Lemma 6.1 of \cite{LLTaztec}.

A more detailed numerical analysis of the limit shapes and their fluctuations using the shuffling algorithm to generate random samples would be an interesting follow-up to the current work.

\begin{remark}
    In \cite{LLTaztec} they define two different versions of the coupled tilings which they call the ``purple-gray" model and the ``white-pink" model. They then prove results about the purple-gray model. In this article, we study the white-pink model, so the results do not immediately apply. However, one can show, by very similar arguments, that analogous results to Theorem 5.4 and Lemma 6.1 of \cite{LLTaztec} hold in the white-pink model as well.  
\end{remark}

\begin{figure}[!htb]
    \centering
    \includegraphics[width=\textwidth]{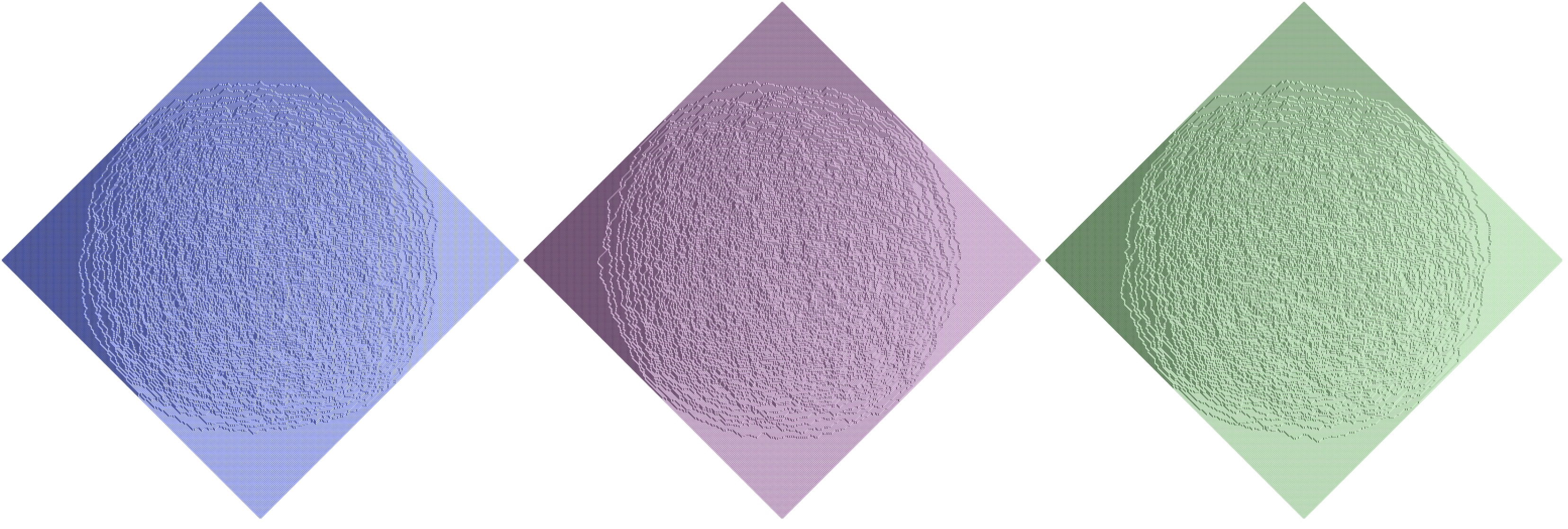}
        
    \caption{A $3$-tiling of rank $N=256$ generated by the shuffling algorithm with $t=1$ and $c_i=b_i=1$ for all $i=1,\ldots,N$.}
    \label{fig:t1b1c3}
\end{figure}

\begin{figure}[!htb]
    \centering
    \includegraphics[width=0.8\textwidth]{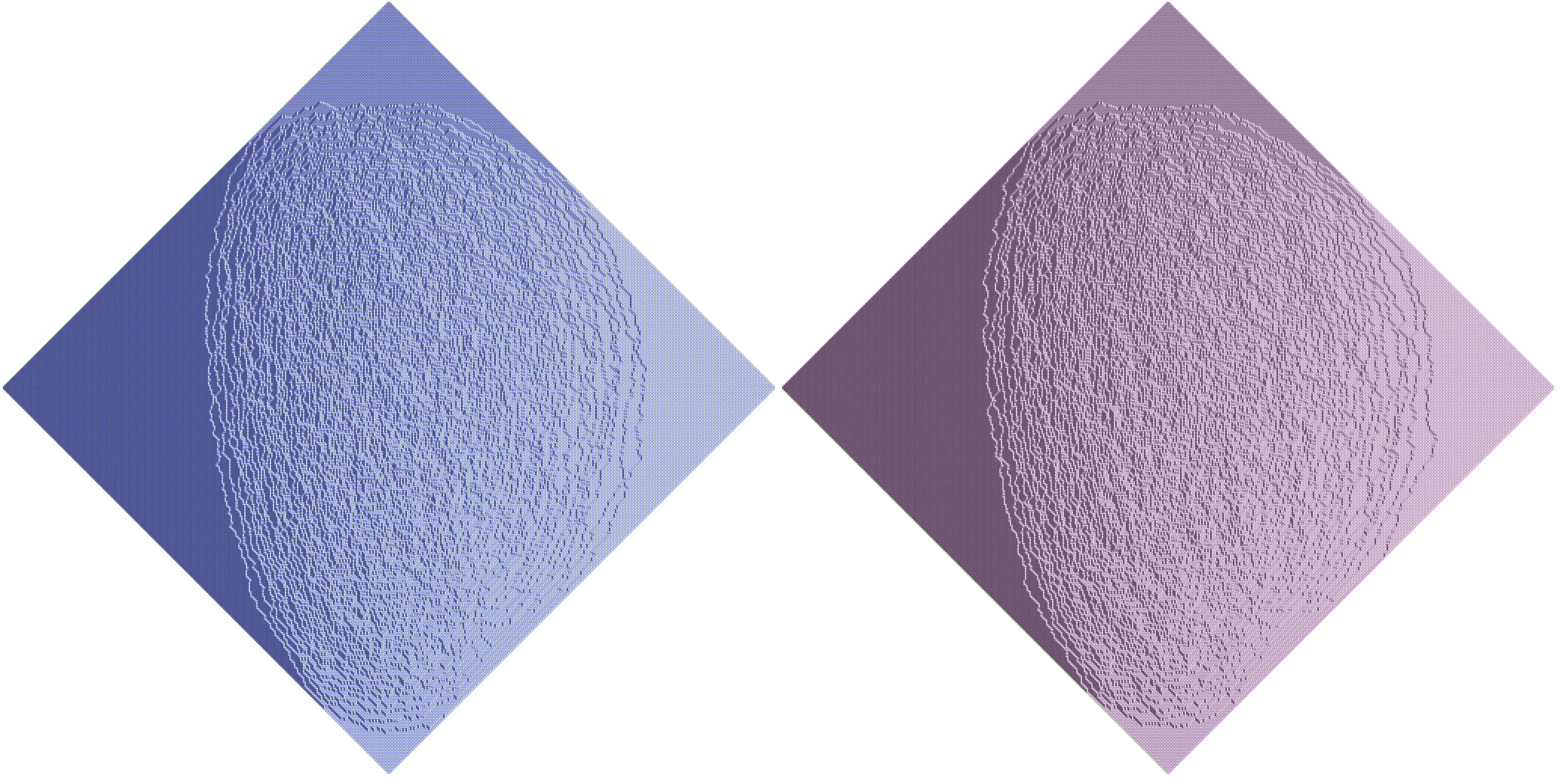}
        
    \caption{A $2$-tiling of rank $N=256$ generated by the shuffling algorithm with $t=0.2$ and $c_i=b_i=1$ for all $i=1,\ldots,N$.}
    \label{fig:t0.2b1}
\end{figure}

\begin{figure}[!htb]
    \centering
    \includegraphics[width=0.8\textwidth]{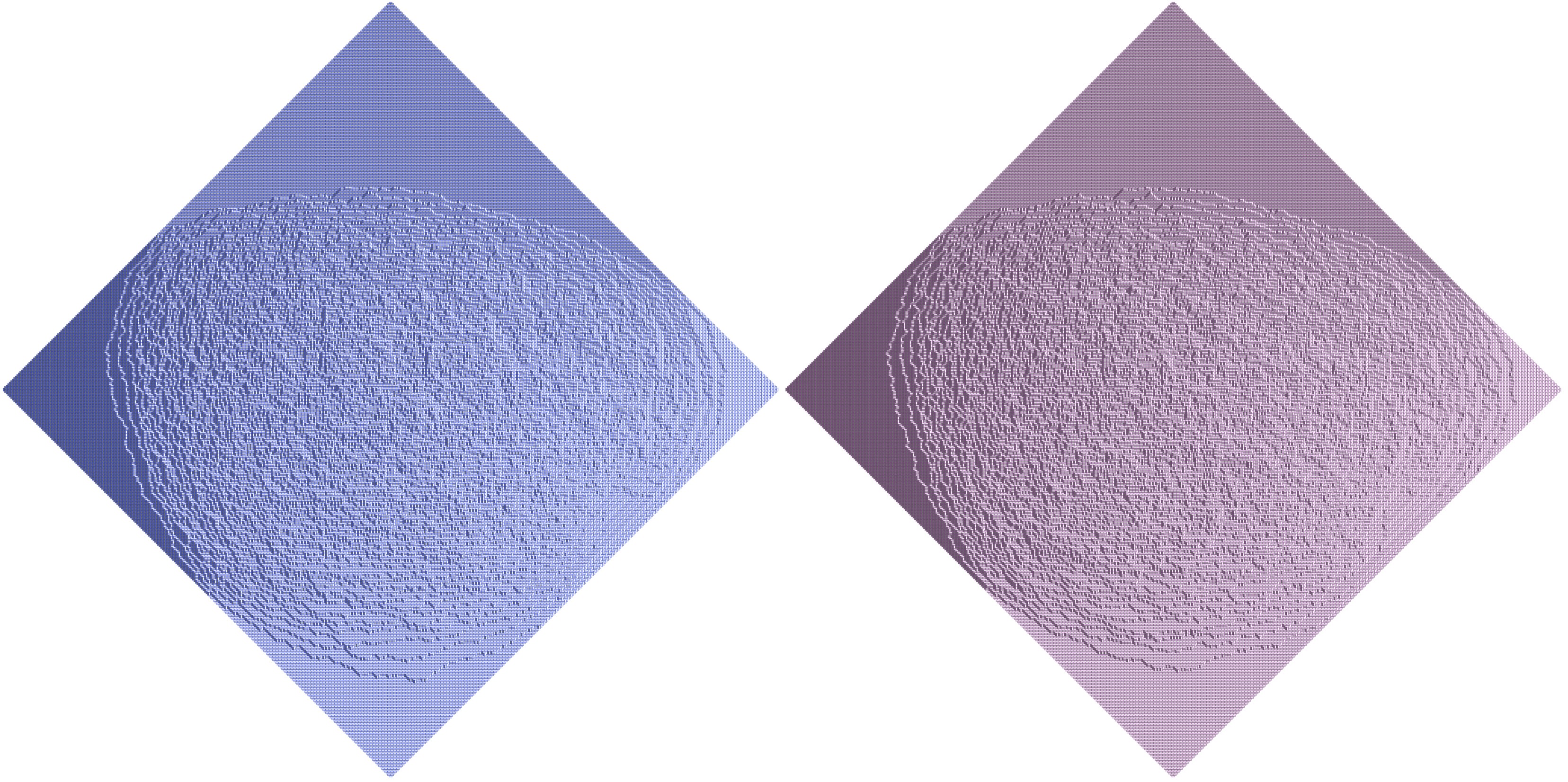}
        
    \caption{A $2$-tiling of rank $N=256$ generated by the shuffling algorithm with $t=0.2$ and $c_i=b_i=2$ for all $i=1,\ldots,N$.}
    \label{fig:t0.2b2}
\end{figure}

\begin{figure}[!htb]
    \centering
    \includegraphics[width=0.8\textwidth]{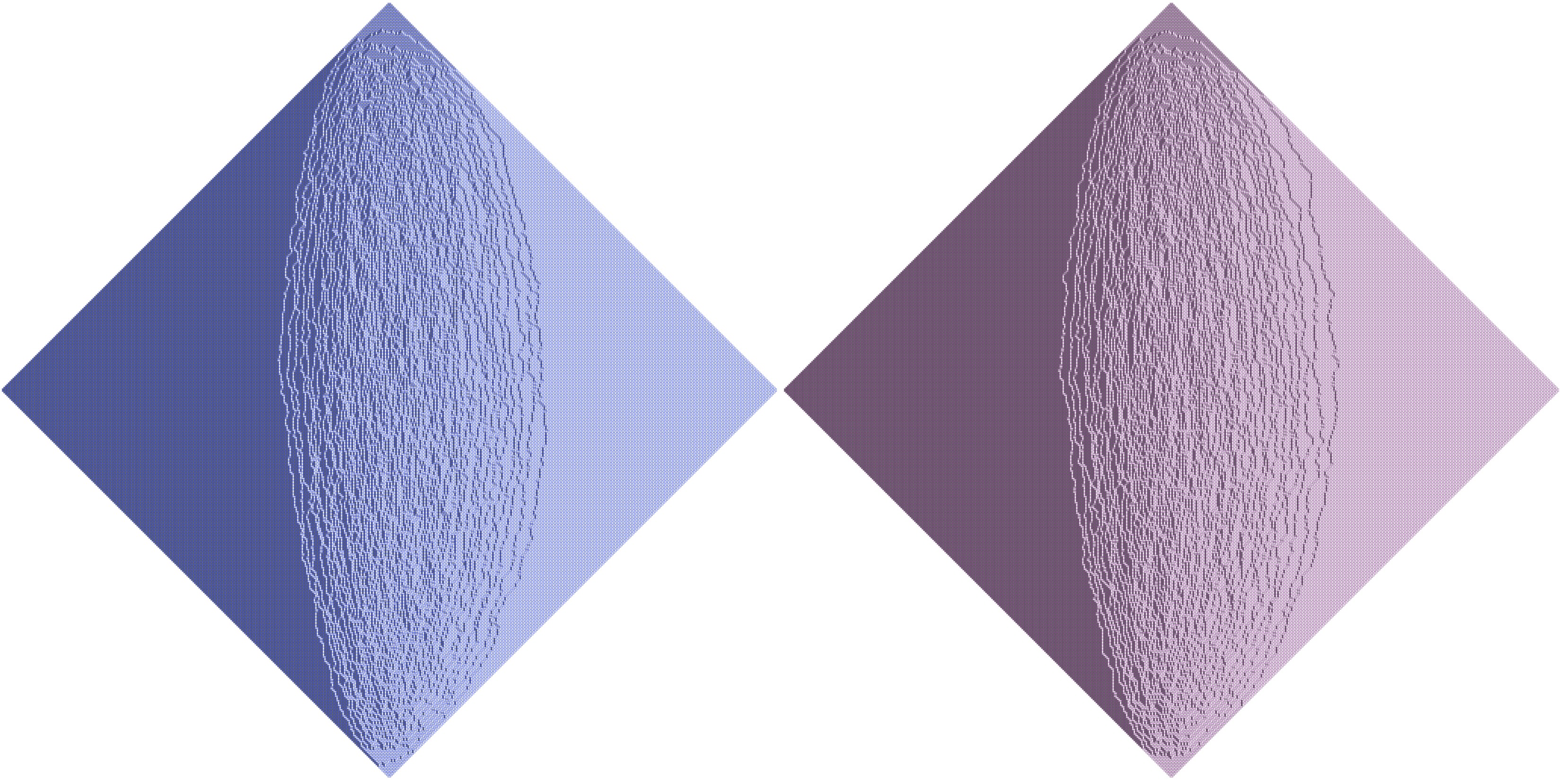}
        
    \caption{A $2$-tiling of rank $N=256$ generated by the shuffling algorithm with $t=0.2$ and $c_i=b_i=0.5$ for all $i=1,\ldots,N$.}
    \label{fig:t0.2b0.5}
\end{figure}

\begin{figure}[!htb]
    \centering
    \includegraphics[width=0.8\textwidth]{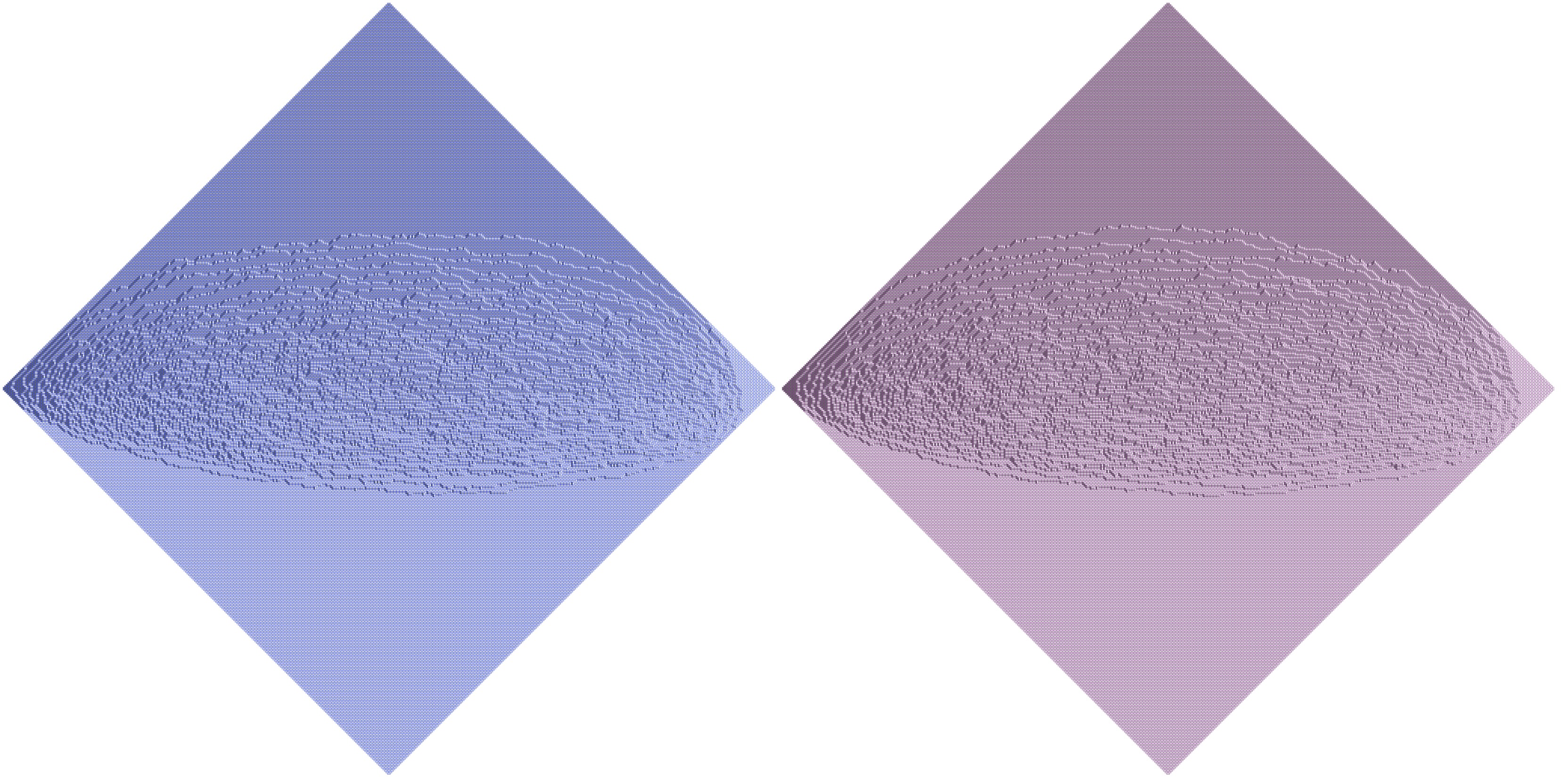}
        
    \caption{A $2$-tiling of rank $N=256$ generated by the shuffling algorithm with $t=5$ and $c_i=b_i=2$ for all $i=1,\ldots,N$.}
    \label{fig:t5b2}
\end{figure}

\begin{figure}[!htb]
    \centering
    \includegraphics[width=0.8\textwidth]{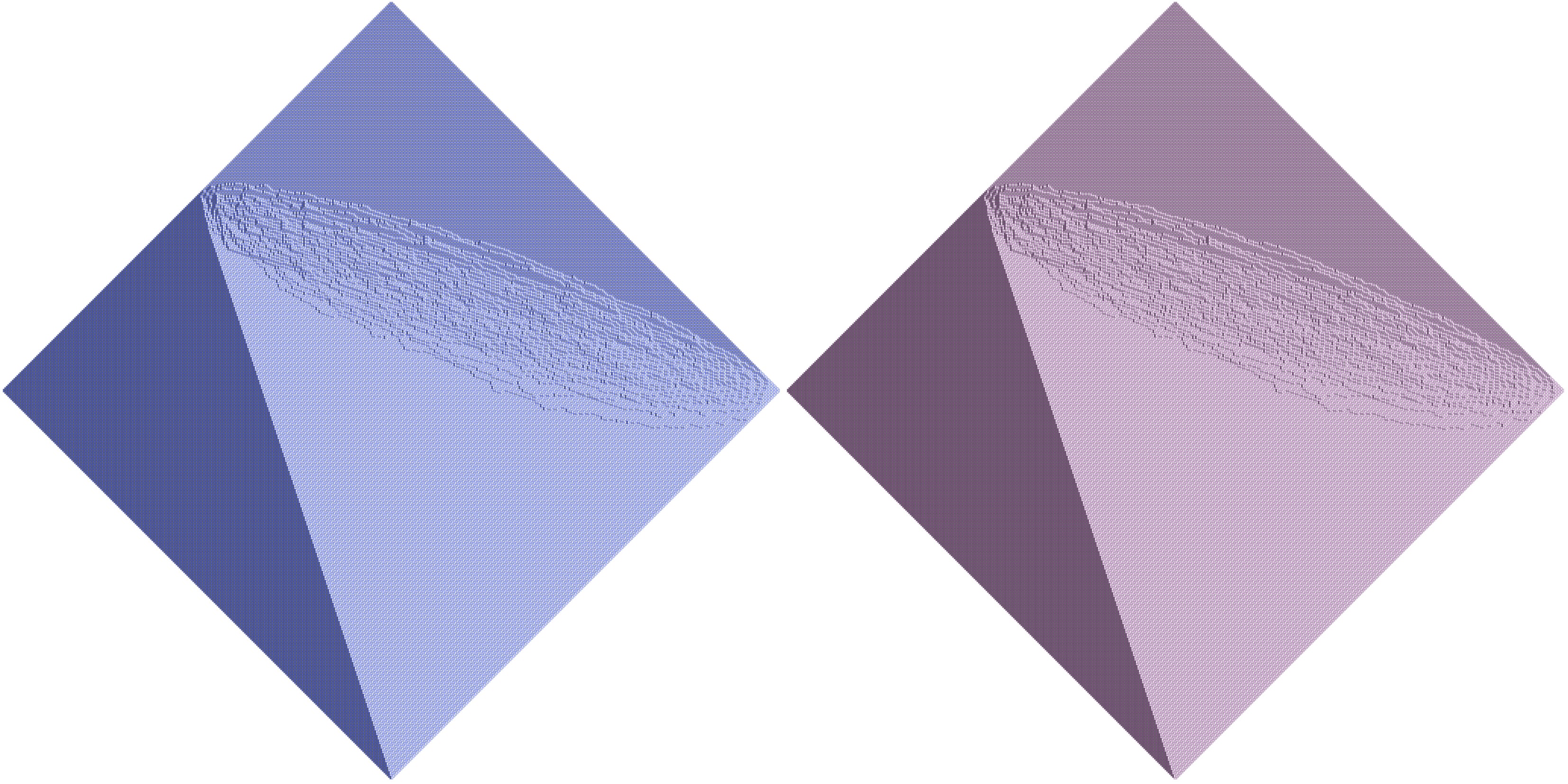}
        
    \caption{A $2$-tiling of rank $N=256$ generated by the shuffling algorithm with $t=0$ and $c_i=b_i=5$ for all $i=1,\ldots,N$.}
    \label{fig:t0b5}
\end{figure}

\begin{figure}[!htb]
    \centering
    \includegraphics[width=0.8\textwidth]{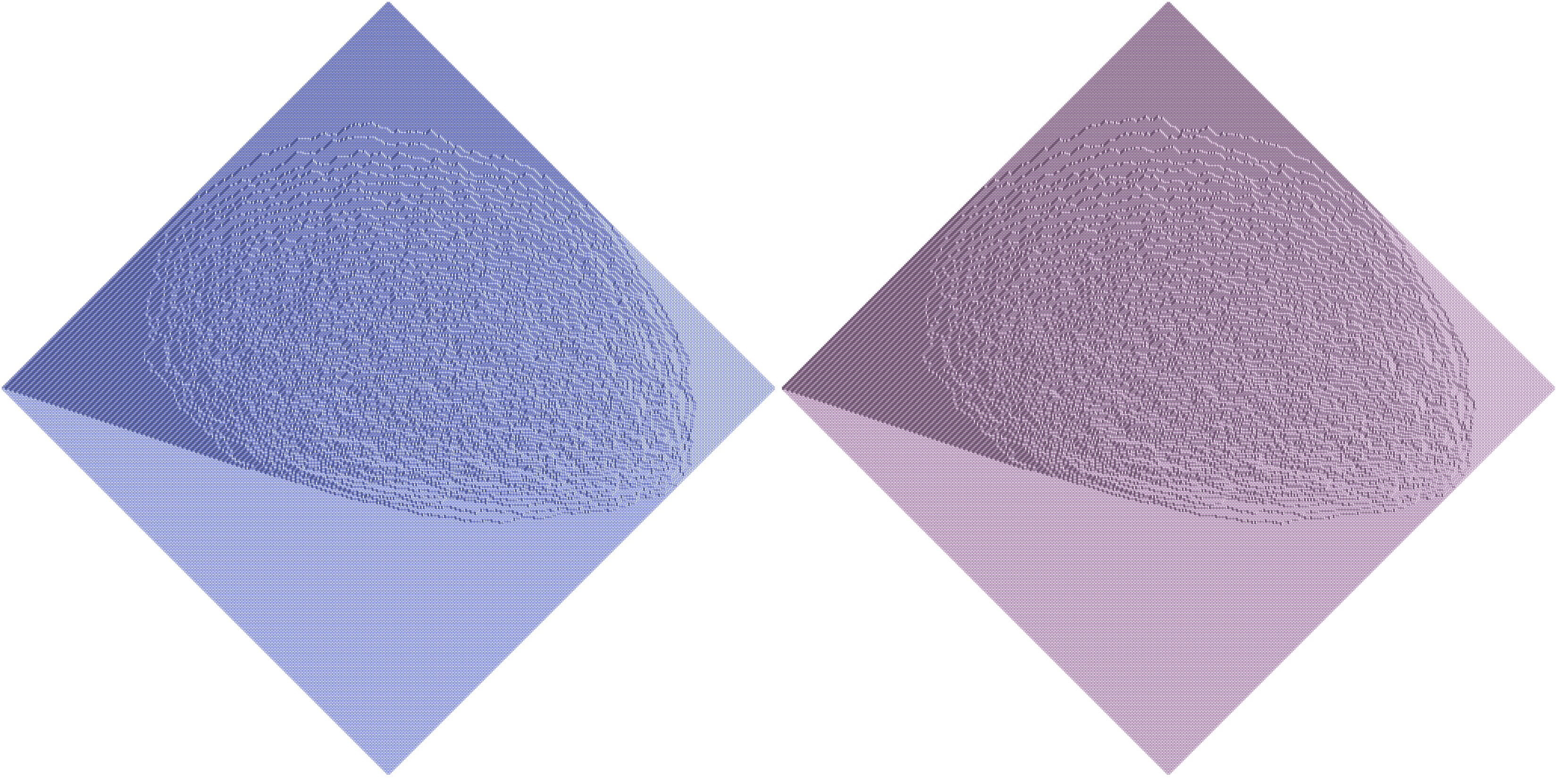}
        
    \caption{A $2$-tiling of rank $N=256$ generated by the shuffling algorithm with $t=10000$ and $c_i=b_i=1$ for all $i=1,\ldots,N$.}
    \label{fig:t10000b1}
\end{figure}

\begin{figure}[!htb]
    \centering
    \includegraphics[width=\textwidth]{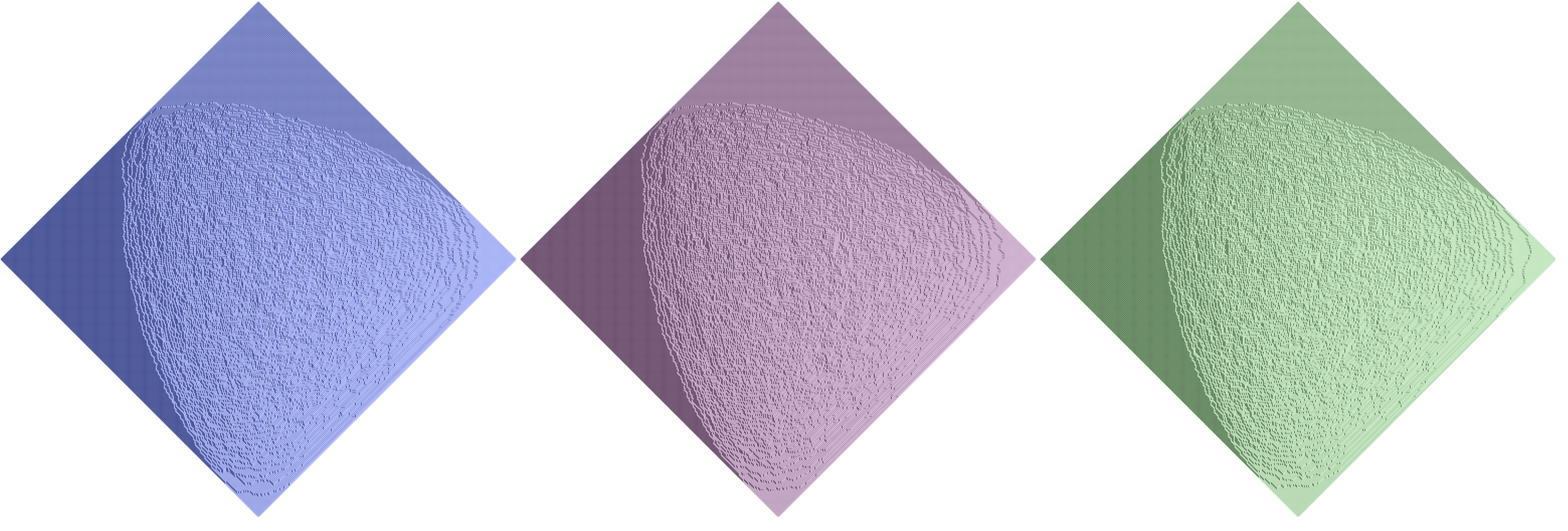}
        
    \caption{A $3$-tiling of rank $N=256$ generated by the shuffling algorithm with $t=0.2$ and $c_i=b_i=2$ for all $i=1,\ldots,N$.}
    \label{fig:t0.2b2c3}
\end{figure}

\end{document}